\documentclass[reqno]{amsart}
\usepackage[new]{old-arrows}
\usepackage{amssymb}
\usepackage{graphicx}
\usepackage{euscript}
\usepackage{cite}
\usepackage{leftindex}
\usepackage[left=2.5cm,right=2.5cm,top=3cm,bottom=3cm]{geometry}
\usepackage{color}
\usepackage{tikz-cd}
\usetikzlibrary{tqft}
\usepackage{mathtools}
\usepackage{hyperref}
\hypersetup{
	colorlinks=true,       			% false: boxed links; true: colored links
	linkcolor=blue,          			% color of internal links
	citecolor=blue,		 		% color of links to bibliography
	filecolor=blue,      				% color of file links
	urlcolor=blue           			% color of external links
}
\usepackage{amsmath}
\usepackage{amsthm}
\usepackage{tikz}
\usepackage{amscd}
\usepackage{latexsym}
\usepackage{epsfig}
\usepackage{graphicx}
\usepackage{caption}
\usepackage{rotating}
\usepackage{mathtools}
\usepackage{xypic}
\usepackage{enumitem}
\usepackage{bigints} 
\usepackage{longtable}
\usepackage{extpfeil}
\usepackage{tensor}
\usepackage{soul}
\usepackage{xcolor}

\setstcolor{red}

\newcommand {\Omit}[1]{}

\newcommand{\g}{\mathfrak{g}}

\newcommand{\h}{\mathfrak{h}}

\newcommand{\p}{\mathfrak{p}}

\newcommand{\greg}{\mathfrak{g}_{\mathrm{reg}}}

\newcommand{\Ad}{\mathrm{Ad}}

\renewcommand{\exp}{\mathrm{exp}}

\newcommand{\MT}{\mathbf{MT}}

\newcommand{\WS}{\mathbf{WS}_1}
\newcommand{\WSQ}{\mathbf{WS}"_{\mkern-11mu 1}}

\newcommand{\Cob}{\mathbf{Cob}}
\newcommand{\htimes}{\times^{\mathrm{h}}}

\newextarrow{\xbigtoto}{{20}{20}{20}{20}}
{\bigRelbar\bigRelbar{\bigtwoarrowsleft\rightarrow\rightarrow}}

\makeatletter
\@namedef{subjclassname@1991}{2020 Mathematics Subject Classification}
\makeatother

\makeatletter
\@namedef{subjclassname@1991}{2020 Mathematics Subject Classification}
\makeatother

\numberwithin{equation}{section}

\newtheorem{theorem}{Theorem}[section]
\newtheorem{proposition}[theorem]{Proposition}
\newtheorem{corollary}[theorem]{Corollary}
\newtheorem{lemma}[theorem]{Lemma}

\newtheorem{conjecture}[theorem]{Conjecture}
\newtheorem{mtheorem}{Main Theorem}

\theoremstyle{definition}
\newtheorem{definition}[theorem]{Definition}

\newtheorem{remark}[theorem]{Remark}

\theoremstyle{plain}
\newtheorem*{theorem*}{Theorem}
\newtheorem*{proposition*}{Proposition}

\newcommand {\IC}{\mathbb{C}}

\newcommand {\reg}{_{\mathrm{reg}}}

\newcommand {\G}{\mathcal G}

\newcommand {\Hom}{\operatorname{Hom}}

\newcommand{\mmm}{\mathsf{m}}

\newcommand{\uuu}{\mathsf{1}}
\newcommand{\fp}[2]{\leftindex_{#1}\times_{#2}\,}

\newcommand{\tto}{\;\substack{\longrightarrow\\[-9pt] \longrightarrow}\;}

\renewcommand{\H}{\mathcal{H}}
\newcommand{\too}{\longrightarrow}

\newcommand{\s}{\subseteq}

\newcommand{\mtoo}{\longmapsto}
\newcommand{\sss}{\mathsf{s}}
\newcommand{\ttt}{\mathsf{t}}
\newcommand{\sll}[1]{\mkern-4mu\mathbin{/\mkern-5mu/}_{\mkern-4mu{#1}}}

\newcommand{\pr}{\mathrm{pr}}
\newcommand{\A}{\mathcal{A}}

\begin{document}
	
	\title[]{Grothendieck--Springer resolutions and TQFTs}
	
	\author[Peter Crooks]{Peter Crooks}
	\author[Maxence Mayrand]{Maxence Mayrand}
	\address[Peter Crooks]{Department of Mathematics and Statistics\\ Utah State University \\ 3900 Old Main Hill \\ Logan, UT 84322, USA}
	\email{peter.crooks@usu.edu}
	\address[Maxence Mayrand]{D\'{e}partement de math\'{e}matiques \\ Universit\'{e} de Sherbrooke \\ 2500 Bd de l’Universit\'{e} \\ Sherbrooke, QC, J1K 2R1, Canada}
	\email{maxence.mayrand@usherbrooke.ca}
	\subjclass{14L30 (primary); 53D20 (secondary)}
	\keywords{Grothendieck--Springer resolution, Moore-Tachikawa conjecture, topological quantum field theory}
	%\keywords{}
	
	\begin{abstract}
		The Moore--Tachikawa conjecture is that each connected complex semisimple group $G$ determines a two-dimensional TQFT in a category of Hamiltonian symplectic varieties. We view the Moore--Tachikawa conjecture as a first step in systematically assigning new TQFTs to purely Lie-theoretic data. At the same time, one should expect these new TQFTs to bear a close relation to those conjectured by Moore and Tachikawa. Our manuscript aims to integrate these two points of view. 
		
		In more detail, let $\g$ be the Lie algebra of $G$. Consider a conjugacy class $\mathcal{C}$ of parabolic subalgebras of $\mathfrak{g}$. This class determines partial Grothendieck--Springer resolutions $\mu_{\mathcal{C}}:\mathfrak{g}_{\mathcal{C}}\longrightarrow\g^*=\mathfrak{g}$ and $\nu_{\mathcal{C}}:G_{\mathcal{C}}\longrightarrow G$. We construct a canonical symplectic groupoid $(T^*G)_{\mathcal{C}}\tto\mathfrak{g}_{\mathcal{C}}$ and quasi-symplectic groupoid $\mathrm{D}(G)_{\mathcal{C}}\tto G_{\mathcal{C}}$. By considering a Kostant slice $\mathrm{Kos}\s\g$ and Steinberg slice $\mathrm{Ste}\s G$, we prove that the pairs  $(((T^*G)_{\mathcal{C}})_{\text{reg}}\tto(\mathfrak{g}_{\mathcal{C}})_{\text{reg}},\mu_{\mathcal{C}}^{-1}(\mathrm{Kos}))$ and $((\mathrm{D}(G)_{\mathcal{C}})_{\text{reg}}\tto(G_{\mathcal{C}})_{\text{reg}},\nu_{\mathcal{C}}^{-1}(\mathrm{Ste}))$ determine new and explicit TQFTs in a $1$-shifted Weinstein symplectic category. We then show that certain symplectic varieties arising from our new TQFTs have canonical Lagrangian relations to the open Moore--Tachikawa varieties.
	\end{abstract}
	
	\maketitle
	\setcounter{tocdepth}{2}%{2}
	\tableofcontents
	
	\vspace{-10pt}
	
	%{\small\tableofcontents} 
	
	%%%%%%%%%%%%%%%%%%%%%%%%%%%%%%
	%%%%%%%%%%%%%%%%%%%%%%%%%%%%%%
	\section{Introduction}
	\subsection{Motivation and context}\label{Subsection: Motivation and context}
	In \cite{moo-tac:11}, Moore and Tachikawa posit the existence of certain two-dimensional topological quantum field theories (TQFTs) valued in a category $\MT$ of holomorphic symplectic varieties. The \textit{Moore--Tachikawa conjecture} has featured in several papers over the last decade \cite{crooks-mayrand,bra-fin-nak:19,balibanu-mayrand,bie:23,ara:18,cal:14,cal:15,cro-may2:24,cro-may:24,dan-kir-mar:24,tac:18,gan-web}, and continues to play a prominent role in geometric representation theory and theoretical physics. In \cite{cro-may2:24}, we prove that any quasi-symplectic groupoid $\G\tto X$ and suitable \textit{global slice} $S\s X$ determine a two-dimensional TQFT $\eta_{\G,S}:\Cob_2\longrightarrow\WS$ behaving analogously to those conjectured by Moore and Tachikawa, where $\WS$ is a completion of a $1$-shifted version of the Weinstein symplectic ``category" \cite{weinstein-symplectic-category}. More precise statements will be forthcoming.
	
	In this manuscript, we view the Moore--Tachikawa conjecture as a mechanism for constructing TQFTs from purely Lie-theoretic data. We are thereby led to consider a class of the TQFTs $\eta_{\G,S}$ described above, and understand interrelationships among these TQFTs. Our class turns out to arise from partial Grothendieck--Springer resolutions in Poisson and quasi-Poisson geometry. One begins by fixing a connected complex semisimple affine algebraic group $G$ with Lie algebra $\g$. The operation $P\mapsto\p\coloneqq \mathrm{Lie}(P)$ induces a bijective correspondence between conjugacy classes of parabolic subgroups $P\s G$ and those of parabolic subalgebras $\p\s\g$. Conjugacy classes $\mathcal{C}$ of the latter type are thereby identified with those of the former type. With this in mind, one has the incidence subvarieties
	$$\mathfrak{g}_{\mathcal{C}}\coloneqq\{(\mathfrak{p},x)\in\mathcal{C}\times\mathfrak{g}:x\in\mathfrak{p}\}\quad\text{and}\quad G_{\mathcal{C}}\coloneqq\{(P,g)\in\mathcal{C}\times G:g\in P\}.$$ One also has the morphisms
	$$\mu_{\mathcal{C}}:\mathfrak{g}_{\mathcal{C}}\longrightarrow\mathfrak{g}^*=\mathfrak{g},\quad (\mathfrak{p},x)\mapsto x\quad\text{and}\quad \nu_{\mathcal{C}}:G_{\mathcal{C}}\longrightarrow G,\quad (P,g)\mapsto g,$$ where the Killing form is used to identify $\g^*$ with $\g$. There is a unique Poisson Hamiltonian $G$-variety structure on $\g_{\mathcal{C}}$ for which $\mu_{\mathcal{C}}$ is a moment map. The partial Grothendieck--Springer resolution (resp. multiplicative partial Grothendieck--Springer resolution) associated to $(G,\mathcal{C})$ is defined to be $(\mathfrak{g}_{\mathcal{C}},\mu_{\mathcal{C}})$ (resp. $(G_{\mathcal{C}},\nu_{\mathcal{C}})$). If $\mathcal{C}=\mathcal{B}$ is the conjugacy class of Borel subalgebras of $\mathfrak{g}$, then $\mathfrak{g}_{\mathcal{C}}=\widetilde{\g}$ is the well-studied Grothendieck--Springer resolution of $\mathfrak{g}$ \cite{eva-lu:07}.
	
	\subsection{Main results}\label{Subsection: Main results}
	Let $G$ be a connected complex semisimple affine algebraic group with Lie algebra $\mathfrak{g}$. One knows that the cotangent groupoid $T^*G\tto\mathfrak{g}^*=\g$ induces the canonical Poisson structure on $\g$. We generalize this fact as follows. 
	\begin{mtheorem}\label{1n9y2fsw}
		Let $\mathcal{C}$ be a conjugacy class of parabolic subalgebras of $\mathfrak{g}$. There is a canonical algebraic symplectic groupoid $(T^*G)_{\mathcal{C}}\tto\g_{\mathcal{C}}$ that induces the Poisson structure on $\g_{\mathcal{C}}$.
	\end{mtheorem}
	
	The term ``canonical" warrants some clarification in this context. We construct $(T^*G)_{\mathcal{C}}\tto\g_{\mathcal{C}}$ for a choice of $\p\in\mathcal{C}$, and then explain the sense in which this construction is choice-independent.
	
	Now write $(\g_{\mathcal{C}})_{\text{reg}}\s\mathfrak{g}_{\mathcal{C}}$ for the open dense subvariety of all points in $\mathfrak{g}_{\mathcal{C}}$ at which the Poisson structure has maximal rank; it coincides with $\g_{\mathcal{C}}$ if and only if $\mathcal{C}=\mathcal{B}$. The pullback of $(T^*G)_{\mathcal{C}}\tto\mathfrak{g}_{\mathcal{C}}$ to 
	$(\g_{\mathcal{C}})_{\text{reg}}$ is an algebraic symplectic groupoid $((T^*G)_{\mathcal{C}})_{\text{reg}}\tto(\g_{\mathcal{C}})_{\text{reg}}$ integrating $(\g_{\mathcal{C}})_{\text{reg}}$. On the other hand, let $\mathrm{Kos}\coloneqq e+\mathfrak{g}_f\s\mathfrak{g}$ be the Kostant slice associated to a principal $\mathfrak{sl}_2$-triple $(e,h,f)\in\mathfrak{g}^{\times 3}$. We show $\mathrm{Kos}_{\mathcal{C}}\coloneqq\mu_{\mathcal{C}}^{-1}(\mathrm{Kos})\s(\g_{\mathcal{C}})_{\text{reg}}$ to be a global slice for $((T^*G)_{\mathcal{C}})_{\text{reg}}\tto(\g_{\mathcal{C}})_{\text{reg}}$, in the sense of \cite[Definition 4.9]{cro-may2:24}. 
	
	This is a point at which \textit{symplectic reduction along a submanifold} \cite{crooks-mayrand} becomes relevant. In more detail, suppose that $m,n\in\mathbb{Z}_{\geq 0}$ satisfy $(m,n)\neq (0,0)$. The symplectic groupoid $(T^*G)_{\mathcal{C}}^{m+n}\times\overline{(T^*G)_{\mathcal{C}}^{m+n}}$ acts on $(T^*G)_{\mathcal{C}}^{m,n}$ in a Hamiltonian fashion. One may reduce by this Hamiltonian action along $$\mathrm{Kos}_{\mathcal{C}}^{m,n}\coloneqq\{(\alpha_1,\ldots,\alpha_{m+n},\beta_1,\ldots,\beta_{m+n})\in\mathfrak{g}_{\mathcal{C}}^{m+n}\times\mathfrak{g}_{\mathcal{C}}^{m+n}:\alpha_{n+1}=\cdots=\alpha_{m+n}=\beta_1=\cdots=\beta_n\in\mathrm{Kos}_{\mathcal{C}}\},$$ yielding $$(T^*G)_{\mathcal{C}}^{m,n}\coloneqq(T^*G)_{\mathcal{C}}^{m+n}\sll{\mathrm{Kos}_{\mathcal{C}}^{m,n}}((T^*G)_{\mathcal{C}}^{m+n}\times\overline{(T^*G)_{\mathcal{C}}^{m+n}}).$$
	We prove that these reduced spaces feature in the following.
	
	\begin{mtheorem}\label{612eq1e1}
		Let $\mathcal{C}$ be a conjugacy class of parabolic subalgebras of $\g$. There is an explicit TQFT $\eta : \Cob_2\longrightarrow\WS$ satisfying $\eta(S^1) = ((T^*G)_{\mathcal{C}})\reg$ and  $\eta(C_{m,n})=[(T^*G)_{\mathcal{C}}^{m,n}]$ for all $(m,n)\neq (0,0)$, where $C_{m,n}$ denotes the genus-0 cobordism from $m$ circles to $n$ circles.
	\end{mtheorem}
	
	We call $(T^*G)_{\mathcal{C}}^{m,n}$ the \textit{open Moore--Tachikawa variety} for $\mathcal{C}$ associated to $(m, n) \ne (0, 0)$. In the case $\mathcal{C}=\{\mathfrak{g}\}$, we omit the subscript and simply write $(T^*G)^{m,n}$. The varieties $(T^*G)^{m,n}$ are often called \textit{open Moore--Tachikawa varieties}, as their affinizations realize a scheme-theoretic version of the Moore--Tachikawa TQFT \cite{gin-kaz:23,cro-may2:24}.
	
	It is natural to seek relationships between the varieties $(T^*G)_{\mathcal{C}}^{m,n}$ for fixed $(m,n)\neq (0,0)$, as $\mathcal{C}$ ranges over the conjugacy classes of parabolic subalgebras of $\mathfrak{g}$. We achieve this by appealing to Weinstein's notion of a \textit{Lagrangian relation} from a symplectic variety $X$ to another such variety $Y$ \cite{Bates,weinstein-symplectic-category,WeinsteinCategories}; this is a set-theoretic relation from $X$ to $Y$ with the property of being an immersed Lagrangian in $X\times\overline{Y}$, where $\overline{Y}$ results from negating the symplectic form on $Y$. While a composition of relations is a relation set-theoretically, a composition of Lagrangian relations need not be a Lagrangian relation. A sufficient condition for the latter composition to be a Lagrangian submanifold is for the Lagrangian relations to be \textit{strongly composable} \cite{WeinsteinCategories}. If one is prepared to accept an immersed Lagrangian submanifold of $X\times\overline{Y}$, then it suffices for the Lagrangian relations to be \textit{composable} \cite{WehrheimWoodward,WeinsteinCategories}. These last few sentences give context for the following main result.
	
	\begin{mtheorem}\label{Theorem: Main Theorem 3}
		Fix $(m,n)\in(\mathbb{Z}_{\geq 0})^2$ with $(m,n)\neq (0,0)$, and let $\mathcal{C}$ be a conjugacy class of parabolic subalgebras of $\mathfrak{g}$. There is an explicit Lagrangian relation from $(T^*G)_{\mathcal{C}}^{m,n}$ to $(T^*G)^{m,n}$.
	\end{mtheorem}
	
	A different approach to building TQFTs from Lie-theoretic data is given in \cite{mai-may:25}.
	
	\subsection{Multiplicative counterparts of main results} We also obtain multiplicative analogues of Main Theorem \ref{1n9y2fsw} and Main Theorem \ref{612eq1e1} by replacing $\g$ and $T^*G$ with $G$ and $\mathrm{D}(G)=G\times G$, respectively. The former (resp. latter) is a quasi-Poisson $G$-variety (resp. quasi-Hamiltonian $G\times G$-variety). On the other hand, the components of the moment map $\mathrm{D}(G)\longrightarrow G\times G$ constitute the source and target of an algebraic Lie groupoid $\mathrm{D}(G)\tto G$. This groupoid is quasi-symplectic with respect to a quasi-Hamiltonian $2$-form on $\mathrm{D}(G)$ and Cartan $3$-form on $G$. Using multiplicative versions of the constructions above, we obtain algebraic quasi-symplectic groupoids $\mathrm{D}(G)_{\mathcal{C}} \tto G_{\mathcal{C}}$.
	Furthermore, replacing the Kostant slice $\mathrm{Kos}\s\g$ with the \textit{Steinberg slice} $\mathrm{Ste}\s G$ yields an analogue $\mathrm{D}(G)_{\mathcal{C}}^{m,n}$ of $(T^*G)_{\mathcal{C}}^{m,n}$. We have an induced TQFT $\delta : \Cob_2\longrightarrow\WS$ satisfying $\delta(S^1)=(\mathrm{D}(G)_{\mathcal{C}})_{\text{reg}}$ and $\delta(C_{m,n})=[(\mathrm{D}(G))_{\mathcal{C}}^{m,n}]$ for all $(m, n) \ne (0, 0)$. Setting $\mathcal{C} = \{G\}$ recovers the multiplicative open Moore--Tachikawa varieties \cite{balibanu-mayrand}.

	\subsection{Organization} Each section begins with a summary of its contents. Section \ref{Section: Dirac} provides some of the crucial background, conventions, and results in Hamiltonian and quasi-Hamiltonian geometry. In Section \ref{Section: New TQFTs}, we develop the pertinent parts of \cite{cro-may2:24} on TQFTs in a $1$-shifted Weinstein symplectic category. We establish the Lie-theoretic underpinnings of partial Grothendieck--Springer resolutions in Section \ref{Section: Some Lie-theoretic considerations in Poisson geometry}. This allows us to discuss the Poisson geometry of $\g_{\mathcal{C}}$ in Section \ref{Section: Partial Grothendieck--Springer resolutions}. In Section \ref{Section: TQFTs from}, we associate a TQFT to each partial Grothendieck--Springer resolution with a fixed Kostant slice. The above-advertised Lagrangian relations are derived in Section \ref{Section: Main}.
	
	Our attention turns to multiplicative / quasi-Poisson counterparts in Section \ref{Section: Some Lie-theoretic considerations in quasi-Poisson geometry}, where we recall the quasi-Hamiltonian $G$-variety structure on $G\times_{U(P)}P$. In Section \ref{Section: Multiplicative TQFTs}, we associate a TQFT to each multiplicative partial Grothendieck--Springer resolution $G_{\mathcal{C}}\longrightarrow G$ with a fixed Steinberg slice.
	
	\subsection{Acknowledgements} The authors thank Philip Boalch, George Lusztig, and Eckhard Meinrenken for explaining parts of \cite{boa-yam:15,boa:07,boa:11,boa:thesis}, \cite{lus-spa:79}, and \cite{ale-mei:24}, respectively. We also thank the referee for detailed, thoughtful, and constructive comments that have improved the manuscript. P.C. acknowledges partial support from the National Science Foundation Grant DMS-2454103 and Simons Foundation Grant MP-TSM-00002292. M.M was partially supported by the NSERC Discovery Grant RGPIN-2023-04587.
	
	\section{Quasi-symplectic groupoids}\label{Section: Dirac}
	We now develop aspects of the theory of quasi-symplectic groupoids. This is intended to contextualize the TQFT-theoretic results of \cite{cro-may2:24}, and lay a broad theoretical foundation for the entire manuscript. In Subsection \ref{Subsection: Fundamental conventions}, we set a few fundamental conventions. Subsection \ref{Subsection: Conventions} then establishes some of our conventions for Lie groupoids. We discuss aspects of quasi-symplectic groupoids in Subsection \ref{Subsection: Quasi-symplectic groupoids}. This allows us to integrate quotients of Hamiltonian $G$-spaces in Subsection \ref{Subsection: Previous}.
	
	In Subsection \ref{Subsection: reduction}, we review the theory of \textit{symplectic reduction along a submanifold} \cite{crooks-mayrand}. Certain Lagrangian relations arise in this context, as detailed in Subsection \ref{Subsection: Isotropic}. Attention subsequently shifts to quasi-Poisson manifolds in Subsection \ref{Subsection: Quasi-Poisson manifolds}. This leads to Subsection \ref{Subsection: Integrating}, where we construct quasi-symplectic groupoids over quotients of quasi-Hamiltonian manifolds.
	
	\subsection{Fundamental conventions}\label{Subsection: Fundamental conventions} In this manuscript, we work exclusively over $\mathbb{C}$; it is the base field underlying all objects for which a base field is presupposed (e.g. vector spaces, manifolds, algebraic varieties). The reader should interpret all differential-geometric notions as being in the holomorphic category.
	\subsection{Conventions for Lie groupoids}\label{Subsection: Conventions}
	Let $\G\tto X$ be a groupoid object in the category of complex manifolds, with source, target, and identity bisection maps denoted $\sss:\G\longrightarrow X$, $\ttt:\G\longrightarrow X$, and $\uuu:X\longrightarrow\G$, respectively. Given a submanifold $S\s X$, we write $\G\big\vert_S$ for the restriction $\sss^{-1}(S)\cap\ttt^{-1}(S)\tto S$. Let us adopt the simplified notation $\G_x\coloneqq\G\vert_{\{x\}}$ for the \textit{isotropy group} of $x\in X$. We also observe the following convention for groupoid multiplication: given $g,h\in\G$, the product $gh\in\G$ is defined if and only if $\sss(g)=\ttt(h)$. Multiplication thereby takes the form of a map $\mmm:\G\fp{\sss}{\ttt}\G\longrightarrow\G$.
	
	One calls $\G\tto X$ a \textit{Lie groupoid} if $\sss$ and $\ttt$ are submersions. Assume this to be the case, and let $M$ be a manifold. An \textit{action} of $\G$ on $M$ consists of holomorphic maps $\mu:M\longrightarrow X$ and $$\mathcal{A}:\G\fp{\sss}{\mu}M\longrightarrow M,\quad (g,p)\mapsto g\cdot p$$ that satisfy the natural generalizations of the left group action axioms. The associated \textit{action groupoid} is $\G\fp{\sss}{\mu}M\tto M$, with source $(g,p)\mapsto p$, target $(g,p)\mapsto g\cdot p$, identity bisection $p\mapsto (\uuu(p),p)$, and multiplication $(g,p)(h,q)=(gh,q)$. We denote the action groupoid by $\G\ltimes M$.
	
	\subsection{Quasi-symplectic groupoids}\label{Subsection: Quasi-symplectic groupoids}
	A \textit{quasi-symplectic groupoid} is a triple $(\G,\omega,\phi)$ of a Lie groupoid $\G\tto X$, multiplicative $2$-form $\omega\in\Omega^2(\G)$, and $3$-form $\phi\in\Omega^3(X)$ satisfying $\mathrm{d}\omega=\sss^*\phi-\ttt^*\phi$, $\dim\G=2\dim X$, and $\ker\omega_x\cap\ker\mathrm{d}\sss_x\cap\ker\mathrm{d}\ttt_x=\{0\}$ for all $x\in X$, where one uses the identity bisection to regard $x$ as belonging to $\G$. We let $\overline{\G}$ denote the ``opposite" $(\G,-\omega,-\phi)$ of a quasi-symplectic groupoid $(\G,\omega,\phi)$.
	
	Consider a quasi-symplectic groupoid $(\G \tto X, \omega, \phi)$ and manifold $M$ equipped with a 2-form $\gamma\in\Omega^2(M)$. Suppose that $\G$ acts on $M$ along a holomorphic map $\mu : M \too X$. This action is called \textit{Hamiltonian} \cite{xu} if it satisfies the following properties:
	\begin{itemize}
		\item
		$\mu^*\phi = -\mathrm{d}\gamma$;
		\item
		the graph of the action, i.e.\ $\{(g, m, g \cdot m) : (g, m) \in \G\fp{\sss}{\mu}M\}$, is isotropic in $\G \times M \times M$ with respect to $(\omega, \gamma, -\gamma)$;
		\item
		$\ker \mathrm{d}\mu \cap \ker \gamma = 0$.
	\end{itemize}
	
	A quasi-symplectic groupoid $(\G\tto X,\omega,\phi)$ is called \textit{symplectic} if $\phi=0$ and $\omega$ is non-degenerate. It turns out that a Lie groupoid $\G\tto X$ and $2$-form $\omega\in\Omega^2(\G)$ constitute a symplectic groupoid if and only if the graph of groupoid multiplication is isotropic in $\G\times\G\times\overline{\G}$. The Lie algebroid of $\G$ is then $T^*X\longrightarrow X$, along with an anchor map that defines a Poisson structure on $X$. An \textit{integration} of a Poisson manifold $X$ is a symplectic groupoid $\G\tto X$ that induces the given Poisson structure on $X$. 
	
	We now discuss a paradigm example of a symplectic groupoid. To this end, let $G$ be a Lie group with Lie algebra $\g$. Use the left trivialization to identify $T^*G$ with $G\times\g^*$. While the former is symplectic, the latter may be viewed as the action groupoid of the coadjoint action. These structures are compatible in the sense that $T^*G\tto\g^*$ is a symplectic groupoid integrating the canonical Poisson structure on $\g^*$. The source $\sss:T^*G\longrightarrow\g^*$ and target $\ttt:T^*G\longrightarrow\g^*$ are given by $\sss(g,\xi)=\mathrm{Ad}_g^*(\xi)$ and $\ttt(g,\xi)=\xi$, respectively, where $\mathrm{Ad}^*:G\longrightarrow\operatorname{GL}(\g^*)$ is the coadjoint representation of $G$. One calls $T^*G\tto\g^*$ the \textit{cotangent groupoid} of $G$. Symplectic Hamiltonian $G$-spaces are equivalent to symplectic manifolds equipped with Hamiltonian actions of $T^*G\tto\g^*$ \cite[Example 3.9]{mikami-weinstein}. 
	
	\subsection{Integrating quotients of symplectic manifolds}\label{Subsection: Previous}
	
	Let $(M, \omega)$ be a symplectic manifold. Suppose that a Lie group $G$ acts on $M$ in a Hamiltonian fashion with moment map $\mu : M \longrightarrow \g^*$. Equip the total space of the pair groupoid $\mathrm{Pair}(M)=M \times M\tto M$ with the symplectic structure on $M\times\overline{M}$, where $\overline{M}$ results from negating the symplectic form on $M$; this yields a symplectic groupoid integrating $M$. The diagonal action of $G$ on $M \times M$ is then Hamiltonian with moment map
	\[
	M \times M \longrightarrow \g^*, \quad (p, q) \mapsto \mu(p) - \mu(q).
	\]
	If the $G$-action on $M$ is free and proper, then $$\mathrm{Pair}(M)\sll{0}G=(M \times_{\g^*} M) / G$$ is a Lie groupoid over $M/G$. We also know $\mathrm{Pair}(M)\sll{0}G$ and $M/G$ to be symplectic and Poisson, respectively. These considerations give context for the following result.
	
	\begin{proposition}\label{Proposition: Integration}
		If the action of $G$ on $M$ is free and proper, then $\mathrm{Pair}(M)\sll{0}G\tto M/G$ is a symplectic groupoid integrating the Poisson manifold $M/G$.
	\end{proposition}
	
	\begin{proof}
		Note that the graph of multiplication $\Gamma \s\mathrm{Pair}(M) \times \mathrm{Pair}(M) \times \overline{\mathrm{Pair}(M)}$ descends to a Lagrangian submanifold of $(\mathrm{Pair}(M)\sll{0}G) \times (\mathrm{Pair}(M)\sll{0}G) \times(\overline{\mathrm{Pair}(M)\sll{0}G})$, so that $\mathrm{Pair}(M)\sll{0}G$ is a symplectic groupoid over $M/G$.
		The symplectic form on $\mathrm{Pair}(M)\sll{0}G$ provides an isomorphism from its Lie algebroid to $T^*(M/G)$. It follows that the anchor map $T^*(M/G) \longrightarrow T(M/G)$ is the Poisson structure on $M/G$.
	\end{proof}
	
	\begin{remark}
		It seems likely that this result could be deduced from techniques in \cite{bra-fer:14}.
	\end{remark}
	
	\subsection{Symplectic reduction along a submanifold}\label{Subsection: reduction}
	Let $\G\tto X$ be a symplectic groupoid. Write $\sigma\in H^0(X,\wedge^2(TX))$ for the Poisson bivector field on $X$, and $\sigma^{\vee}:T^*X\longrightarrow TX$ for its contraction with cotangent vectors. One calls a submanifold $S\s X$ \textit{pre-Poisson} if $$L_S\coloneqq(\sigma^{\vee})^{-1}(TS)\cap\mathrm{ann}_{T^*X}(TS)\longrightarrow S$$ has constant fiber dimension, in which case $L_S$ is Lie subalgebroid of $(T^*X,\sigma^{\vee})$ \cite{cattaneo-zambon-2007,cattaneo-zambon-2009}. Suppose that $\H\tto S$ is a Lie subgroupoid of $\G$ over a pre-Poisson submanifold $S\s X$. We call $\H$ a \textit{stabilizer subgroupoid} if $\H$ is isotropic in $\G$ and has $L_S$ as its Lie algebroid \cite{crooks-mayrand}. Noteworthy special cases of these definitions arise for $\G\tto X$ the (left-trivialized) cotangent groupoid $T^*G\tto\g^*$ of a Lie group $G$. The submanifold $\{\xi\}\s\g^*$ is pre-Poisson for all $\xi\in\g^*$, and admits $G_{\xi}\longrightarrow\{\xi\}$ as a stabilizer subgroupoid. This observation foreshadows a generalization of Marsden--Weinstein reduction \cite{marsden-weinstein}; the following are some details.
	
	Suppose that a symplectic groupoid $\G\tto X$ acts on a symplectic manifold $(M,\omega)$ in a Hamiltonian fashion with moment map $\mu:M\longrightarrow X$. Let $\H\tto S$ be a stabilizer subgroupoid of $\G$ over a pre-Poisson submanifold $S\s X$. Note that $\H$ acts on $\mu^{-1}(S)$. Consider the inclusion $j:\mu^{-1}(S)\longrightarrow M$ and quotient $\pi:\mu^{-1}(S)\longrightarrow\mu^{-1}(S)/\mathcal{H}$. If $\H$ acts freely and properly on $\mu^{-1}(S)$, then $\mu$ is transverse to $S$, and the manifold $\mu^{-1}(S)/\mathcal{H}$ carries a unique symplectic form $\overline{\omega}$ satisfying $\pi^*\overline{\omega}=j^*\omega$. The symplectic manifold
	\[
	M \sll{S, \H} \G \coloneqq (\mu^{-1}(S) / \H,\overline{\omega})
	\]
	is called the \emph{symplectic reduction of $M$ by $\G$ along $S$ with respect to $\H$} \cite{crooks-mayrand}.
	We use the simpler notation $M \sll{S} \G$ if $\H$ is source-connected; all source-connected integrations of $L_S$ yield the same symplectic manifold.
	
	The preceding construction is more general than described above; it also holds for smooth manifolds, complex analytic spaces, complex algebraic varieties, and affine schemes. Another feature of this construction is that it generalizes many of the approaches to symplectic reduction that have developed over the last 50 years. We refer the interested reader to \cite{crooks-mayrand} and \cite{cro-may:24} for precise details. A more leisurely introduction can be found in \cite{cro+:25}.
	
	\subsection{Some Lagrangian relations}\label{Subsection: Isotropic}
	Let $X$ be a Poisson manifold. Recall that a submanifold $S\s X$ is called a \textit{Poisson transversal} if it intersects each symplectic leaf of $X$ transversely and in a symplectic submanifold of that leaf. The Poisson structure on $X$ then induces a Poisson structure on $S$ \cite[Lemma 3]{fre-mar:17}. One also knows that the inverse image of a Poisson transversal under a Poisson map is a Poisson transversal \cite[Lemma 7]{fre-mar:17}. It is also clear that the Poisson transversals in a symplectic manifold are the symplectic submanifolds of that manifold. These considerations feature in the following result.
	
	\begin{lemma}\label{Lemma: Lagrangian correspondence}
		Let a symplectic groupoid $\G \tto X$ act on a symplectic manifold $(M,\omega)$ in a Hamiltonian fashion with moment map $\mu : M \longrightarrow X$. Suppose that $\H\tto S$ is a stabilizer subgroupoid of $\G$ over a pre-Poisson submanifold $S\s X$. Assume that $\H$ acts freely and properly on $\mu^{-1}(S)$.
		\begin{itemize}
			\item[\textup{(i)}] The submanifold $\{(m,[m]):m\in\mu^{-1}(S)\}\s M\times\overline{M\sll{S,\mathcal{H}}\G}$ is isotropic.
			\item[\textup{(ii)}] The submanifold in \textup{(i)} is Lagrangian if and only if $S$ is coisotropic in $X$.
			\item[\textup{(iii)}] If $P\s X$ is a Poisson transversal containing $S$ as a coisotropic submanifold, then the submanifold in \textup{(i)} is Lagrangian in $\mu^{-1}(P)\times \overline{M\sll{S,\mathcal{H}}\G}$.
		\end{itemize}
	\end{lemma}
	
	\begin{proof}
		Write $Y$ for the submanifold in (i) and $\pi:\mu^{-1}(S)\longrightarrow M\sll{S,\mathcal{H}}\G$ for the quotient map. Given $(m,[m])\in Y$, one finds that $$T_{(m,[m])}Y=\{(v,\mathrm{d}\pi_m(v)):v\in(\mathrm{d}\mu_m)^{-1}(T_{\mu(m)}S)\}\s T_m M\oplus T_{[m]}(M\sll{S,\mathcal{H}}\G).$$ Let us therefore fix two vectors $(v_1,\mathrm{d}\pi_m(v_1))$ and $(v_2,\mathrm{d}\pi_m(v_2))$ in $T_{(m,[m])}Y$. Write $\overline{\omega}$ for the symplectic forms on $M\sll{S,\mathcal{H}}\G$ and $j:\mu^{-1}(S)\longrightarrow M$ for the inclusion. Since $\pi^*\overline{\omega}=j^*\omega$, we have $$\overline{\omega}_{[m]}(\mathrm{d}\pi_m(v_1),\mathrm{d}\pi_m(v_2))=\omega_m(v_1,v_2).$$ The calculation
		$$\omega_m(v_1,v_2)-\overline{\omega}_{[m]}(\mathrm{d}\pi_m(v_1),\mathrm{d}\pi_m(v_2))=0$$ then shows $Y$ to be isotropic in $M\times\overline{M\sll{S,\mathcal{H}}\G}$, verifying (i).
		
		In light of (i), $Y$ is Lagrangian in $M\times\overline{M\sll{S,\mathcal{H}}\G}$ if and only if $\dim Y=\frac{1}{2}(\dim(M\times\overline{M\sll{S,\mathcal{H}}\G})$. We also have the identities
		$$\dim Y=\dim\mu^{-1}(S)=\dim M-\dim X+\dim S\quad\text{and}\quad\dim(M\sll{S,\mathcal{H}}\G)=\dim\mu^{-1}(S)-\mathrm{rank}\hspace{2pt}L_S.$$ It is now straightforward to verify that $\dim Y=\frac{1}{2}(\dim(M\times\overline{M\sll{S,\mathcal{H}}\G})$ if and only if $\mathrm{rank}\hspace{2pt}L_S=\dim X-\dim S$. The latter equation holds if and only if $L_S=\mathrm{ann}_{T^*X}(TS)$, as $\mathrm{rank}\hspace{2pt}(\mathrm{ann}_{T^*X}(TS))=\dim X-\dim S$. Letting $\sigma^{\vee}:T^*X\longrightarrow TX$ denote the contraction of the Poisson structure $\sigma$ on $X$, the condition $L_S=\mathrm{ann}_{T^*X}(TS)$ becomes $\mathrm{ann}_{T^*X}(TS)\s(\sigma^{\vee})^{-1}(TS)$. This new condition is precisely the definition of $S$ being coisotropic in $X$, completing the proof of (ii).
		
		Part (iii) follows immediately from (ii).
	\end{proof}
	
	There is a another class of Lagrangian submanifolds arising in the context of symplectic reduction along a submanifold. To this end, let a symplectic groupoid $\G \tto X$ act in a Hamiltonian fashion on a symplectic manifold $M$ with moment map $\mu : M \longrightarrow X$.
	The definition of a Hamiltonian action implies that
	\[
	\mathcal{L}_\mu \coloneqq \{(g \cdot p, p, g) \in M \times M \times \G : (g, p) \in \G \ltimes M\}
	\]
	is a Lagrangian subgroupoid of $\mathrm{Pair}(M) \times \overline{\G}$. More generally, consider a symplectic groupoid $\G \tto X$, a Lie group $G$, and a Hamiltonian $\G\times T^*G$-space $M$ with moment map $(\mu,\nu):M\longrightarrow X\times\g^*$. The equivalence between Hamiltonian $G$-spaces and Hamiltonian $T^*G$-spaces renders $M$ a Hamiltonian $G$-space. Suppose that $G$ acts freely and properly on $M$. Proposition \ref{Proposition: Integration} yields a symplectic groupoid $$(\tilde{\sss}, \tilde{\ttt}) :\mathrm{Pair}(M)\sll{0}G=(M\times_{\nu}\overline{M})\tto M/G.$$ It is also clear that $\mu$ descends to a Poisson map $\overline{\mu}:M/G\longrightarrow X$.
	We conclude that
	\[
	\mathcal{L}_{\overline{\mu}} \coloneqq \{([g \cdot p, p], g) \in (M \times_\nu \overline{M})/G \times \overline{\G} : (g, p) \in \G \ltimes M\}
	\]
	is a Lagrangian relation from $\mathrm{Pair}(M)\sll{0}G$ to $\G$.
	The following result will be useful in Section \ref{Section: Main}.
	
	\begin{lemma}\label{jmzayi83}
		Retain the objects and notation of the previous paragraph. Let $S \s X$ be a Poisson transversal, and set $\tilde{S} \coloneqq \overline{\mu}^{-1}(S) \s M/G$.
		Then $\tilde{\ttt}^{-1}(\tilde{S}) \times \ttt^{-1}(S)$ is a symplectic submanifold of $(M \times_\nu \overline{M})/G \times \overline{\G}$, and its intersection with $\mathcal{L}_{\overline{\mu}}$ is a Lagrangian submanifold.
	\end{lemma}
	
	\begin{proof}
		As $\overline{\mu}$ is a Poisson map, $\tilde{S}$ is a Poisson transversal in $M/G$. This implies the first claim.
		For the second claim, it remains to check that $\mathcal{L}_{\overline{\mu}} \cap (\tilde{\ttt}^{-1}(\tilde{S}) \times \ttt^{-1}(S))$ is smooth and has the correct dimension.
		
		Let us set
		\[
		\G \ltimes_S M/G \coloneqq \{(g, [p]) \in \G \times M/G : \ttt(g) = \mu(p) \in S\} = \ttt^{-1}(S) \times_S \overline{\mu}^{-1}(S).
		\]
		Note that we have an embedding
		\[
		\G \ltimes_S M/G \too (M \times_\nu \overline{M})/G \times \overline{\G}, \quad
		(g, [p]) \mtoo ([g \cdot p, p], g),
		\]
		whose image is $\mathcal{L}_{\overline{\mu}} \cap (\tilde{\ttt}^{-1}(\tilde{S}) \times \ttt^{-1}(S))$.
		At the same time,
		\begin{align*}
			\dim(\G \ltimes_S M/G) &= \dim \ttt^{-1}(S) + \dim \overline{\mu}^{-1}(S) - \dim S \\
			&= \big(\dim \G + \dim S - \dim X\big) + \big(\dim(M/G) + \dim S - \dim X\big) - \dim S \\
			&= \dim S + \dim(M/G).
		\end{align*}
		On the other hand,
		\begin{align*}
			\dim(\tilde{\ttt}^{-1}(\tilde{S}) \times \ttt^{-1}(S)) &= \dim(\tilde{S}) + \dim(M/G) + \dim(S) + \dim(X) \\
			&= \big(\dim(S) + \dim(M/G) - \dim(X)\big) + \dim(M/G) + \dim(S) + \dim(X) \\
			&= 2 \dim(S) + 2 \dim(M/G) \\
			&= 2 \dim(\G \ltimes_S M/G).\qedhere
		\end{align*}
	\end{proof}

	\subsection{Quasi-Hamiltonian and quasi-Poisson manifolds}\label{Subsection: Quasi-Poisson manifolds}
	Following Boalch \cite{boa:07} and a number of his subsequent manuscripts \cite{boa-yam:15}, we now outline the algebraic and holomorphic approaches to quasi-Hamiltonian and quasi-Poisson geometry in this manuscript. Let $G$ be a Lie group with Lie algebra $\g$, exponential map $\exp:\g\longrightarrow G$, and fixed $G$-invariant, non-degenerate, symmetric bilinear form $(\cdot,\cdot):\g\otimes\g\longrightarrow\mathbb{C}$. The \textit{Cartan $3$-form} on $G$ is the unique $G$-invariant element $\eta_G\in\Omega^3(G)$ satisfying $(\eta_G)_e(x,y,z)=\frac{1}{12}(x,[y,z])$ for all $x,y,z\in\g$. Note that this expression defines an element of $\wedge^3(\g^*)$. Using $(\cdot,\cdot)$ to identify $\mathfrak{g}$ and $\g^*$, this is an element $\chi_G\in\wedge^3\g$. One also has $\theta^L,\theta^R\in\Omega^1(G,\g)$, the left and right-invariant Maurer--Cartan forms on $G$, respectively.
	
	Suppose that $G$ acts holomorphically on a manifold $M$. Each $\xi\in\g$ thereby determines a generating vector field $\xi_M\in H^0(M,TM)$, i.e.
	$$(\xi_M)(p)\coloneqq\frac{d}{dt}\bigg\vert_{t=0}\exp(-t\xi)\cdot p$$ for all $p\in M$. The Lie algebra morphism $$\g\longrightarrow H^0(M,TM),\quad \xi\mapsto\xi_M$$ extends uniquely to a graded algebra morphism $\wedge^{\bullet}\g\longrightarrow H^0(M,\wedge^{\bullet}(TM))$. Write $(\chi_G)_M\in H^0(M,\wedge^3(TM))$ for the image of $\chi_G\in\wedge^3\g$ under the latter morphism. One calls $M$ a \textit{quasi-Poisson $G$-manifold} \cite{ale-kos-mei:02} if it comes equipped with a $G$-invariant bivector field $\sigma\in H^0(M,\wedge^2(TM))$ satisfying $[\sigma,\sigma]=(\chi_G)_M$, where $[\sigma,\sigma]\in H^0(M,\wedge^3TM)$ denotes the Schouten bracket of $\sigma$ with itself. Let us write $\overline{M}$ for the quasi-Poisson $G$-manifold obtained by negating the quasi-Poisson structure on a quasi-Poisson $G$-manifold $M$.
	
	Let $M$ be a quasi-Poisson $G$-manifold. A holomorphic map $\mu:M\longrightarrow G$ is called a \textit{moment map} if it is equivariant with respect to the conjugation action of $G$ on itself, and satisfies a quasi-Poisson counterpart of Hamilton's equations; see \cite[Definition 2.2]{ale-kos-mei:02}. One uses the term \textit{quasi-Poisson Hamiltonian $G$-space} for a quasi-Poisson $G$-manifold with a prescribed moment map. Let $(M_1,\sigma_1,\mu_1)$ and $(M_2,\sigma_2,\mu_2)$ be quasi-Poisson Hamiltonian $G$-spaces. By a \textit{map of quasi-Poisson Hamiltonian $G$-spaces}, we mean a holomorphic, $G$-equivariant map $f:M_1\longrightarrow M_2$ satisfying $\mu_1=\mu_2\circ f$ and $f_*((\sigma_1)_p)=(\sigma_2)_{f(p)}$ for all $p\in M$. A quasi-Poisson Hamiltonian $G$-space $(M,\sigma,\mu)$ is called \textit{non-degenerate} if $T_pM=\mathrm{im}(\sigma_p^{\vee})+T_p(G\cdot p)$ for all $p\in M$. 
	
	To describe a specialization of the previous paragraph, we again let $G$ act holomorphically on a manifold $M$. Consider a $G$-invariant $2$-form $\omega\in\Omega^2(M)^G$ and $G$-equivariant holomorphic map $\mu:M\longrightarrow G$, where $G$ acts on itself by conjugation. One calls $(M,\omega,\mu)$ a \textit{quasi-Hamiltonian $G$-space} \cite{ale-mal-mei:jdg} if the following conditions are satisfied:
	\begin{itemize}
		\item[\textup{(i)}] $\mathrm{d}\omega=-\mu^*\eta_G$;
		\item[\textup{(ii)}] $\mathrm{ker}(\omega_p)=\{\xi_M(p):\xi\in\mathrm{ker}(\mathrm{Ad}_{\mu(p)}+\mathrm{id}_{\g})\}$ for all $p\in M$;
		\item[\textup{(iii)}] $\iota_{\xi_M}\omega=\frac{1}{2}\mu^*(\theta^L+\theta^R,\xi)$ for all $\xi\in\g$.
	\end{itemize}
	If $(M,\sigma,\mu)$ is a non-degenerate quasi-Poisson Hamiltonian $G$-space, then there exists a unique $2$-form $\omega\in\Omega^2(M)$ for which $\sigma$ and $\omega$ satisfy \cite[Equation (32)]{ale-kos-mei:02} and $(M,\omega,\mu)$ is a quasi-Hamiltonian $G$-space. Quasi-Hamiltonian $G$-spaces are thereby in bijective correspondence with non-degenerate quasi-Poisson Hamiltonian $G$-spaces. At the same time, every quasi-Poisson Hamiltonian $G$-space partitions into quasi-Hamiltonian leaves \cite{ale-kos-mei:02}.

	\subsection{Quasi-symplectic groupoids over quotients of quasi-Hamiltonian manifolds}\label{Subsection: Integrating}
	
	Suppose that $G$ and $H$ are Lie groups with fixed, non-degenerate, invariant, symmetric bilinear forms on their respective Lie algebras. Let $(M,\omega)$ be a quasi-Hamiltonian $G\times H$-space with moment map $(\mu,\nu):M\longrightarrow G\times H$. If $G$ acts freely and properly on $M$, then $M/G$ carries a unique quasi-Poisson $H$-space structure for which the quotient map $\pi:M\longrightarrow M/G$ is a morphism of quasi-Poisson $H$-spaces \cite[Corollary 4.7]{balibanu-mayrand}. The moment map on $M/G$ is obtained by descending $\nu$ to a holomorphic map $\overline{\nu}:M/G\longrightarrow H$. 
	
	On the other hand, $M\times\overline{M}$ is a quasi-Hamiltonian $G\times G\times H\times H$-space. Two applications of \textit{fusion} \cite[Theorem 6.1]{ale-mal-mei:jdg} render $M\times\overline{M}$ a quasi-Hamiltonian $G\times H$-space. The group $G\times H$ thereby acts diagonally on $M\times M$, and with moment map
	$$M\times M\tto G\times H,\quad (p,q)\mapsto (\mu(p)\mu(q)^{-1},\nu(p)\nu(q)^{-1}).$$
	It follows that the quasi-Hamiltonian reduction $(M\times\overline{M})\sll{e}G=(M\times_GM)/G$ is a quasi-Hamiltonian $H$-space; write $\sigma$ for the underlying $2$-form on $(M\times_G M)/G$. One also finds that the Lie groupoid structure on $\mathrm{Pair}(M)$ induces such a structure on $(M\times_G M)/G \tto M/G$; write $\mathrm{Pair}(M)\sll{e}G$ for this new Lie groupoid. 
	
	\begin{proposition}\label{Proposition: Integrating quasi-Poisson manifolds}
		If $G$ acts freely and properly on $M$, then $(\mathrm{Pair}(M)\sll{e}G,\sigma,\overline{\nu}^*\eta_H)$ is a quasi-symplectic groupoid.
	\end{proposition}
	
	\begin{proof}
		We first check the following compatibility conditions:
		\begin{enumerate}[label={(\arabic*)}]
			\item\label{ganjezmq} $\sigma$ is multiplicative;
			\item\label{pnusd6gu} $\mathrm{d}\sigma = \sss^*\overline{\nu}^*\eta_H - \ttt^*\overline{\nu}^*\eta_H$; 
			\item\label{oqxgxxr1} $\mathrm{d}\overline{\nu}^*\eta_H = 0$.
		\end{enumerate}
		Condition \ref{oqxgxxr1} is clear.
		
		Let $i : M \times_G M \too M \times M$ be the inclusion map and $\pi : M \times_G M \too (M \times_G M) / G$ the quotient map.
		It follows from the definition of the quasi-Hamiltonian 2-form on the fusion product $M \times \overline{M}$ that its pullback by $i^*$ coincides with $i^*(\pr_1^*\omega - \pr_2^*\omega)$, where $\pr_i : M \times M \to M$ are the projections. This implies that
		\begin{equation}\label{6ky82fwl}
			\pi^*\sigma = i^*(\pr_1^*\omega - \pr_2^*\omega).
		\end{equation}
		Since $\pr_1^*\omega - \pr_2^*\omega$ is multiplicative on $\mathrm{Pair}(M)$, it follows that $\sigma$ is also multiplicative, verifying \ref{ganjezmq}.
		To check that \ref{pnusd6gu} holds, note that \eqref{6ky82fwl} gives
		\begin{align*}
			\pi^*\mathrm{d}\sigma &= i^*(\pr_1^*\mathrm{d}\omega - \pr_2^*\mathrm{d}\omega) \\
			&= i^*(\pr_1^*(-\mu^*\eta_G - \nu^*\eta_H) + \pr_2^*(\mu^*\eta_G + \nu^*\eta_H)) \\
			&= \pi^*(\sss^*\overline{\nu}^*\eta_H - \ttt^*\overline{\nu}^*\eta_H),
		\end{align*}
		where the last equality follows from the fact that $\mu \circ \pr_1 \circ i = \mu \circ \pr_2 \circ i$.
		
		We have $\dim (M \times_G M)/G = 2 \dim M/G$, so it only remains to show that $\ker \sigma_{(p, p)} \cap \ker \mathrm{d}\sss \cap \ker \mathrm{d}\ttt = 0$ for all $p \in M$.
		A general vector in $\ker \mathrm{d}\sss \cap \ker \mathrm{d}\ttt$ takes the form $(0,(\xi_M)_p)$ for $\xi\in\g$, where $\mathrm{d}\mu_p((\xi_M)_p) = 0$, i.e.\ $\Ad_{\mu(p)^{-1}}\xi - \xi = 0$.
		The condition that $(0,(\xi_M)_p) \in \ker \sigma$ then amounts to the condition that $(\xi_M)_p \in \ker \omega_p = \{v_\zeta : \zeta \in \ker(\Ad_{\mu(p)} + 1)\}$.
		It follows that also $\Ad_{\mu(p)}(\xi) + \xi = 0$, and hence $\xi = 0$.
	\end{proof}
	
	\begin{remark}\label{Remark: Leaves}
		Suppose that the hypotheses of Proposition \ref{Proposition: Integrating quasi-Poisson manifolds} are satisfied. A straightforward exercise shows that the orbits of $\mathrm{Pair}(M)\sll{e}G\tto M/G$ are precisely the subsets $\mu^{-1}(\mathcal{C})/G\s M/G$, where $\mathcal{C}$ ranges over the conjugacy classes $\mathcal{C}\s G$. We exploit this observation in the proof of Theorem \ref{Theorem: Quasi leaf intersection}.
	\end{remark}
	
	\section{TQFTs in a $1$-shifted Weinstein symplectic category}\label{Section: New TQFTs}
	This section outlines some results of \cite{cro-may2:24}, including the construction of TQFTs in a completion of a $1$-shifted Weinstein symplectic ``category". In Subsection \ref{Subsection: MT conjecture}, we recall the Moore--Tachikawa conjecture and outline progress on its resolution. Details on a completed $1$-shifted Weinstein symplectic ``category" $\WS$ are given in Subsection \ref{Subsection: Shifted}. Subsection \ref{Subsection: TQFTs} then outlines a main result of \cite{cro-may2:24}, in which quasi-symplectic groupoids with admissible global slices determine TQFTs in $\WS$. Connections to the Moore--Tachikawa conjecture are described in Subsection \ref{Subsection: Connections}.
	
	\subsection{The Moore--Tachikawa conjecture}\label{Subsection: MT conjecture}
	Let $\mathbf{C}$ be a symmetric monoidal category. A two-dimensional topological quantum field theory (TQFT) in $\mathbf{C}$ is a symmetric monoidal functor $\Cob_2\longrightarrow\mathbf{C}$, where $\Cob_2$ is the category of two-dimensional cobordisms. While one often takes $\mathbf{C}$ to be vector spaces over a fixed field, other cases warrant consideration. This is a context in which to formulate the Moore--Tachikawa conjecture \cite{moo-tac:11}. Moore and Tachikawa take $\mathbf{C}$ to be $\MT$, a category of affine symplectic varieties with Hamiltonian actions. Complex semisimple affine algebraic groups constitute the objects of $\MT$. Morphisms in $\mathrm{Hom}_{\MT}(G,H)$ are certain isomorphism classes of affine Hamiltonian $G\times H$-varieties. One composes morphisms in $\MT$ via Hamiltonian reduction at level zero: given $[X]\in\mathrm{Hom}_{\MT}(G,H)$ and $[Y]\in\mathrm{Hom}_{\MT}(H,K)$, we have $[Y]\circ[X]\coloneqq [(X\times Y)\sll{0}H]\in\mathrm{Hom}_{\MT}(G,K)$. Products of objects and morphisms constitute the tensor product of a symmetric monoidal structure on $\MT$. 
	
	Now let $G$ be a connected semisimple affine algebraic group with Lie algebra $\g$. Consider the Kostant slice \cite{kostant} $\mathrm{Kos}\s\g$ associated to a principal $\mathfrak{sl}_2$-triple in $\g$ \cite{kostant-american}, i.e. an $\mathfrak{sl}_2$-triple $(e,h,f)\in\g^{\times 3}$ with $e,h,f\in\greg$. It turns out that the affine variety $G\times\mathrm{Kos}$ is symplectic in a natural way, as follows from generalities on \textit{Kostant--Whittaker reduction} \cite{kostant-inventiones}; see \cite{bielawski97,crooks-roeser} as well. Left multiplication in the first factor defines a Hamiltonian action of $G$ on $G\times\mathrm{Kos}$. As such, the isomorphism class of $G\times\mathrm{Kos}$ is a morphism from $G$ to the trivial group in $\MT$. This is a context for the Moore--Tachikawa conjecture; we state it below.
	
	\begin{conjecture}[Moore--Tachikawa \cite{moo-tac:11}]\label{a2j0tu3a}
		Let $G$ be a connected semisimple affine algebraic group with Lie algebra $\g$. Suppose that $\mathrm{Kos}\s\g$ is the Kostant slice associated to a principal $\mathfrak{sl}_2$-triple in $\g$. There exists a two-dimensional TQFT $\eta_{G}:\Cob_2\longrightarrow\MT$ satisfying $\eta_G(S^1)=G$ and $\eta_G(\begin{tikzpicture}[
			baseline=-2.5pt,
			every tqft/.append style={
				transform shape, rotate=90, tqft/circle x radius=4pt,
				tqft/circle y radius= 2pt,
				tqft/boundary separation=0.6cm, 
				tqft/view from=incoming,
			}
			]
			\pic[
			tqft/cup,
			name=d,
			every incoming lower boundary component/.style={draw},
			every outgoing lower boundary component/.style={draw},
			every incoming boundary component/.style={draw},
			every outgoing boundary component/.style={draw},
			cobordism edge/.style={draw},
			cobordism height= 1cm,
			];
		\end{tikzpicture}) = [G \times \mathrm{Kos}]$.
	\end{conjecture}  
	
	The Moore--Tachikawa conjecture is known to hold in Lie type $A$, as follows from combining the unpublished work of Ginzburg--Kazhdan \cite{gin-kaz:23} with the results of Braverman--Finkelberg--Nakajima \cite{bra-fin-nak:19}. In \cite{cro-may2:24}, we reformulate the Moore--Tachikawa conjecture in the language of shifted symplectic geometry. We also generalize this reformulation to obtain new TQFTs in a $1$-shifted Weinstein symplectic category. The following are some pertinent details.  
	
	\subsection{A $1$-shifted Weinstein symplectic ``category"}\label{Subsection: Shifted}
	Recall that two morphisms in a category are composable if and only if the source of one is the target of the other. By relaxing this to a necessary condition for composing morphisms, we arrive at the definition of a ``category". One particularly notable example is Weinstein's symplectic ``category" \cite{weinstein-symplectic-category}. We briefly recall our $1$-shifted counterpart of this ``category"; see \cite[Section 3]{cro-may2:24} for a more comprehensive discussion. 
	
	Let $(\G\tto X,\omega,\phi)$ be a quasi-symplectic groupoid. A \textit{1-shifted Lagrangian} on $\G$ is the data of a Lie groupoid $\mathcal{L}\tto Y$, $2$-form $\gamma$ on $Y$, and Lie groupoid morphism $\nu:\mathcal{L}\longrightarrow\G$, such that certain compatibility conditions are satisfied; see \cite{ptvv} or \cite[Section 3.1]{may:23}. A \textit{$1$-shifted Lagrangian relation} from a quasi-symplectic groupoid $\G_1$ to a quasi-symplectic groupoid $\G_2$ is a $1$-shifted Lagrangian $\mathcal{L}\longrightarrow\G_1\times\overline{\G_2}$ on $\G_1\times\overline{\G_2}$. We sometimes adopt the notation \[
	\begin{tikzcd}[row sep={1.5em,between origins},column sep={3em,between origins}]
		& \mathcal{L} \arrow{dl} \arrow{dr} & \\
		\G_1 & & \G_2
	\end{tikzcd}
	\] 
	for such a correspondence. Two such correspondences
	$$\begin{tikzcd}[row sep={1.5em,between origins},column sep={3em,between origins}]
		& \mathcal{L}_1 \arrow{dl} \arrow{dr} & & \mathcal{L}_2 \arrow{dl} \arrow{dr} & \\
		\G_1 & & \G_2 & & \G_3
	\end{tikzcd}$$
	are declared to be composable if a transversality condition is satisfied; see \cite[Subsection 2.3]{cro-may2:24} for further details. In this case, the homotopy fiber product \cite[Subsection 4.2]{may:23} $$\begin{tikzcd}[row sep={2em,between origins},column sep={4em,between origins}]
		& \mathcal{L}_1 \htimes_{\G_2} \mathcal{L}_2 \arrow{dl} \arrow{dr} & \\
		\G_1 & & \G_3
	\end{tikzcd}$$
	is a $1$-shifted Lagrangian relation from $\G_1$ to $\G_3$.
	
	In light of the above, we define the \textit{1-shifted Weinstein symplectic ``category"} $\WSQ$ as follows. Quasi-symplectic groupoids constitute the objects of $\WSQ$. Morphisms from $\G_1$ and $\G_2$ are weak equivalence classes \cite[Section 2.4]{cro-may2:24} of $1$-shifted Lagrangian relations from $\G_1$ to $\G_2$. Given composable $1$-shifted Lagrangian relations $$\begin{tikzcd}[row sep={1.5em,between origins},column sep={3em,between origins}]
		& \mathcal{L}_1 \arrow{dl} \arrow{dr} & & \mathcal{L}_2 \arrow{dl} \arrow{dr} & \\
		\G_1 & & \G_2 & & \G_3
	\end{tikzcd},$$ we define
	$$[\mathcal{L}_2]\circ[\mathcal{L}_1]\coloneqq [\mathcal{L}_1 \htimes_{\G_2} \mathcal{L}_2]\in\Hom_{\WSQ}(\G_1,\G_3).$$ Using a variant of the Wehrheim--Woodward approach \cite{weh-woo:10}, we complete $\WSQ$ to a genuine category $\WS$. We subsequently show $\WS$ to be a symmetric monoidal category; see \cite[Section 3.3]{cro-may2:24}. 
	
	\subsection{TQFTs valued in $\WS$}\label{Subsection: TQFTs}
	Recall that if $M$ is a Hamiltonian $\G_1 \times \overline{\G_2}$-space for quasi-symplectic groupoids $\G_1$ and $\G_2$, then the action groupoid $(\G_1 \times \overline{\G}_2) \ltimes M$ is a 1-shifted Lagrangian relation from $\G_1$ to $\G_2$ via the two projections; see \cite[Example 1.31]{cal:15} or \cite[Proposition 9.3]{may:23}. In this way, Hamiltonian $\G_1 \times \overline{\G_2}$-spaces are equivalent to $1$-shifted Lagrangian structures on $(\G_1 \times \overline{\G_2}) \ltimes M$ from $\G_1$ to $\G_2$. Note that this perspective allows one to interpret Hamiltonian $\G_1 \times \overline{\G_2}$-spaces as morphisms in $\WSQ$.
	
	Let $\G\tto X$ be a Lie groupoid. A submanifold $S\s X$ is called a \textit{global slice} if its intersection with each $\G$-orbit in $X$ is transverse and a singleton. This slice is called \textit{admissible} if the isotropy group $\G_x\coloneqq\sss^{-1}(x)\cap\ttt^{-1}(x)$ is abelian for all $x\in S$. For $\G\tto X$ a symplectic groupoid with a global slice $S\s X$, it follows that  $\ttt^{-1}(S) \s \G$ is a Hamiltonian $\G$-space with respect to the action by right multiplication. This leads to the following result from \cite{cro-may2:24}.
	
	\begin{theorem}\label{Theorem: TQFTs from quasi-symplectic groupoids}
		Consider a quasi-symplectic groupoid $(\G\tto X,\omega,\phi)$ and admissible global slice $S\s X$. If $\phi$ restricts to an exact 3-form on $S$, then there exists an explicit TQFT $\eta_{\G,S}:\Cob_2\longrightarrow\WS$ satisfying $\eta_{\G,S}(S^1)=\G$ and $\eta_G(\begin{tikzpicture}[
			baseline=-2.5pt,
			every tqft/.append style={
				transform shape, rotate=90, tqft/circle x radius=4pt,
				tqft/circle y radius= 2pt,
				tqft/boundary separation=0.6cm, 
				tqft/view from=incoming,
			}
			]
			\pic[
			tqft/cup,
			name=d,
			every incoming lower boundary component/.style={draw},
			every outgoing lower boundary component/.style={draw},
			every incoming boundary component/.style={draw},
			every outgoing boundary component/.style={draw},
			cobordism edge/.style={draw},
			cobordism height= 1cm,
			];
		\end{tikzpicture})=[\ttt^{-1}(S)]$.
	\end{theorem}
	
	As explained in \cite[Subsection 4.4]{cro-may2:24}, the TQFT structure comes from the abelian symplectic groupoid $\A \coloneqq \G\big\vert_S$; it is Morita equivalent to $\G$ via the inclusion map $\A \too \G$. Let us set
	\[
	\A^{m, n} \coloneqq \{(a, b) \in \A^{*m} * \A^{*n} : a_1 \cdots a_m = b_1 \cdots b_n\},
	\]
	for $m,n\in\mathbb{Z}_{\geq 0}$ with $(m,n)\neq (0,0)$, where $*$ denotes the fiber product over $S$.
	We have that $\eta_{\G,S}(C_{m, n}) = [\A^{m, n}]$ for all $m,n\in\mathbb{Z}_{\geq 0}$ with $(m,n)\neq (0,0)$, where $C_{m,n}$ is the standard $2$-dimensional cobordism from $m$ circles to $n$ circles. One views $\A^{m, n}$ as a Lagrangian relation from $\G^m$ to $\G^n$ via the maps $\A^{m, n} \longrightarrow \A^m \longhookrightarrow \G^m$ and $\A^{m, n} \longrightarrow \A^n \longhookrightarrow \G^n$; see \cite[Remark 4.6 and Subsection 4.4]{cro-may2:24}.
	
	The morphism $\eta_{\G,S}(C_{m, n})$ from $\G^m$ to $\G^n$ can be described as a Hamiltonian $\G^m \times \overline{\G}^n$-space in the following way.
	The action groupoid $(\G \times \G) \ltimes \G$ for the Hamiltonian $\G \times \overline{\G}$-space $\G$ acts as an identity in $\WS$.
	By composing on both sides with identities
	\[
	\begin{tikzcd}[column sep={6em,between origins}]
		&(\G^m \times \G^m) \ltimes \G^m \arrow{dl} \arrow{dr} & & \A^{m, n} \arrow{dl} \arrow{dr} & & (\G^n \times \G^n) \ltimes \G^n \arrow{dl} \arrow{dr} & \\
		\G^m & & \G^m & & \G^n & & \G^n
	\end{tikzcd},
	\]
	we see that $\eta_{\G}(C_{m, n})$ can be described as
	\begin{equation}\label{2hmb6k3v}
		\begin{tikzcd}[column sep={6em,between origins}]
			&(\G^m \times \A^{m, n} \times \G^n) \ltimes Z \arrow{dl} \arrow{dr} & \\
			\G^m & & \G^n
		\end{tikzcd},
	\end{equation}
	where
	\[
	Z \coloneqq \{(g, h) \in \G^m \times \G^n : \ttt(g_1) = \cdots = \ttt(g_m) = \sss(h_1) = \cdots = \sss(h_n) \in S\}.
	\]
	Since $\A^{m, n}$ is closed in $\G^m \times \G^n$, it acts freely and properly on $Z$.
	It follows that \eqref{2hmb6k3v} is equivalent to
	\[
	\begin{tikzcd}[column sep={6em,between origins}]
		&(\G^m \times \G^n) \ltimes (Z/\A^{m, n}) \arrow{dl} \arrow{dr} & \\
		\G^m & & \G^n;
	\end{tikzcd}
	\]
	see \cite[Theorem 7.2]{may:23}.
	On the other hand, a Hamiltonian space is equivalent to a 1-shifted Lagrangian structure on an action groupoid \cite[Proposition 9.3]{may:23}. It follows that
	\[
	\G^{m, n}_S \coloneqq Z/\A^{m, n}
	\]
	is a Hamiltonian $\G^m \times \overline{\G}^n$-space with respect to the $2$-form $\eta$ determined by
	\begin{equation}\label{psri9l81}
		\pi^*\eta = i^*(\omega, \ldots, \omega),
	\end{equation}
	where $\pi : Z \to Z / \A^{m, n}$ is the quotient map, $i : Z \longhookrightarrow \G^m \times \G^n$ is the inclusion map, and $(\omega, \ldots, \omega)$ is the natural 2-form on $\G^m \times \G^n$.
	We summarize this discussion in the following proposition.
	
	\begin{proposition}\label{Proposition: Explicit}
		Consider a quasi-symplectic groupoid $(\G\tto X,\omega,\phi)$ and admissible global slice $S\s X$ with the property that $\phi$ restricts to an exact 3-form on $S$. The TQFT $\eta_{\G,S}:\Cob_2\longrightarrow\WS$ then satisfies $\eta_{\G,S}(C_{m,n}) = \G_S^{m,n}$ for all $(m, n) \ne (0, 0)$.
	\end{proposition}
	
	In the case where $\G$ is a symplectic groupoid, the Hamiltonian space $Z/\A^{m, n}$ can be seen as a reduction along a submanifold in the sense of Section \ref{Subsection: reduction}.
	To this end, note that $\G^{m + n}$ is a Hamiltonian $\overline{\G}^m \times \G^n \times \overline{\G}^n \times \G^m$ space. One finds that $$S^{m, n} = \{(x, y) \in X^n \times X^m : x_1 = \cdots = x_n = y_1 = \cdots = y_m \in S\}$$ is pre-Poisson with stabilizer subgroupoid $\A^{m, n}$. It follows from \eqref{psri9l81} that $\G^{m, n}_S$ is the reduction of $\G^{m + n}$ along $S^{m,n}$ with respect to $\A^{m, n} \longrightarrow S^{m, n}$.
	
	\begin{remark}
		More generally, if $\G$ is a quasi-symplectic groupoid, one can see that $\G^{m,n}_S$ is a reduction in the sense of \cite{balibanu-mayrand}.
	\end{remark}
	
	\subsection{Connections to the Moore--Tachikawa conjecture and multiplicative counterparts}\label{Subsection: Connections}
	Let $G$ be a connected semisimple affine algebraic group with Lie algebra $\g$. Use the Killing form to freely identify $\g^*$ and $\g$ as Poisson varieties. At the same time, let $(e,h,f)\in\g^{\times 3}$ be a principal $\mathfrak{sl}_2$-triple. The affine subvariety $\mathrm{Kos}\coloneqq e+\g_f\s\g$ is an admissible global slice to the restricted cotangent groupoid $(T^*G)_{\text{reg}}\coloneqq (T^*G)\big\vert_{\g_{\text{reg}}}\tto\greg$. Theorem \ref{Theorem: TQFTs from quasi-symplectic groupoids} then yields an explicit TQFT $\eta_{\G,S}:\Cob_2\longrightarrow\WS$, where $\G=(T^*G)_{\text{reg}}$ and $S=\mathrm{Kos}$. Note that Proposition \ref{Proposition: Explicit} gives explicit Hamiltonian spaces in the image of $\eta_{\G,S}$; these are called the \textit{open Moore--Tachikawa varieties}. In \cite{cro-may2:24}, we ``affinize" $\eta_{\G,S}$ to construct the Moore--Tachikawa TQFT in a category of affine Hamiltonian schemes. 
	
	There is a multiplicative counterpart to the above. A first step is to replace $T^*G\tto\g$ with the \textit{double} $\mathrm{D}(G)=G\times G\tto G$, a quasi-symplectic groupoid integrating the Cartan--Dirac structure on $G$ \cite{ale-mal-mei:jdg}. One then replaces $\mathrm{Kos}\s\g$ with a \textit{Steinberg slice} $\mathrm{Ste}\s G$ \cite{Steinberg}. If $G$ is simply-connected, it turns out that $\mathrm{Ste}$ is an admissible global slice to the quasi-symplectic groupoid $\mathrm{D}(G)_{\text{reg}}\coloneqq\mathrm{D}(G)\big\vert_{G_{\text{reg}}}\tto G_{\text{reg}}$. By \cite[Example 2.30]{balibanu-mayrand}, this groupoid $\G$ and slice $S$ satisfy the hypotheses of Theorem \ref{Theorem: TQFTs from quasi-symplectic groupoids} and Proposition \ref{Proposition: Explicit}. The quasi-Hamiltonian spaces $\eta_{\G,S}(C_{m,n})$ are studied in \cite{balibanu-mayrand}, as multiplicative analogues of the open Moore--Tachikawa varieties. We refer the reader to Sections \ref{Section: Some Lie-theoretic considerations in quasi-Poisson geometry} and \ref{Section: Multiplicative TQFTs} for precise descriptions of the objects in this paragraph.
	
	\section{Some Lie-theoretic considerations in Poisson geometry}\label{Section: Some Lie-theoretic considerations in Poisson geometry}
	In this section, we prove certain Lie-theoretic results needed to realize Poisson-geometric aspects of partial Grothendieck--Springer resolutions. Subsection \ref{Subsection: Parabolic subgroups and subalgebras} introduces universal Levi factors for parabolic subgroups and subalgebras, while Subsection \ref{Subsection: Regular} reviews regular elements in reductive Lie algebras. A useful characterization of regular elements in Levi factors is given in Subsection \ref{Subsection: Regular elements in Levi factors}. We then devote Subsections \ref{Subsection: Residual} and \ref{Subsection: The variety} to the Hamiltonian Poisson geometry of $G\times_P\p$, where $P$ is a parabolic subgroup with Lie algebra $\p$.
	
	For the balance of this manuscript, $G$ is a connected semisimple affine algebraic group with Lie algebra $\mathfrak{g}$, rank $\ell$, adjoint representation $\mathrm{Ad}:G\longrightarrow\operatorname{GL}(\mathfrak{g})$, and exponential map $\exp:\g\longrightarrow G$.
	
	\subsection{Parabolic subgroups and subalgebras}\label{Subsection: Parabolic subgroups and subalgebras}
	Given a parabolic subgroup $P\s G$ (resp. parabolic subalgebra $\mathfrak{p}\s\mathfrak{g}$), write $U(P)\s P$ (resp. $\mathfrak{u}(\mathfrak{p})\s\mathfrak{p}$) for the unipotent radical of $P$ (resp. nilpotent radical of $\mathfrak{p}$). The quotients $L(P)\coloneqq P/U(P)$ and $\mathfrak{l}(\mathfrak{p})\coloneqq\mathfrak{p}/\mathfrak{u}(\mathfrak{p})$ will be called the \textit{universal Levi factors} of $P$ and $\mathfrak{p}$, respectively. If $\mathfrak{p}$ is the Lie algebra of $P$, then $\mathfrak{u}(\mathfrak{p})$ and $\mathfrak{l}(\mathfrak{p})$ are the Lie algebras of $U(P)$ and $L(P)$, respectively. 
	
	\subsection{Regular elements in reductive Lie algebras}\label{Subsection: Regular} Let $K$ be a connected reductive affine algebraic group with Lie algebra $\mathfrak{k}$. Write $K_x\s K$ and $\mathfrak{k}_x\s\mathfrak{k}$ for the $K$ and $\mathfrak{k}$-centralizers of $x\in\mathfrak{k}$ under the adjoint representation, respectively. Consider the \textit{regular locus} $$\mathfrak{k}_{\text{reg}}\coloneqq\{x\in\mathfrak{k}:\dim\mathfrak{k}_x=\mathrm{rank}\hspace{2pt}\mathfrak{k}\}.$$ If we identify $\mathfrak{k}$ with $\mathfrak{k}^*$ via a non-degenerate, $K$-invariant, symmetric bilinear form on the former, then $\mathfrak{k}_{\text{reg}}$ becomes the regular locus of the Poisson variety $\mathfrak{k}^*$. We are principally interested in the regular locus of $\mathfrak{k}=\mathfrak{l}(\p)$ for a parabolic subalgebra $\p\s\g$.  
	
	\subsection{Regular elements in Levi factors}\label{Subsection: Regular elements in Levi factors}
	Let $P\s G$ be a parabolic subgroup with Lie algebra $\p\s\g$. Observe that $U(P)$ acts trivially on the $P$-module $\mathfrak{l}(\p)$. One thereby obtains an action of the reductive group $L(P)$ on $\mathfrak{l}(\p)$; it is the adjoint representation of $L(P)$. Write $P_{[x]}\s P$ and $L(P)_{[x]}\s L(P)$ for the stabilizers of $[x]\in\mathfrak{l}(\p)$ under the actions of $P$ and $L(P)$, respectively. It is straightforward to verify that $L(P)_{[x]}=P_{[x]}/U(P)$ for all $[x]\in\mathfrak{l}(\p)$.
	
	\begin{lemma}\label{Lemma: Single orbit}
		Suppose that $x\in\p$. 
		\begin{itemize}
			\item[\textup{(i)}] If $(x+\mathfrak{u}(\p))\cap\greg\neq\emptyset$, then $(x+\mathfrak{u}(\p))\cap\greg$ is an orbit of $P_{[x]}$ in $\g$.
			\item[\textup{(ii)}] One has $(x+\mathfrak{u}(\p))\cap\greg\neq\emptyset$ if and only if $[x]\in\mathfrak{l}(\p)_{\emph{reg}}$.
			\item[\textup{(iii)}] If $\mathfrak{p}$ is a Borel subalgebra, then $(x+\mathfrak{u}(\p))\cap\greg$ is non-empty.
		\end{itemize}
	\end{lemma}
	
	\begin{proof}
		We begin by proving (i). To this end, observe that $x+\mathfrak{u}(\p)$ is stable under the adjoint action of $P_{[x]}$ on $\g$. It therefore suffices to prove that the $P_{[x]}$-orbit of any point in $(x+\mathfrak{u}(\p))\cap\greg$ is open and dense in $x+\mathfrak{u}(\p)$. We also recognize that $P_{[x]}$-orbits are open in their closures, and that $x+\mathfrak{u}(\p)$ is irreducible. This further reduces us to proving that the dimension of the $P_{[x]}$-orbit of any point in $(x+\mathfrak{u}(\p))\cap\greg$ equals $\dim(x+\mathfrak{u}(\p))$. 
		
		Suppose that $y\in (x+\mathfrak{u}(\p))\cap\greg$. It follows that \begin{equation}\label{Equation: Multiple}\begin{split}\dim (P_{[x]}\cdot y) & = \dim P_{[x]}-\dim(P_{[x]}\cap G_y) \\ & \geq \dim P_{[x]}-\dim G_y \\ & =\dim P_{[x]}-\ell\\ &=\dim L(P)_{[x]}+\dim U(P)-\ell \\ & \geq \dim U(P) \\ & =\dim(x+\mathfrak{u}(\p)),\end{split}\end{equation} where the fifth line follows from the fact that $$\dim L(P)_{[x]}\geq\mathrm{rank}\hspace{1pt}L(P)=\mathrm{rank}\hspace{1pt}G=\ell.$$ Since $P_{[x]}\cdot y\s x+\mathfrak{u}(\p)$, the two inequalities in \eqref{Equation: Multiple} must be equalities. We conclude that $\dim (P_{[x]}\cdot y)=\dim(x+\mathfrak{u}(\p))$ and $\dim L(P)_{[x]}=\ell$. This verifies (i), as well as the forward implication in (ii).
		
		To prove the backward implication in (ii), choose a Cartan subalgebra $\h\s\g$ and Borel subalgebra $\mathfrak{b}\s\g$ satisfying $\h\s\mathfrak{b}\s\p$. These choices yield sets of roots $\Phi\s\mathfrak{h}^*$, positive roots $\Phi^{+}\s\Phi$, and simple roots $\Delta\s\Phi^{+}$. There is a unique subset $\Gamma\s\Delta$ satisfying $\p=\mathfrak{l}\oplus\mathfrak{u}(\p)$, where $$\mathfrak{l}\coloneqq\bigg(\bigoplus_{\alpha\in\mathrm{span}_{\mathbb{Z}}(\Gamma)\cap\Phi^+}\g_{-\alpha}\bigg)\oplus\h\oplus\bigg(\bigoplus_{\alpha\in\mathrm{span}_{\mathbb{Z}}(\Gamma)\cap\Phi^+}\g_{\alpha}\bigg).$$ It suffices to prove that $x+\mathfrak{u}(\p)\cap\greg\neq\emptyset$ for all $x\in\mathfrak{l}_{\text{reg}}$.
		
		Suppose that $x\in\mathfrak{l}_{\text{reg}}$. Choose $e_{\alpha}\in\g_{\alpha}\setminus\{0\}$ for each $\alpha\in\Delta$, and consider the elements
		$$e_{\Gamma}\coloneqq\sum_{\alpha\in\Gamma}e_{\alpha}\quad\text{and}\quad e_{\Delta}\coloneqq\sum_{\alpha\in\Delta}e_{\alpha}.$$
		At the same time, let $\mathfrak{b}_{-}\s\g$ denote the opposite Borel subalgebra with respect to $\mathfrak{h}$ and $\mathfrak{b}$. The regularity of $x$ in $\mathfrak{l}$ allows one to find $g\in L$ satisfying 
		$\mathrm{Ad}_g(x)\in e_{\Gamma}+\mathfrak{b}_{-}$ \cite[Theorem 8]{kostant}, where $L\s G$ is a Levi subgroup integrating $\mathfrak{l}$. Setting $y\coloneqq\mathrm{Ad}_g(x)$ and noting that $\mathrm{Ad}_g(\mathfrak{u}(\p))=\mathfrak{u}(\p)$, it suffices to prove that $(y+\mathfrak{u}(\p))\cap\greg\neq\emptyset$. On the other hand, note that $(e_{\Delta}-e_{\Gamma})\in\mathfrak{u}(\p)$ and $y+(e_{\Delta}-e_{\Gamma})\in e_{\Delta}+\mathfrak{b}_{-}$. The former implies that $y+(e_{\Delta}-e_{\Gamma})\in y+\mathfrak{u}(\p)$, and the latter tells us that $y+(e_{\Delta}-e_{\Gamma})\in\greg$ \cite[Lemma 10]{kostant}. We conclude that $(y+\mathfrak{u}(\p))\cap\greg\neq\emptyset$. This completes the proof of (ii). Part (iii) follows immediately from (ii).
	\end{proof}
	
	\begin{lemma}\label{Lemma: In Borel}
		Let $\p\s\g$ be a parabolic subalgebra. If $x\in\mathfrak{p}\cap\greg$, then $G_x\s P$. 
	\end{lemma}
	
	\begin{proof}
		Since $Z(G)\s P$, we may assume that $G$ is of adjoint type. We begin by letting $\mathfrak{p}_{[x]}\s\g$ denote the Lie algebra of $P_{[x]}$, and recalling that the inequalities in \ref{Equation: Multiple} are equalities. By using these equalities in the special case $y=x$, we find that $\dim\p_{[x]}-\dim(\p_{[x]}\cap\g_x)=\dim\p_{[x]}-\ell$. This amounts to having $\dim(\p_{[x]}\cap\g_x)=\ell$. We also have $\dim\g_x=\ell$, implying that $\mathfrak{p}_{[x]}\cap\g_x=\g_x$. This is equivalent to the inclusion $\g_{x}\s\p_{[x]}$. In particular, $\g_{x}\s\p$. The inclusion $G_x\s P$ now follows from the fact that $G_x$ is connected \cite[Proposition 14]{kostant}.
	\end{proof}
	
	\subsection{Residual actions on associated bundles}\label{Subsection: Residual} Consider a closed subgroup $H\s G$ and finite-dimensional $H$-module $V$. Let $G\times H$ act on $G\times V$ by
	$$(k,h)\cdot(g,v)\coloneqq(kgh^{-1},h\cdot v),\quad (k,h)\in G\times H,\text{ }(g,v)\in G\times V.$$ The action of $H=\{e\}\times H\s G\times H$ admits a geometric quotient $$G\times_H V\coloneqq (G\times V)/H,$$ to which the action of $G=G\times\{e\}\s G\times H$ descends. This quotient variety forms a $G$-equivariant vector bundle over $G/H$, with bundle projection
	$$\theta_{V}:G\times_H V\longrightarrow G/H,\quad [g:v]\mapsto [g].$$
	
	Now suppose that $K\s H$ is a closed, normal subgroup. We may form the geometric quotient variety $G\times_K V$, as above. Note that the action of $G\times H$ on $G\times V$ descends to an action of $G\times (H/K)$ on $G\times_K V$. The natural map $G\times_K V\longrightarrow G\times_H V$ is then a geometric quotient of $G\times_K V$ by $H/K = \{e\}\times (H/K)\s G\times (H/K)$.

	\subsection{The Poisson Hamiltonian $G$-variety $G\times_P\p$}\label{Subsection: The variety} Consider a parabolic subgroup $P\s G$ with Lie algebra $\mathfrak{p}\s\mathfrak{g}$. Let $\pi:G\longrightarrow G/U(P)$ denote the quotient morphism. The differential $\mathrm{d}\pi_e:\g\longrightarrow T_{[e]}(G/U(P))$ is surjective with kernel $\mathfrak{u}(\p)$, and so descends to a $U(P)$-module isomorphism $$\mathfrak{g}/\mathfrak{u}(\p)\overset{\cong}\longrightarrow T_{[e]}(G/U(P)),\quad [\xi]\mapsto \overline{\xi}\coloneqq\mathrm{d}\pi_e(\xi).$$ By inverting the induced isomorphism of dual vector spaces, one obtains a $U(P)$-module isomorphism
	$$\alpha:(\mathfrak{g}/\mathfrak{u}(\p))^*\overset{\cong}\longrightarrow T_{[e]}^*(G/U(P)).$$ Note that $$\alpha(\phi)(\overline{\xi})=\phi([\xi])$$ for all $\phi\in(\mathfrak{g}/\mathfrak{u}(\p))^*$ and $\xi\in\g$, and that $\alpha$ induces an isomorphism of $G$-equivariant vector bundles over $G/U(P)$. This isomorphism is given by
	\begin{equation}\label{Equation: Formula}\varphi:G\times_{U(P)}(\mathfrak{g}/\mathfrak{u}(\p))^*\overset{\cong}\longrightarrow T^*(G/U(P)),\quad [g:\phi]\mapsto ([g],\alpha(\phi)\circ ((\mathrm{d}L_g)_{[e]})^{-1}),\end{equation} where $L_g:G/U(P)\longrightarrow G/U(P)$ is left multiplication by $g\in G$.
	
	Consider the action of $G\times G$ on $G$ defined by 
	$$(h,k)\cdot g\coloneqq hgk^{-1},\quad (h,k)\in G\times G,\text{ }g\in G.$$ There is an induced action of $G\times L(P)$ on $G/U(P)$. The group $G\times L(P)$ thereby acts on the cotangent bundle $T^*(G/U(P))$. On the other hand, consider the $P$-modules $(\mathfrak{g}/\mathfrak{u}(\p))^*$ and $\mathfrak{p}$. The discussion in Subsection \ref{Subsection: Residual} yields induced actions of $G\times P$ on $G\times(\mathfrak{g}/\mathfrak{u}(\p))^*$ and $G\times\mathfrak{p}$. The same discussion implies that these induced actions descend to ones of $G\times L(P)$ on $G\times_{U(P)}(\mathfrak{g}/\mathfrak{u}(\p))^*$ and $G\times_{U(P)}\mathfrak{p}$. 
	
	\begin{proposition}\label{Proposition: Equivariance}
		The vector bundle isomorphism $\varphi:G\times_{U(P)}(\mathfrak{g}/\mathfrak{u}(\p))^*\longrightarrow T^*(G/U(P))$ is $G\times L(P)$-equivariant.
	\end{proposition}
	
	\begin{proof}
		It suffices to prove that $\varphi$ is $L(P)$-equivariant. An examination of \eqref{Equation: Formula} reveals that this is so if and only if
		$$\alpha(p\cdot\phi)\circ((\mathrm{d}L_{gp^{-1}})_{[e]})^{-1}=\alpha(\phi)\circ((\mathrm{d}L_g)_{[e]})^{-1}\circ((\mathrm{d}R_{p^{-1}})_{[e]})^{-1}$$ as elements of $T_{[gp^{-1}]}^*(G/U(P))$ for all $p\in P$ and $[g:\phi]\in G\times_{U(P)}(\mathfrak{g}/\mathfrak{u}(\p))^*$, where $R_{p^{-1}}:G/U(P)\longrightarrow G/U(P)$ is the result of letting right multiplication by $p^{-1}$ descend from a map $G\longrightarrow G$ to a map $G/U(P)\longrightarrow G/U(P)$. To this end, note that each element of $T_{[gp^{-1}]}(G/U(P))$ is given by $(\mathrm{d}L_{gp^{-1}})_{[e]}(\overline{\xi})$ for some $\xi\in\g$. We are therefore reduced to proving that
		\begin{equation}\label{Equation: Identity}(\alpha(p\cdot\phi)\circ((\mathrm{d}L_{gp^{-1}})_{[e]})^{-1})((\mathrm{d}L_{gp^{-1}})_{[e]}(\overline{\xi}))=(\alpha(\phi)\circ((\mathrm{d}L_g)_{[e]})^{-1}\circ((\mathrm{d}R_{p^{-1}})_{[e]})^{-1})((\mathrm{d}L_{gp^{-1}})_{[e]})(\overline{\xi}))\end{equation} for all for all $p\in P$, $[g:\phi]\in G\times_{U(P)}(\mathfrak{g}/\mathfrak{u}(\p))^*$, and $\xi\in\g$. Our proof will consist of showing that each side of \eqref{Equation: Identity} is equal to $\phi([\mathrm{Ad}_{p^{-1}}(\xi)])$. 
		
		Observe that left-hand side of \eqref{Equation: Identity} is
		$$\alpha(p\cdot\phi)(\overline{\xi})=(p\cdot\phi)([\xi])=\phi([\mathrm{Ad}_{p^{-1}}(\xi)]).$$ The right-hand side of \eqref{Equation: Identity} can be simplified by observing that
		$$((\mathrm{d}L_g)_{[e]})^{-1}\circ ((\mathrm{d}R_{p^{-1}})_{[e]})^{-1}\circ (\mathrm{d}L_{gp^{-1}})_{[e]}=\mathrm{d}(L_{g^{-1}}\circ R_p\circ L_{gp^{-1}})_{[e]}=\mathrm{d}(L_{p^{-1}}\circ R_p)_{[e]}=\overline{\mathrm{Ad}_{p^{-1}}},$$ where $\overline{\mathrm{Ad}_{p^{-1}}}:T_{[e]}(G/U(P))\longrightarrow T_{[e]}(G/U(P))$ is the result of letting $\mathrm{Ad}_{p^{-1}}:\g\longrightarrow\g$ descend to an automorphism of $T_{[e]}(G/U(P))$. The right-hand side is therefore given by
		$$\alpha(\phi)(\overline{\mathrm{Ad}_{p^{-1}}}(\overline{\xi}))=\alpha(\phi)(\overline{\mathrm{Ad}_{p^{-1}}(\xi)})=\phi([\mathrm{Ad}_{p^{-1}}(\xi)]),$$ completing the proof.
	\end{proof}
	
	Let $\langle\cdot,\cdot\rangle:\g\otimes\g\longrightarrow\IC$ be the Killing form. One knows that $\mathfrak{p}$ is the annihilator of $\mathfrak{u}(\p)$ in $\mathfrak{g}$ with respect to this form. We thereby obtain a $P$-module isomorphism $\mathfrak{p}\longrightarrow(\mathfrak{g}/\mathfrak{u}(\p))^*$, whose composition with $\alpha:(\mathfrak{g}/\mathfrak{u}(\p))^*\overset{\cong}\longrightarrow T^*_{[e]}(G/U(P))$ is 
	\begin{equation}\label{Equation: Isomorphism beta}\beta:\mathfrak{p}\overset{\cong}\longrightarrow T^*_{[e]}(G/U(P)),\quad\beta(x)(\overline{\xi})=\langle x,\xi\rangle,\quad x\in\mathfrak{p},\text{ }\xi\in\mathfrak{g}.\end{equation}
	This combines with \eqref{Equation: Formula} and Proposition \ref{Proposition: Equivariance} to give a $G\times L(P)$-equivariant isomorphism
	\begin{equation}\label{Equation: Eq iso}\psi_{\p}:G\times_{U(P)}\mathfrak{p}\overset{\cong}\longrightarrow T^*(G/U(P)),\quad [g:x]\mapsto([g],\beta(x)\circ((\mathrm{d}L_g)_{[e]})^{-1})\end{equation} of vector bundles over $G/U(P)$. On the other hand, the $G\times L(P)$-action on $T^*(G/U(P))$ is Hamiltonian with respect to the canonical symplectic structure on the latter. We may therefore equip $G\times_{U(P)}\p$ with the symplectic Hamiltonian $G\times L(P)$-variety structure for which $\psi_{\p}$ is an isomorphism of symplectic Hamiltonian $G\times L(P)$-varieties. A straightforward exercise shows
	$$G\times_{U(P)}\p\longrightarrow\g\times\mathfrak{l}(\p),\quad [g:x]\mapsto (\mathrm{Ad}_g(x),-[x])$$ to be a moment map, where the Killing form is used to identify $\g^*$ (resp. $\mathfrak{l}(\p)^*$) with $\g$ (resp. $\mathfrak{l}(\p)$). It follows that $$G\times_P\p=(G\times_{U(P)}\p)/L(P)$$ is a Poisson Hamiltonian $G$-space with moment map
	$$\mu_{\p}:G\times_P\p\longrightarrow\g,\quad [g:x]\mapsto\mathrm{Ad}_g(x).$$
	
	\section{Partial Grothendieck--Springer resolutions}\label{Section: Partial Grothendieck--Springer resolutions}
	We now develop the relevant Poisson-geometric features of the partial Grothendieck--Springer resolutions $\mu_{\mathcal{C}}:\g_{\mathcal{C}}\longrightarrow\g$. These resolutions are formally defined in Subsection \ref{Subsection: The partial}. In Subsection \ref{Subsection: Technical}, we give several characterizations of a canonical Poisson Hamiltonian $G$-variety structure on $\mathfrak{g}_{\mathcal{C}}$.
	
	\subsection{The partial Grothendieck--Springer resolution $\mu_{\mathcal{C}}:\g_{\mathcal{C}}\longrightarrow\g$}\label{Subsection: The partial} The adjoint representation induces an action of $G$ on the set of parabolic subalgebras of $\mathfrak{g}$. An orbit of this action will be called a \textit{conjugacy class} of parabolic subalgebras. If $\mathcal{C}$ is one such conjugacy class, then the non-negative integer $d_{\mathcal{C}}\coloneqq\dim\mathfrak{p}$ is independent of $\mathfrak{p}\in\mathcal{C}$. It is straightforward to verify that $\mathcal{C}$ constitutes a closed $G$-orbit in the Grassmannian $\mathrm{Gr}(d_{\mathcal{C}},\g)$ of $d_{\mathcal{C}}$-dimensional subspaces of $\mathfrak{g}$. It follows that $\mathcal{C}$ is a smooth projective variety carrying an algebraic $G$-action.
	
	Fix a conjugacy class $\mathcal{C}$ of parabolic subalgebras of $\g$. Let $\pi_{\mathcal{C}}:\mathfrak{g}_{\mathcal{C}}\longrightarrow\mathcal{C}$ be the $G$-equivariant vector bundle obtained by pulling the tautological bundle on $\mathrm{Gr}(d_{\mathcal{C}},\g)$ back along the inclusion $\mathcal{C}\s\mathrm{Gr}(d_{\mathcal{C}},\g)$. It follows that
	$$\mathfrak{g}_{\mathcal{C}}=\{(\mathfrak{p},x)\in\mathcal{C}\times\g:x\in\mathfrak{p}\},$$ and that $\pi_{\mathcal{C}}(\p,x)=\mathfrak{p}$ for all $(\p,x)\in\g_{\mathcal{C}}$. One also has the $G$-equivariant morphism
	$$\mu_{\mathcal{C}}:\g_{\mathcal{C}}\longrightarrow\g,\quad(\p,x)\mapsto x.$$
	
	\begin{definition}
		The morphism $\mu_{\mathcal{C}}:\g_{\mathcal{C}}\longrightarrow\g$ is called the \textit{partial Grothendieck--Springer resolution} determined by $\mathcal{C}$.
	\end{definition}
	
	\subsection{The Poisson geometry of $\g_{\mathcal{C}}$}\label{Subsection: Technical}
	Let $\mathcal{C}$ be a conjugacy class of parabolic subalgebras of $\g$.
	Suppose that $\mathfrak{p}\in\mathcal{C}$, and let $P\s G$ be the parabolic subgroup integrating $\mathfrak{p}$. One has the $G$-equivariant variety isomorphisms
	$$\phi:G/P\overset{\cong}\longrightarrow\mathcal{C},\quad [g]\mapsto\mathrm{Ad}_g(\mathfrak{p})\quad\text{and}\quad \delta_{\p}:G\times_P\mathfrak{p}\overset{\cong}\longrightarrow\mathfrak{g}_{\mathcal{C}},\quad [g:x]\mapsto(\mathrm{Ad}_g(\mathfrak{p}),\mathrm{Ad}_g(x)),$$ and commutative diagram
	$$\begin{tikzcd}
		G\times_P\mathfrak{p}\arrow[r, "\delta_{\mathfrak{p}}"] \arrow[d] & \g_{\mathcal{C}} \arrow[d, "\pi_{\mathcal{C}}"] \\
		G/P \arrow[r, swap, "\phi"] & \mathcal{C}
	\end{tikzcd}.$$
	
	\begin{proposition}\label{Proposition: Image}
		If $(\p,x)\in\g_{\mathcal{C}}$, then the differential $(\mathrm{d}\mu_{\mathcal{C}})_{(\p,x)}:T_{(\p,x)}\g_{\mathcal{C}}\longrightarrow\g$ is an isomorphism if and only if $\g_x\cap\mathfrak{u}(\p)=\{0\}$.
	\end{proposition}
	
	\begin{proof}
		As $\dim\g_{\mathcal{C}}=\dim\g$, it suffices to prove that $(\mathrm{d}\mu_{\mathcal{C}})_{(\p,x)}$ is surjective if and only if $\g_x\cap\mathfrak{u}(\p)=\{0\}$. Consider the quotient morphism $\pi:G\times\mathfrak{p}\longrightarrow G\times_P\mathfrak{p}$. The composite map $\delta_{\mathfrak{p}}\circ\pi:G\times\mathfrak{p}\longrightarrow\g_{\mathcal{C}}$ is a submersion that sends $(e,x)$ to $(\p,x)$. It therefore suffices to prove that $$\mathrm{d}(\mu_{\mathcal{C}}\circ\delta_{\mathfrak{p}}\circ\pi)_{(e,x)}:\g\oplus\mathfrak{p}\longrightarrow\g$$ is surjective if and only if $\g_x\cap\mathfrak{u}(\p)=\{0\}$. On the other hand, a straightforward calculation shows $\mu_{\mathcal{C}}\circ\delta_{\mathfrak{p}}\circ\pi:G\times\mathfrak{p}\longrightarrow\g$ to be given by $$(\mu_{\mathcal{C}}\circ\delta_{\mathfrak{p}}\circ\pi)(g,y)=\mathrm{Ad}_g(y)$$ for all $(g,y)\in G\times\mathfrak{p}$. It follows that $$\mathrm{d}(\mu_{\mathcal{C}}\circ\delta_{\mathfrak{p}}\circ\pi)_{(e,x)}(\xi,\eta)=[\xi,x]+\eta$$ for all $(\xi,\eta)\in\g\oplus\mathfrak{p}$. The image of $\mathrm{d}(\mu_{\mathcal{C}}\circ\delta_{\mathfrak{p}}\circ\pi)_{(e,x)}$ is therefore $[\g,x]+\p$. We conclude that $\mathrm{d}(\mu_{\mathcal{C}}\circ\delta_{\mathfrak{p}}\circ\pi)_{(e,x)}$ is surjective if and only if $([\g,x]+\p)^{\perp}=\{0\}$, or equivalently $\g_x\cap\mathfrak{u}(\p)=\{0\}$.
	\end{proof}
	
	Recall the Poisson Hamiltonian $G$-variety structure on $G\times_P\p$ from Subsection \ref{Subsection: The variety}. In this context, we have the following result.
	
	\begin{proposition}\label{Proposition: Canonical} The following statements are true.
		\begin{itemize}
			\item[\textup{(i)}] There exists a unique Poisson structure on $\g_{\mathcal{C}}$ that makes the $G$-equivariant isomorphism $\delta_{\mathfrak{p}}:G\times_P\mathfrak{p}\longrightarrow\g_{\mathcal{C}}$ Poisson for all $\mathfrak{p}\in\mathcal{C}$.
			\item[\textup{(ii)}] The diagram $$\begin{tikzcd}
				G\times_P\p \arrow{rr}{\delta_{\p}} \arrow[swap]{dr}{\mu_{\p}} & & \g_{\mathcal{C}} \arrow{dl}{\mu_{\mathcal{C}}} \\[10pt]
				& \g
			\end{tikzcd}$$ commutes for all $\p\in\mathcal{C}$.
			\item[\textup{(iii)}] The action of $G$ on $\g_{\mathcal{C}}$ is Hamiltonian with respect to the Poisson structure in \textup{(i)}, and it admits $\mu_{\mathcal{C}}$ as a moment map.
			\item[\textup{(iv)}] The Poisson structure and Hamiltonian $G$-action on $\g_{\mathcal{C}}$ are uniquely determined by the property that $\mu_{\mathcal{C}}$ is a moment map.
		\end{itemize}
	\end{proposition}
	
	\begin{proof}
		Part (ii) is a straightforward computation, while (iii) is an immediate consequence of (i) and (ii). It therefore remains only to prove (i) and (iv). To prove (i), note that the uniqueness assertion follows immediately from the fact that $\delta_{\mathfrak{p}}$ is a bijection for all $\p\in\mathcal{C}$. The same fact also allows one to choose $\p\in\mathcal{C}$, and then endow $\g_{\mathcal{C}}$ with the Poisson variety structure for which $\delta_{\mathfrak{p}}$ is an isomorphism of Poisson varieties. Our task is to prove that this Poisson structure does not depend on $\p$. This is the task of showing $$\delta_{\mathfrak{p}_2}^{-1}\circ\delta_{\mathfrak{p}_1}:G\times_{P_1}\p_1\longrightarrow G\times_{P_2}\p_2$$ to be a Poisson variety isomorphism for all $\p_1,\p_2\in\mathcal{C}$ with corresponding parabolic subgroups $P_1,P_2\s G$, respectively.
		
		Suppose that $\p_1,\p_2\in\mathcal{C}$, and choose $h\in G$ satisfying $\mathrm{Ad}_h(\p_1)=\p_2$. A straightforward exercise reveals that
		$$(\delta_{\mathfrak{p}_2}^{-1}\circ\delta_{\mathfrak{p}_1})([g:x])=[gh^{-1}:\mathrm{Ad}_h(x)]$$ for all $[g:x]\in G\times_{P_1}\p_1$. Let us also observe that $$\eta:G\times_{U(P_1)}\p_1\longrightarrow G\times_{U(P_2)}\p_2,\quad [g:x]\mapsto [gh^{-1}:\mathrm{Ad}_h(x)],\quad [g:x]\in G\times_{U(P_1)}P_1$$ is well-defined and makes $$\begin{tikzcd}[column sep=large]
			G\times_{U(P_1)}\mathfrak{p}_1\arrow[r, "\eta"] \arrow[d] & G\times_{U(P_2)}\p_2 \arrow[d] \\
			G\times_{P_1}\mathfrak{p}_1 \arrow[r, "\delta_{\mathfrak{p}_2}^{-1}\circ\delta_{\mathfrak{p}_1}"] & G\times_{P_2}\p_2
		\end{tikzcd}$$ commute. It therefore suffices to prove that $\eta$ is an isomorphism of symplectic varieties. 
		
		Consider the variety isomorphism $$G/U(P_1)\overset{\cong}\longrightarrow G/U(P_2),\quad [g]\mapsto [gh^{-1}],\quad [g]\in G/U(P)_1,$$ and the symplectic variety isomorphism $$\zeta:T^*(G/U(P_1))\overset{\cong}\longrightarrow T^*(G/U(P_2))$$ that it induces. It is straightforward to check that $$\begin{tikzcd}
			G\times_{U(P_1)}\mathfrak{p}_1\arrow[r, "\eta"] \arrow[d, "\psi_{\mathfrak{p}_1}"] & G\times_{U(P_2)}\p_2 \arrow[d, "\psi_{\mathfrak{p}_2}"] \\
			T^*(G/U(P_1)) \arrow[r, "\zeta"] & T^*(G/U(P_2))
		\end{tikzcd}$$ commutes, where $\psi_{\p_1}$ and $\psi_{\p_2}$ are defined in \eqref{Equation: Eq iso}. Since $\zeta$, $\psi_{\p_1}$, and $\psi_{\p_2}$ are symplectic variety isomorphisms, the same must be true of $\eta$. This proves (i).
		
		We now prove (iv). A first step is to invoke Proposition \ref{Proposition: Image}; it implies that the differential of $\mu_{\mathcal{C}}$ is an isomorphism at each point in $\mu_{\mathcal{C}}^{-1}(\g_{\text{reg}}\cap \g_{\text{ss}})\s \g_{\mathcal{C}}$, where $\g_{\text{ss}}\s\g$ is the locus of semisimple elements. Since moment maps are Poisson morphisms (resp. $G$-equivariant morphisms), requiring $\mu_{\mathcal{C}}$ to be a moment map determines the Poisson bivector field (resp. generating vector fields of the $G$-action) on $\mu_{\mathcal{C}}^{-1}(\g_{\text{reg}}\cap \g_{\text{ss}})\s\g_{\mathcal{C}}$. We also note that $\mu_{\mathcal{C}}^{-1}(\g_{\text{reg}}\cap \g_{\text{ss}})$ is a non-empty open subset of the irreducible variety $\g_{\mathcal{C}}$. It follows that forcing $\mu_{\mathcal{C}}$ to be a moment map determines the Poisson bivector field on $\g_{\mathcal{C}}$, as well as the generating vector fields of the $G$-action on $\g_{\mathcal{C}}$. As $G$ is connected, these generating vector fields determine the $G$-action on $\g_{\mathcal{C}}$. The previous two sentences imply (iv).
	\end{proof}
	
	Recall that the \textit{regular locus} $X_{\text{reg}}$ of a Poisson manifold $X$ is the union of its top-dimensional symplectic leaves. The rank of a Poisson manifold is the supremum of its symplectic leaf dimensions. Returning to the notation of this subsection, we have the following result.
	
	\begin{proposition}\label{Proposition: Nice}
		The following statements are true.
		\begin{itemize}
			\item[\textup{(i)}]We have $(\g_{\mathcal{C}})_{\emph{reg}}=\{(\p,x)\in\g_{\mathcal{C}}:[x]\in\mathfrak{l}(\p)_{\emph{reg}}\}$.
			\item[\textup{(ii)}] The rank of $\g_{\mathcal{C}}$ is $\dim\g-\ell$.
		\end{itemize}
	\end{proposition}
	
	\begin{proof}
		We first prove (i). Suppose that $(\p,x)\in\g_{\mathcal{C}}$, and let $P\s G$ be the parabolic subgroup integrating $\p$. Recall that the canonical isomorphism $\delta_{\p}:G\times_P\p\longrightarrow\g_{\mathcal{C}}$ is one of Hamiltonian $G$-spaces. Since $\delta([e:x])=(\p,x)$, it suffices to prove that $[e:x]\in (G\times_P\p)_{\text{reg}}$ if and only if $[x]\in\mathfrak{l}(\p)_{\text{reg}}$.
		
		Recall that $G\times_{U(P)}\p$ is a symplectic Hamiltonian ($G\times L(P)$)-space with moment map
		$$(\nu_1,\nu_2):G\times_{U(P)}\p\longrightarrow\g\times\mathfrak{l}(\p),\quad [g:y]\mapsto (\mathrm{Ad}_g(y),-[y]).$$ In light of Subsection \ref{Subsection: The variety}, the symplectic leaves of $G\times_P\p$ are the connected components of the Hamiltonian reductions of $G\times_{U(P)}\p$ by $L(P)$. It follows that the symplectic leaves of $G\times_P\p$ are given by $$\nu_2^{-1}([y])/L(P)_{[y]}=G\times_{P_{[y]}}((-y)+\mathfrak{u}(\p))\s G\times_P\p,$$ as $[y]$ ranges over $\mathfrak{l}(\p)$. We conclude that $G\times_{P_{[x]}}(x+\mathfrak{u}(\p))$ is the symplectic leaf of $G\times_P\p$ through $[e:x]$. This implies that $[e:x]\in (G\times_P\p)_{\text{reg}}$ if and only if $[x]$ achieves the minimal $P$-centralizer dimension of vectors in $\mathfrak{l}(\p)$. On the other hand, it is clear that $\dim P_{[x]}=\dim L(P)_{[x]}+\dim U(P)$. One concludes that $\dim P_{[x]}$ is minimal among the $P$-centralizer dimensions of vectors in $\mathfrak{l}(\p)$ if and only if $[x]\in\mathfrak{l}(\p)_{\text{reg}}$. The proof of (i) is therefore complete.
		
		To prove (ii), suppose that $(\p,x)\in(\g_{\mathcal{C}})_{\text{reg}}$. It suffices to prove that the symplectic leaf of $\g_{\mathcal{C}}$ through $(\p,x)$ has dimension $\dim\g-\ell$. In light of the previous paragraph, this is the task of showing the dimension of $G\times_{P_{[x]}}(x+\mathfrak{u}(\p))$ to be $\dim\g-\ell$. We have \begin{align*}\dim (G\times_{P_{[x]}}(x+\mathfrak{u}(\p))) &= \dim G+\dim U(P)-\dim P_{[x]}\\ &=\dim G-\dim L(P)_{[x]}\\
			&=\dim\g-\ell,
		\end{align*}
		where the last line follows from (i).
	\end{proof}
	
	Let $\mathcal{B}$ denote the conjugacy class of Borel subalgebras of $\g$, i.e. the full flag variety of $G$. Note that $\mathfrak{g}_{\mathcal{B}}=\widetilde{\g}$ is the full Grothendieck--Springer resolution of $\g$ \cite{eva-lu:07}. The following well-known result is a special case of the previous result.
	
	\begin{corollary}\label{Corollary: Easy}
		The Poisson variety $\mathfrak{g}_{\mathcal{B}}=\widetilde{\g}$ is regular of rank $\dim\g-\ell$. 
	\end{corollary}
	
	\section{TQFTs from Grothendieck--Springer resolutions}\label{Section: TQFTs from}
	We now explain that each partial Grothendieck--Springer resolution $\mu_{\mathcal{C}}:\g_{\mathcal{C}}\longrightarrow\g$ determines a two-dimensional, $\WS$-valued TQFT, in such a way that setting $\mathcal{C}=\{\g\}$ recovers the open Moore--Tachikawa TQFT. Our first steps are to integrate $\g_{\mathcal{C}}$ to a symplectic groupoid $(T^*G)_{\mathcal{C}}\tto\g_{\mathcal{C}}$, and compute the isotropy groups of $(T^*G)_{\mathcal{C}}$; see Subsection \ref{Subsection: The symplectic groupoid}. Subsection \ref{Subsection: Global slices} then establishes that $\mathrm{Kos}_{\mathcal{C}}\coloneqq\mu_{\mathcal{C}}^{-1}(\mathrm{Kos})\s\mathfrak{g}_{\mathcal{C}}$ is an admissible global slice to the pullback groupoid $((T^*G)_{\mathcal{C}})_{\text{reg}}\tto(\g_{\mathcal{C}})_{\text{reg}}$, where $\mathrm{Kos}\s\g$ is a Kostant slice. Using Theorem \ref{Theorem: TQFTs from quasi-symplectic groupoids} and Proposition \ref{Proposition: Explicit}, we obtain an explicit TQFT in $\WS$. Subsection \ref{Subsection: Alterations} concludes with some connections to the Moore--Tachikawa conjecture.
	
	\subsection{The symplectic groupoid $(T^*G)_{\mathcal{C}}$}\label{Subsection: The symplectic groupoid}
	Let $\p\s\g$ be a parabolic subalgebra integrating to a parabolic subgroup $P\s G$. Recall the discussion of the symplectic variety $G\times_{U(P)}\p$ and Poisson variety $G\times_P\p=(G\times_{U(P)}\p)/L(P)$ in Subsection \ref{Subsection: The variety}. In light of Proposition \ref{Proposition: Integration}, we may form the symplectic groupoid
	$$(T^*G)_{\p}\coloneqq\mathrm{Pair}(G\times_{U(P)}\p)\sll{0}L(P)\tto G\times_P\p.$$
	Write $\sss_{\p}:(T^*G)_{\p}\longrightarrow\g_{\mathcal{C}}$ (resp. $\ttt_{\p}:(T^*G)_{\p}\longrightarrow\g_{\mathcal{C}}$) for the result of composing the source (resp. target) of $(T^*G)_{\p}\tto G\times_P\p$ with the Poisson isomorphism $\delta_{\p}:G\times_P\p\longrightarrow\g_{\mathcal{C}}$. It follows that $(T^*G)_{\p}$ is a symplectic groupoid over $\g_{\mathcal{C}}$ with source $\sss_{\p}$ and target $\ttt_{\p}$; groupoid multiplication and inversion are induced from those of the pair groupoid of $G\times_{U(P)}\p$, and composing $\delta_{\p}^{-1}:\g_{\mathcal{C}}\longrightarrow G\times_P\p$ with the unit bisection of $(T^*G)_{\p}\tto G\times_{P}\p$ gives the unit bisection of $(T^*G)_{\p}\tto\g_{\mathcal{C}}$. Straightforward computations also reveal that
	$$\sss_{\p}([[g_1:x_1]:[g_2:x_2]])=(\mathrm{Ad}_{g_1}(\p),\mathrm{Ad}_{g_1}(x_1))\quad\text{and}\quad\ttt_{\p}([[g_1:x_1]:[g_2:x_2]])=(\mathrm{Ad}_{g_2}(\p),\mathrm{Ad}_{g_2}(x_2))$$ for all $[[g_1:x_1]:[g_2:x_2]]\in(T^*G)_{\p}$.
	
	\begin{proposition}\label{Proposition: Second canonical}
		Let $\mathcal{C}$ be a conjugacy class of parabolic subalgebras of $\g$. If $\p_1,\p_2\in\mathcal{C}$, then there is a canonical $G$-equivariant isomorphism
		\begin{equation}\label{Equation: Diag}\begin{tikzcd}[column sep=30pt, row sep=40pt]
				(T^*G)_{\p_1} \ar[rr, "\Psi_{(\p_1,\p_2)}"] \ar[dr, shift left = 0.8ex, "\ttt_{\p_1}"]\ar[dr, shift right = 0.8ex, swap, "\sss_{\p_1}"] & & (T^*G)_{\p_2} \ar[dl, shift left = 0.8ex, "\ttt_{\p_2}"]\ar[dl, shift right = 0.8ex, swap, "\sss_{\p_2}"] \\
				& \g_{\mathcal{C}} &
		\end{tikzcd}\end{equation}
		of symplectic groupoids.
	\end{proposition}
	
	\begin{proof}
		Choose $h\in G$ satisfying $\mathrm{Ad}_h(\p_1)=\p_2$. Let $P_1,P_2\s G$ be the parabolic subgroups integrating $\mathfrak{p}_1,\mathfrak{p}_2$, respectively. It follows that $hP_1h^{-1}=P_2$, and that conjugation by $h$ descends to an algebraic group isomorphism $\phi:L(P_1)\overset{\cong}\longrightarrow L(P_2)$. Let us also recall that $G\times L(P_i)$ acts on $G/U(P_i)$ for $i=1,2$. In this context, we observe that the isomorphism
		$$\zeta:G/U(P_1)\overset{\cong}\longrightarrow G/U(P_2),\quad [g]\mapsto[gh^{-1}]$$
		satisfies the equivariance condition $$\zeta((g,[p])\cdot[k])=(g,\phi([p]))\cdot\zeta([k])$$ for all $(g,[p])\in G\times L(P_1)$ and $[k]\in G/U(P_1)$.
		On the other hand, the proof of Proposition \ref{Proposition: Canonical} explains that $$\eta:G\times_{U(P_1)}\p_1\longrightarrow G\times_{U(P_2)}\p_2,\quad [g:x]\mapsto [gh^{-1}:\mathrm{Ad}_h(x)],\quad [g:x]\in G\times_{U(P_1)}P_1$$ is the isomorphism of cotangent bundles induced by $\zeta$. We conclude that $\eta$ is an isomorphism from the Hamiltonian $(G\times L(P_1))$-variety $G\times_{U(P_1)}\p_1$ to the Hamiltonian $(G\times L(P_2))$-variety $G\times_{U(P_2)}\p_2$, where $L(P_1)$ and $L(P_2)$ are identified via $\phi$. This makes it clear that
		$$\Psi_{(\p_1,\p_2)}:(T^*G)_{\p_1}\longrightarrow (T^*G)_{\p_2},\quad [\alpha_1:\alpha_2]\mapsto [\eta(\alpha_1):\eta(\alpha_2)]$$ is a well-defined, $G$-equivariant symplectic variety isomorphism. It is also clear that \eqref{Equation: Diag} commutes if this isomorphism is taken as the top horizontal arrow, and straightforward to verify that $\Psi_{(\p_1,\p_2)}$ is an isomorphism of symplectic groupoids.
		
		It remains only to prove that $\Psi_{(\p_1,\p_2)}$ does not depend on the choice of $h\in G$ satisfying $\mathrm{Ad}_h(\p_1)=\p_2$. To this end, let $h,k\in G$ be such that $\mathrm{Ad}_h(\p_1)=\p_2$ and $\mathrm{Ad}_k(\p_1)=\p_2$. It follows that $\mathrm{Ad}_{hk^{-1}}(\p_2)=\p_2$, or equivalently that $hk^{-1}\in P_2$. We also have $$(gh^{-1},\mathrm{Ad}_h(x))=(gk^{-1}(hk^{-1})^{-1},\mathrm{Ad}_{hk^{-1}}\mathrm{Ad}_k(x))$$ for all $(g,x)\in G\times\mathfrak{p}_1$. In particular, $(gh^{-1},\mathrm{Ad}_h(x))$ and $(gk^{-1},\mathrm{Ad}_k(x))$ belong to the same $L(P_2)$-orbit in $G\times_{U(P_2)}\p_2$ for all $[g:x]\in G\times_{U(P_1)}\p_1$. This fact forces $\Psi_{(\p_1,\p_2)}$ to be independent of the choice mentioned in the first sentence of this paragraph.
	\end{proof}
	
	We are now equipped to prove Main Theorem \ref{1n9y2fsw}. For the sake of convenience, we restate it below.
	
	\begin{theorem*}
		Let $\mathcal{C}$ be a conjugacy class of parabolic subalgebras of $\mathfrak{g}$. There is a canonical algebraic symplectic groupoid $(T^*G)_{\mathcal{C}}\tto\g_{\mathcal{C}}$ that induces the Poisson structure on $\g_{\mathcal{C}}$.
	\end{theorem*}
	
	Let $\mathcal{C}$ be a conjugacy class of parabolic subalgebras of $\g$. Consider the set $$(T^*G)_{\mathcal{C}}\coloneqq\bigg(\bigsqcup_{\p\in\mathcal{C}}(T^*G)_{\p}\bigg)\Bigm/\sim,$$ where $\sim$ is the equivalence relation defined by $$(\alpha_1\in(T^*G)_{\p_1})\sim(\alpha_2\in(T^*G)_{\p_2})\Longleftrightarrow\alpha_2=\Psi_{(\p_1,\p_2)}(\alpha_1).$$ Main Theorem \ref{1n9y2fsw} now follows from our next result.
	
	\begin{corollary}\label{Corollary: Canonical}
		The set $(T^*G)_{\mathcal{C}}$ has a unique $G$-equivariant symplectic groupoid structure $(T^*G)_{\mathcal{C}}\xbigtoto[\ttt_{\mathcal{C}}]{\sss_{\mathcal{C}}}\g_{\mathcal{C}}$ such that $$\begin{tikzcd}[column sep=30pt, row sep=40pt]
			(T^*G)_{\p} \ar[rr] \ar[dr, shift left = 0.8ex, "\ttt_{\p}"]\ar[dr, shift right = 0.8ex, swap, "\sss_{\p}"] & & (T^*G)_{\mathcal{C}} \ar[dl, shift left = 0.8ex, "\ttt_{\mathcal{C}}"]\ar[dl, shift right = 0.8ex, swap, "\sss_{\mathcal{C}}"] \\
			& \g_{\mathcal{C}} &
		\end{tikzcd}$$
		is a $G$-equivariant symplectic groupoid isomorphism for all $\p\in\mathcal{C}$.
	\end{corollary}
	
	\begin{proof}
		This follows immediately from Proposition \ref{Proposition: Second canonical} and the definition of $(T^*G)_{\mathcal{C}}$.
	\end{proof}
	
	Given $(\mathfrak{p},x)\in\g_{\mathcal{C}}$, the isomorphism $(T^*G)_{\mathcal{C}}\cong (T^*G)_{\p}$ restricts to an isomorphism
	\begin{align}\label{Equation: Explicit descriptions}((T^*G)_{\mathcal{C}})_{(\mathfrak{p},x)} & \cong ((T^*G)_{\p})_{[e:x]} \\ & = \left\{[[g:y]:[h:z]]\in \left(G\times_{U(P)}\mathfrak{p})\times_{\mathfrak{l}(\p)}(G\times_{U(P)}\mathfrak{p})\right)/L(P):[g:y]=[e:x]=[h:z]\text{ in }G\times_P\mathfrak{p}\right\}\notag\end{align} between the isotropy groups of $(\mathfrak{p},x)$ and $[e:x]\in G\times_P\mathfrak{p}$.  We use this isomorphism to freely identify the two isotropy groups in our next proposition.
	
	\begin{proposition}\label{Proposition: Isotropy group}
		If $(\mathfrak{p},x)\in\g_{\mathcal{C}}$, then 
		$$L(P)_{[x]}\longrightarrow ((T^*G)_{\mathcal{C}})_{(\mathfrak{p},x)}=\left(\left(G\times_{U(P)}\mathfrak{p})\times_{\mathfrak{l}(\p)}(G\times_{U(P)}\mathfrak{p})\right)/L(P)\right)_{[e:x]},\quad [p]\mapsto [[p:\mathrm{Ad}_{p^{-1}}(x)]:[e:x]]$$
		is a well-defined isomorphism of algebraic groups.
	\end{proposition}
	
	\begin{proof}
		A straightforward exercise reveals that
		\begin{equation}\label{Equation: Group morphism} P_{[x]}\longrightarrow \left(\left(G\times_{U(P)}\mathfrak{p})\times_{\mathfrak{l}(\p)}(G\times_{U(P)}\mathfrak{p})\right)/L(P)\right)_{[e:x]},\quad p\mapsto [[p:\mathrm{Ad}_{p^{-1}}(x)]:[e:x]]\end{equation} defines an algebraic group morphism. It therefore suffices to prove that \eqref{Equation: Group morphism} is surjective with kernel $U(P)$. To establish surjectivity, suppose that $[[g:y]:[h:z]]\in\left(\left(G\times_{U(P)}\mathfrak{p})\times_{\mathfrak{l}(\p)}(G\times_{U(P)}\mathfrak{p})\right)/L(P)\right)_{[e:x]}$. It follows that $g,h\in P$, $y=\mathrm{Ad}_{g^{-1}}(x)$, and $z=\mathrm{Ad}_{h^{-1}}(x)$. The condition $[y]=[z]\in\mathfrak{l}(\p)$ then implies that $gh^{-1}\in P_{[x]}$, so that
		$$[[g:y]:[h:z]]=[[g:\mathrm{Ad}_{g^{-1}}(x)]:[h:\mathrm{Ad}_{h^{-1}}(x)]=[[gh^{-1}:\mathrm{Ad}_{(gh^{-1})^{-1}}(x)]:[e:x]]$$ must be in the image of \eqref{Equation: Group morphism}.
		
		It remains to prove that $U(P)$ is the kernel of \eqref{Equation: Group morphism}. To this end, note that $p\in P_{[x]}$ belongs to this kernel if and only if
		$[[p:\mathrm{Ad}_{p^{-1}}(x)]:[e:x]]=[[e:x]:[e:x]]$. This is equivalent to the existence of $q\in P$ such that
		$$[[p:\mathrm{Ad}_{p^{-1}}(x)]:[e:x]]=[[q:\mathrm{Ad}_{q^{-1}}(x)]:[q:\mathrm{Ad}_{q^{-1}}(x)]].$$ We may rephrase this as the condition that
		$$(p,\mathrm{Ad}_{p^{-1}}(x))=(qu,\mathrm{Ad}_{(qu)^{-1}}(x))\quad\text{and}\quad (e,x)=(qv,\mathrm{Ad}_{(qv)^{-1}}(x))$$ for some $q\in P$ and $u,v\in U(P)$. Straightforward manipulations and substitutions show that this holds if and only if $$(p,\mathrm{Ad}_{p^{-1}}(x))=(v^{-1}u,\mathrm{Ad}_{(v^{-1}u)^{-1}}(x))$$ for some $u,v\in U(P)$, or equivalently $p\in U(P)$. The kernel of \eqref{Equation: Group morphism} is therefore equal to $P_{[x]}\cap U(P)$. Since $U(P)$ acts trivially on the $P$-module $\mathfrak{l}(\p)$, the kernel must be $U(P)$.
	\end{proof}
	
	\subsection{Global slices to $(T^*G)_{\mathcal{C}}$}\label{Subsection: Global slices}
	Let $\mathcal{C}$ be a conjugacy class of parabolic subalgebras of $\g$. Fix a principal $\mathfrak{sl}_2$-triple $(e,h,f)\in\g^{\times 3}$, and consider the associated \textit{Kostant slice}
	$$\mathrm{Kos}\coloneqq e+\g_f\s\g.$$ One knows that $\mathrm{Kos}$ is a fundamental domain for the adjoint action of $G$ on $\greg$ \cite{kostant}. It is also known to be a Poisson transversal for the Poisson structure on $\g=\g^*$, where the Killing form is used to identify $\g$ and $\g^*$; see \cite{kostant-american,kostant} for the required ingredients, and \cite{gan-ginzburg} for the case of arbitrary Slodowy slices. We conclude that 
	$$\mathrm{Kos}_{\mathcal{C}}\coloneqq\mu_{\mathcal{C}}^{-1}(\mathrm{Kos})\s\g_{\mathcal{C}}$$ is a Poisson transversal. On the other hand, let $((T^*G)_{\mathcal{C}})_{\text{reg}}\tto(\g_{\mathcal{C}})_{\text{reg}}$ denote the pullback of $(T^*G)_{\mathcal{C}}\tto\g_{\mathcal{C}}$ to $(\g_{\mathcal{C}})_{\text{reg}}\s\g_{\mathcal{C}}$. We relate the preceding discussion to the notion of an \textit{admissible global slice}, as defined in Subsection \ref{Subsection: TQFTs}.  
	
	\begin{theorem}\label{Theorem: Leaf intersection}
		The subvariety $\mathrm{Kos}_{\mathcal{C}}$ is an admissible global slice to $((T^*G)_{\mathcal{C}})_{\emph{reg}}\tto(\g_{\mathcal{C}})_{\emph{reg}}$. 
	\end{theorem}
	
	\begin{proof}
		We first show $\mathrm{Kos}_{\mathcal{C}}$ to be a global slice to $((T^*G)_{\mathcal{C}})_{\text{reg}}\tto(\g_{\mathcal{C}})_{\text{reg}}$. To this end, note that the orbits of $((T^*G)_{\mathcal{C}})_{\text{reg}}\tto(\g_{\mathcal{C}})_{\text{reg}}$ are precisely the top-dimensional symplectic leaves of $\g_{\mathcal{C}}$. This combines with Proposition \ref{Proposition: Nice} to reduce us to proving the following: $\mathrm{Kos}_{\mathcal{C}}$ has a non-empty intersection with the symplectic leaf of $\mathfrak{g}_{\mathcal{C}}$ through $(\p,x)$ if and only if $[x]\in\mathfrak{l}(\p)_{\text{reg}}$, in which case the intersection is a singleton. By Lemma \ref{Lemma: Single orbit} and Proposition \ref{Proposition: Canonical}, this is equivalent to proving the following: $\mu_{\p}^{-1}(\mathrm{Kos})$ has a non-empty intersection with the symplectic leaf of $G\times_P\p$ through $[e:x]$ if and only if $(x+\mathfrak{u}(\p))\cap\greg\neq\emptyset$, in which case the intersection is a singleton. This leaf is $G\times_{P_{[x]}}(x+\mathfrak{u}(\p))$, as established in the proof of Proposition \ref{Proposition: Nice}. The ``if and only if" assertion now follows immediately from the fact that $\mathrm{Kos}$ is a fundamental domain for the adjoint action of $G$ on $\greg$.
		
		To complete our proof that $\mathrm{Kos}_{\mathcal{C}}$ is a global slice, let $(\p,x)\in\g_{\mathcal{C}}$ be such that $(x+\mathfrak{u}(\p))\cap\greg\neq\emptyset$. Our task is to show that any two elements of $G\times_{P_{[x]}}(x+\mathfrak{u}(\p))\cap\mu_{\p}^{-1}(\mathrm{Kos})$ must coincide. To this end, suppose that $[g,y],[h:z]\in G\times_P(x+\mathfrak{u}(\p))\cap\mu_{\p}^{-1}(\mathrm{Kos})$. It follows that $\mathrm{Ad}_g(y),\mathrm{Ad}_h(z)\in\mathrm{Kos}$, implying that $y,z\in\greg$. By Lemma \ref{Lemma: Single orbit}(i), $z=\mathrm{Ad}_p(y)$ for some $p\in P$. One consequence is that $y$ and $z$ must be conjugate to the same element of $\mathrm{Kos}$, i.e. $\mathrm{Ad}_g(y)=\mathrm{Ad}_h(z)=\mathrm{Ad}_{hp}(y)$. We conclude that $g^{-1}hp\in G_y$. Since Lemma \ref{Lemma: In Borel} now implies that $g^{-1}hp\in P$, we have
		$$[h:z]=[h:\mathrm{Ad}_p(y)]=[hp:y]=[hp(g^{-1}hp)^{-1}:\mathrm{Ad}_{g^{-1}hp}(y)]=[g:y].$$ This establishes that $\mathrm{Kos}_{\mathcal{C}}$ is a global slice to $((T^*G)_{\mathcal{C}})_{\text{reg}}\tto(\g_{\mathcal{C}})_{\text{reg}}$.
		
		It remains only to establish that the global slice $\mathrm{Kos}_{\mathcal{C}}$ is admissible. Given $(\p,x)\in\mathrm{Kos}_{\mathcal{C}}$, the previous two paragraphs imply that $[x]\in\mathfrak{l}(\p)_{\text{reg}}$. We conclude that $L(P)_{[x]}$ is abelian \cite[Proposition 14]{kostant}. By Proposition \ref{Proposition: Isotropy group}, the isotropy group $((T^*G)_{\mathcal{C}})_{(\p,x)}$ is also abelian. It follows that $\mathrm{Kos}_{\mathcal{C}}$ is admissible. 
	\end{proof}
	
	We have the following specialization to the case of the flag variety $\mathcal{B}$ of $\g$. It is presumably well-known to experts. Note that $\g_{\mathcal{B}}=\widetilde{\g}$ is the full Grothendieck--Springer resolution in this case. 
	
	\begin{corollary}\label{Corollary: Borel}
		The subvariety $\mathrm{Kos}_{\mathcal{B}}$ is an admissible global slice to $(T^*G)_{\mathcal{B}}\tto\g_{\mathcal{B}}$.
	\end{corollary}
	
	\begin{proof}
		This follows immediately from Corollary \ref{Corollary: Easy} and Theorem \ref{Theorem: Leaf intersection}.
	\end{proof}
	
	\subsection{Grothendieck--Springer alterations of the Moore--Tachikawa TQFT}\label{Subsection: Alterations}
	
	Recall that Main Theorem \ref{612eq1e1} is the following assertion.
	
	\begin{theorem*}
		Let $\mathcal{C}$ be a conjugacy class of parabolic subalgebras of $\g$. There is an explicit TQFT $\eta : \Cob_2\longrightarrow\WS$ satisfying $\eta(S^1) = ((T^*G)_{\mathcal{C}})\reg$ and  $\eta(C_{m,n})=[(T^*G)_{\mathcal{C}}^{m,n}]$ for all $(m,n)\neq (0,0)$, where $C_{m,n}$ denotes the genus-0 cobordism from $m$ circles to $n$ circles.
	\end{theorem*}
	
	The proof is now straightforward. Let $\mathcal{C}$ be a conjugacy class of parabolic subalgebras of $\g$. As $(\g_{\mathcal{C}})_{\text{reg}}$ is Poisson, Theorem \ref{Theorem: Leaf intersection} implies that $\G=((T^*G)_{\mathcal{C}})_{\text{reg}}$ and $S=\mathrm{Kos}_{\mathcal{C}}$ satisfy the hypotheses of Theorem \ref{Theorem: TQFTs from quasi-symplectic groupoids}. Write $\eta_{\mathcal{C}}:\Cob_2\longrightarrow\WS$ for the resulting TQFT; it satisfies Main Theorem \ref{612eq1e1}. Proposition \ref{Proposition: Explicit} gives an explicit variety
	$(T^*G)_{\mathcal{C}}^{m,n}\coloneqq \eta_{\mathcal{C}}(C_{m,n})$ for $(m,n)\neq (0,0)$.
	The following two special cases warrant further discussion. 
	
	\begin{itemize}
		\item[\textup{(i)}] If $\mathcal{C}=\{\g\}$, then $\eta_{\mathcal{C}}$ is the \textit{open Moore--Tachikawa TQFT} constructed in \cite{cro-may2:24}. The varieties $$(T^*G)^{m,n}\coloneqq (T^*G)_{\mathcal{C}}^{m,n}$$ are sometimes called \textit{open Moore--Tachikawa varieties}, and feature in several recent works \cite{crooks-mayrand,bie:23,gin-kaz:23,cro-may2:24}. One may affinize these varieties to obtain a scheme-theoretic version of the Moore--Tachikawa TQFT \cite{gin-kaz:23,cro-may2:24}.  
		\item[\textup{(ii)}] Suppose that $\mathcal{C}=\mathcal{B}$ is the conjugacy class of Borel subalgebras. By Corollary \ref{Corollary: Borel}, $\mathrm{Kos}_{\mathcal{C}}=\mathrm{Kos}_{\mathcal{B}}$ is an admissible global slice to the entire symplectic groupoid $(T^*G)_{\mathcal{C}}=(T^*G)_{\mathcal{B}}$.
	\end{itemize}
	
	Fix a pair of non-negative integers $(m,n)\neq (0,0)$. One might seek relationships among the varieties $(T^*G)_{\mathcal{C}}^{m,n}$, as $\mathcal{C}$ ranges over the conjugacy classes of parabolic subalgebras of $\g$. Section \ref{Section: Main} addresses this issue. 
	
	\section{Lagrangian relations and TQFTs}\label{Section: Main}
	
	We now prove Main Theorem \ref{Theorem: Main Theorem 3}. In the interest of clarity, we restate this theorem below.
	
	\begin{theorem*}
		Fix $(m,n)\in(\mathbb{Z}_{\geq 0})^2$ with $(m,n)\neq (0,0)$, and let $\mathcal{C}$ be a conjugacy class of parabolic subalgebras of $\mathfrak{g}$. There is an explicit Lagrangian relation from $(T^*G)_{\mathcal{C}}^{m,n}$ to $(T^*G)^{m,n}$.
	\end{theorem*}
	
	A preparatory result is derived in Subsection \ref{Subsection: A useful result}. The proof of Main Theorem \ref{Theorem: Main Theorem 3} constitutes Subsection \ref{Subsection: Proof}. 
	
	\subsection{A useful result}\label{Subsection: A useful result} Let $\mathcal{C}$ be a conjugacy class of parabolic subalgebras of $\g$. Suppose that $\p\in\mathcal{C}$, and let $P\s G$ be the parabolic subgroup integrating $\p$. Recall from Subsection \ref{Subsection: The variety} that $G\times_{U(P)}\p$ is a symplectic Hamiltonian $G\times L(P)$-variety. Let $\phi_{\p}:G\times_{U(P)}\p\longrightarrow\g$ denote the moment map for the Hamiltonian action of $G=G\times\{e\}\s G\times L(P)$. On the other hand, recall that $G\times_P\p=(G\times_{U(P)}\p)/L(P)$ is a Poisson Hamiltonian $G$-variety with moment map $\mu_{\p}$. One has $\mathcal{L}_{\mu_{\p}}=\mathcal{L}_{\overline{\phi_{\p}}}$ in the notation of Subsection \ref{Subsection: Isotropic}. Note that $\mathcal{L}_{\mu_{\p}}$ is a Lagrangian subgroupoid of $(\mathrm{Pair}(G\times_{U(P)}\p)\sll{0}L(P))\times\overline{T^*G}=(T^*G)_{\p}\times\overline{T^*G}$, and that Lemma \ref{jmzayi83} gives information about Lagrangian intersections with $\mathcal{L}_{\mu_{\p}}$. 
	
	Use the isomorphism in Corollary \ref{Corollary: Canonical} to identify $(T^*G)_{\mathcal{C}}$ with $(T^*G)_{\p}$, and Proposition \ref{Proposition: Canonical}(i) to identify $\g_{\mathcal{C}}$ with $G\times_P\p$. Lemma \ref{jmzayi83} then implies the following about $\mathcal{L}_{\mu_{\mathcal{C}}}$.
	
	\begin{proposition}\label{Proposition: Lagrangian subvariety}
		Let $\sss:T^*G\longrightarrow\g$ and $\ttt:T^*G\longrightarrow\g$ denote the source and target of $T^*G$, respectively. The intersections $(\mathcal{L}_{\mu_{\mathcal{C}}})\cap(\sss_{\mathcal{C}}^{-1}(\mathrm{Kos}_{\mathcal{C}})\times\sss^{-1}(\mathrm{Kos}))$ and $(\mathcal{L}_{\mu_{\mathcal{C}}})\cap(\ttt_{\mathcal{C}}^{-1}(\mathrm{Kos}_{\mathcal{C}})\times\ttt^{-1}(\mathrm{Kos}))$ are Lagrangian subvarieties of $\sss_{\mathcal{C}}^{-1}(\mathrm{Kos}_{\mathcal{C}})\times\overline{\sss^{-1}(\mathrm{Kos})}$ and $\ttt_{\mathcal{C}}^{-1}(\mathrm{Kos}_{\mathcal{C}})\times\overline{\ttt^{-1}(\mathrm{Kos})}$, respectively.
	\end{proposition}
	
	\subsection{Proof of Main Theorem \ref{Theorem: Main Theorem 3}}\label{Subsection: Proof} 
	To prepare for what follows, let $X$ and $Y$ be symplectic varieties. Our convention is to define a \textit{Lagrangian} relation from $X$ to $Y$ to be an immersed Lagrangian submanifold $L\s X\times\overline{Y}$. We sometimes adopt the notation $L:X\Longrightarrow Y$ for a Lagrangian relation $L$ from $X$ to $Y$. If $L:X\Longrightarrow Y$ and $M:Y\Longrightarrow Z$ are Lagrangian relations between symplectic varieties $X$, $Y$, and $Z$, then $L$ and $M$ may be composed as set-theoretic relations. Let $\Delta(Y)\s Y\times Y$ denote the diagonal copy of $Y$. We call $M$ and $N$ \textit{composable} if $L\times M$ are $X\times\Delta(Y)\times Z$ are transverse in $X\times Y\times Y\times Z$, and the projection map $X\times Y\times Y\times Z\longrightarrow X\times Z$ restricts to an immersion $(L\times M)\cap (X\times\Delta(Y)\times Z)\longrightarrow X\times Z$. The image of this projection is then Lagrangian relation $M\circ L:X\Longrightarrow Z$ \cite[Proposition 5.28]{Bates}. We call $L$ and $M$ \textit{strongly composable} if the map $(L\times M)\cap (X\times\Delta(Y)\times Z)\longrightarrow X\times Z$ is an embedding of manifolds. In this case, $M\circ L$ is a Lagrangian submanifold of $X\times\overline{Y}$.
	
	We now return to the matter at hand. Let $\mathcal{C}$ be a conjugacy class of parabolic subalgebras of $\g$. Recall that $\mathrm{Kos}_{\mathcal{C}}\coloneqq\mu_{\mathcal{C}}^{-1}(\mathrm{Kos})$ for $\mathrm{Kos}\coloneqq e+\g_f\s\g$ the Kostant slice associated to a principal $\mathfrak{sl}_2$-triple $(e,h,f)\in\g^{\times 3}$. We may specialize the end of Subsection \ref{Subsection: TQFTs} to $\G=(T^*G)_{\mathcal{C}}$ and $S=\mathrm{Kos}_{\mathcal{C}}$. Given $m,n\in\mathbb{Z}_{\geq 0}$ with $(m,n)\neq (0,0)$, this specialization yields a pre-Poisson subvariety
	$$\mathrm{Kos}_{\mathcal{C}}^{m,n}\coloneqq\{(\alpha_1,\ldots,\alpha_{m+n},\beta_1,\ldots,\beta_{m+n})\in\mathfrak{g}_{\mathcal{C}}^{m+n}\times\mathfrak{g}_{\mathcal{C}}^{m+n}:\alpha_{n+1}=\cdots=\alpha_{m+n}=\beta_1=\cdots=\beta_n\in\mathrm{Kos}_{\mathcal{C}}\}$$ and abelian stabilizer subgroupoid $\mathcal{H}_{\mathrm{Kos}_{\mathcal{C}}^{m,n}}\longrightarrow\mathrm{Kos}_{\mathcal{C}}^{m,n}$. On the other hand, consider the TQFT $\eta_{\mathcal{C}}:\Cob_2\longrightarrow\WS$ introduced in Subsection \ref{Subsection: Alterations}, and set $(T^*G)_{\mathcal{C}}^{m,n}\coloneqq \eta_{\mathcal{C}}(C_{m,n})$ for $m,n\in\mathbb{Z}_{\geq 0}$ with $(m,n)\neq (0,0)$. The end of Subsection \ref{Subsection: TQFTs} implies that
	$$(T^*G)_{\mathcal{C}}^{m,n}\coloneqq((T^*G)_{\mathcal{C}})_{\mathrm{Kos}_{\mathcal{C}}}^{m,n}=(T^*G)_{\mathcal{C}}^{m+n}\sll{\mathrm{Kos}_{\mathcal{C}}^{m,n},\mathcal{H}_{\mathrm{Kos}_{\mathcal{C}}^{m,n}}}\left((T^*G)_{\mathcal{C}}^{m+n}\times\overline{(T^*G)_{\mathcal{C}}^{m+n}}\right).$$	
	
	Now observe that $\mathrm{Kos}_{\mathcal{C}}^{m,n}$ is coisotropic in the Poisson transversal $\g_{\mathcal{C}}^n\times\mathrm{Kos}_{\mathcal{C}}^{m+n}\times\g_{\mathcal{C}}^m\s\g_{\mathcal{C}}^{m+n}\times\overline{\g_{\mathcal{C}}^{m+n}}$. The preimage of this Poisson transversal under $$\nu_{\mathcal{C}}^{m,n}:(\underbrace{\sss_{\mathcal{C}},\ldots,\sss_{\mathcal{C}}}_{m+n\text{ times}},\underbrace{\ttt_{\mathcal{C}},\ldots,\ttt_{\mathcal{C}}}_{m+n\text{ times}}):(T^*G)_{\mathcal{C}}^{m+n}\longrightarrow\g_{\mathcal{C}}^{m+n}\times\overline{\g_{\mathcal{C}}^{m+n}}$$ is $((T^*G)_{\mathcal{C}}^{0,1})^n\times((T^*G)_{\mathcal{C}}^{1,0})^m\s(T^*G)_{\mathcal{C}}^{m+n}$. The previous two sentences combine with Lemma \ref{Lemma: Lagrangian correspondence}(iii) to imply that
	$$\Lambda_{\mathcal{C}}^{m,n}\coloneqq\{([(\alpha_1,\ldots,\alpha_{m+n})],(\alpha_1,\ldots,\alpha_{m+n})):(\alpha_1,\ldots,\alpha_{m+n})\in(\nu_{\mathcal{C}}^{m,n})^{-1}(\mathrm{Kos}_{\mathcal{C}}^{m,n})\}$$ is a Lagrangian subvariety of $(T^*G)_{\mathcal{C}}^{m,n}\times\overline{((T^*G)_{\mathcal{C}}^{0,1})^n\times((T^*G)_{\mathcal{C}}^{1,0})^m}$, i.e.
	$$\Lambda_{\mathcal{C}}^{m,n}:(T^*G)_{\mathcal{C}}^{m,n}\Longrightarrow ((T^*G)_{\mathcal{C}}^{0,1})^n\times((T^*G)_{\mathcal{C}}^{1,0})^m.$$
	On the other hand, Proposition \ref{Proposition: Lagrangian subvariety} implies that $$\Gamma_{\mathcal{C}}^{m,n}\coloneqq\left\{(\alpha_1,\ldots,\alpha_{m+n},\beta_1,\ldots,\beta_{m+n}):\substack{(\alpha_j,\beta_j)\in((T^*G)_{\mathcal{C}}^{0,1}\times(T^*G)^{0,1})\cap\mathcal{L}_{\mu_{\mathcal{C}}}\text{ for all }j\in\{1,\ldots,n\}\\ (\alpha_k,\beta_k)\in((T^*G)_{\mathcal{C}}^{1,0}\times(T^*G)^{1,0})\cap\mathcal{L}_{\mu_{\mathcal{C}}}\text{ for all }k\in\{n+1,\ldots,m+n\}}\right\}$$ is a Lagrangian subvariety of $((T^*G)_{\mathcal{C}}^{0,1})^n\times((T^*G)_{\mathcal{C}}^{1,0})^m\times\overline{((T^*G)^{0,1})^n\times((T^*G)^{1,0})^m}$, i.e.
	$$\Gamma_{\mathcal{C}}^{m,n}:((T^*G)_{\mathcal{C}}^{0,1})^n\times((T^*G)_{\mathcal{C}}^{1,0})^m\Longrightarrow ((T^*G)^{0,1})^n\times((T^*G)^{1,0})^m.$$
	The following result shows that these two Lagrangian relations are strongly composable.
	
	\begin{theorem}\label{Theorem: Technical}
		Suppose that $m,n\in\mathbb{Z}_{\geq 0}$ and $(m,n)\neq (0,0)$.
		\begin{itemize}
			\item[\textup{(i)}] The subvarieties $$\Lambda_{\mathcal{C}}^{m,n}\times\Gamma_{\mathcal{C}}^{m,n}$$ and $$(T^*G)_{\mathcal{C}}^{m,n}\times\Delta\left(((T^*G)_{\mathcal{C}}^{0,1})^n\times((T^*G)_{\mathcal{C}}^{1,0})^m\right)\times((T^*G)^{0,1})^n\times((T^*G)^{1,0})^m$$ are transverse in $$(T^*G)_{\mathcal{C}}^{m,n}\times((T^*G)_{\mathcal{C}}^{0,1})^n\times((T^*G)_{\mathcal{C}}^{1,0})^m\times((T^*G)_{\mathcal{C}}^{0,1})^n\times((T^*G)_{\mathcal{C}}^{1,0})^m\times((T^*G)^{0,1})^n\times((T^*G)^{1,0})^m.$$
			\item[\textup{(ii)}] The projection of $$(T^*G)_{\mathcal{C}}^{m,n}\times((T^*G)_{\mathcal{C}}^{0,1})^n\times((T^*G)_{\mathcal{C}}^{1,0})^m\times((T^*G)_{\mathcal{C}}^{0,1})^n\times((T^*G)_{\mathcal{C}}^{1,0})^m\times((T^*G)^{0,1})^n\times((T^*G)^{1,0})^m$$ to the first and last two factors $$(T^*G)_{\mathcal{C}}^{m,n}\times((T^*G)^{0,1})^n\times((T^*G)^{1,0})^m$$ restricts to an algebraic, locally closed embedding of $$(\Lambda_{\mathcal{C}}^{m,n}\times\Gamma_{\mathcal{C}}^{m,n})\cap((T^*G)_{\mathcal{C}}^{m,n}\times\Delta(((T^*G)_{\mathcal{C}}^{0,1})^n\times((T^*G)_{\mathcal{C}}^{1,0})^m)\times((T^*G)^{0,1})^n\times((T^*G)^{1,0})^m)$$ into $(T^*G)_{\mathcal{C}}^{m,n}\times((T^*G)^{0,1})^n\times((T^*G)^{1,0})^m$. 
		\end{itemize}
	\end{theorem}
	
	\begin{proof}
		As this proof is technical, we begin with an overview of its structure. Our approach is to choose $\p\in\mathcal{C}$. This allows us to identify $\g_{\mathcal{C}}$ and $(T^*G)_{\mathcal{C}}$ with $G\times_P\p$ and the pair groupoid $$(T^*G)_{\p}=\mathrm{Pair}(G\times_{U(P)}\p)\sll{0}L(P)\tto G\times_P\p,$$ respectively, where $P\subset G$ is the parabolic subgroup integrating $\p\subset\g$. These identifications yield concrete descriptions of $\Lambda_{\mathcal{C}}^{m,n}$, $\Gamma_{\mathcal{C}}^{m,n}$, $(T^*G)_{\mathcal{C}}^{m,n}$, $(T^*G)_{\mathcal{C}}^{1,0}$, and $(T^*G)_{\mathcal{C}}^{0,1}$. As a consequence, we obtain a concrete description of points in the intersection of the subvarieties in \textup{(i)}; this intersection is denoted by $X$ in our proof. We then find a sufficient linear-algebraic condition for the subvarieties in \textup{(i)} to be transverse at points in $X$: it suffices for the vector subspaces $T_{\alpha}(\nu^{m,n})^{-1}(\mathrm{Kos}_{\mathcal{C}}^{m,n})\oplus V$ and $T_{(\alpha,\alpha)}\Delta(((T^*G)_{\mathcal{C}}^{0,1})^n\times((T^*G)_{\mathcal{C}}^{1,0})^m)$ to be transverse in $T_{(\alpha,\alpha)}((((T^*G)_{\mathcal{C}}^{0,1})^n\times((T^*G)_{\mathcal{C}}^{1,0})^m)\times(((T^*G)_{\mathcal{C}}^{0,1})^n\times((T^*G)_{\mathcal{C}}^{1,0})^m))$, where precise definitions of $\alpha$ and these vector spaces are given in the body of the proof. By combining Equations \eqref{Equation: Intersection} and \eqref{Equation: Kernel} with dimension counts, we show that this sufficient condition holds if and only if $\dim V=(m+n)\dim\g$. This second condition is easily established, verifying \textup{(i)}.
		
		The proof of \textup{(ii)} also uses the identifications of $\g_{\mathcal{C}}$ and $(T^*G)_{\mathcal{C}}$ with $G\times_P\p$ and $(T^*G)_{\p}$, respectively. We first show that the projection indicated in (ii) defines a topological embedding of $X$ into the codomain of that projection. This embedding is subsequently shown to be one of smooth manifolds. Further arguments are then given to establish that this is a locally closed embedding of algebraic varieties.
		
		We begin by proving (i). A point in the intersection $$X\coloneqq(\Lambda_{\mathcal{C}}^{m,n}\times\Gamma_{\mathcal{C}}^{m,n})\cap((T^*G)_{\mathcal{C}}^{m,n}\times\Delta(((T^*G)_{\mathcal{C}}^{0,1})^n\times((T^*G)_{\mathcal{C}}^{1,0})^m)\times((T^*G)^{0,1})^n\times((T^*G)^{1,0})^m)$$ must take the form $([\alpha],\alpha,\alpha,\beta)$ for some $\alpha=(\alpha_1,\ldots,\alpha_{m+n})\in(\nu_{\mathcal{C}}^{m,n})^{-1}(\mathrm{Kos}_{\mathcal{C}}^{m,n})$ and $\beta=(\beta_1,\ldots\beta_{m+n})\in((T^*G)^{0,1})^n\times((T^*G)^{1,0})^m$ satisfying $(\alpha_j,\beta_j)\in((T^*G)_{\mathcal{C}}^{0,1}\times(T^*G)^{0,1})\cap\mathcal{L}_{\mu_{\mathcal{C}}}\text{ for all }j\in\{1,\ldots,n\}$ and $(\alpha_k,\beta_k)\in((T^*G)_{\mathcal{C}}^{1,0}\times(T^*G)^{1,0})\cap\mathcal{L}_{\mu_{\mathcal{C}}}\text{ for all }k\in\{n+1,\ldots,m+n\}$. 
		
		Choose $\p\in\mathcal{C}$. Use Proposition \ref{Proposition: Canonical} and Corollary \ref{Corollary: Canonical} to freely identify the Poisson Hamiltonian $G$-variety $\g_{\mathcal{C}}$ and $G$-equivariant symplectic groupoid $(T^*G)_{\mathcal{C}}$ with $G\times_P\p$ and $(T^*G)_{\p}$, respectively. It follows that $\mathrm{Kos}_{\mathcal{C}}=\mu_{\p}^{-1}(\mathrm{Kos})$. We may find $(g_i,[h_i:x_i])\in G\times\mathrm{Kos}_{\mathcal{C}}$ for $i\in\{1,\ldots,m+n\}$ such that $$(\alpha_j,\beta_j)=([[g_jh_j:x_j]:[h_j:x_j]],(g_j,\mathrm{Ad}_{h_j}(x_j)))\text{ for }j\in\{1,\ldots,n\}$$
		and
		$$(\alpha_k,\beta_k)=([[h_k:x_k]:[g_k^{-1}h_k:x_k]],(g_k,\mathrm{Ad}_{g_k^{-1}h_k}(x_k)))\text{ for }k\in\{n+1,\ldots,m+n\}.$$ The condition $\alpha\in(\nu^{m,n})^{-1}(\mathrm{Kos}_{\mathcal{C}}^{m,n})$ then yields $[h_1:x_1]=\cdots=[h_{m+n}:x_{m+n}]\in\mathrm{Kos}_{\mathcal{C}}$. Denote this uniform element of $\mathrm{Kos}_{\mathcal{C}}$ by $\gamma$. We have 
		\begin{align*}T_{\alpha}(\nu^{m,n})^{-1}(\mathrm{Kos}_{\mathcal{C}}^{m,n}) & = (\mathrm{d}\nu^{m,n}_{\alpha})^{-1}(T_{\nu^{m,n}(\alpha)}\mathrm{Kos}_{\mathcal{C}}^{m,n})\\
			& = \{(v_1,\ldots,v_{m+n})\in T_{\alpha}(((T^*G)_{\mathcal{C}}^{0,1})^n\times((T^*G)_{\mathcal{C}}^{1,0})^m):\substack{(\mathrm{d}\sss_{\mathcal{C}})_{\alpha_{n+1}}(v_{n+1})=\cdots=(\mathrm{d}\sss_{\mathcal{C}})_{\alpha_{m+n}}(v_{m+n})\\=(\mathrm{d}\ttt_{\mathcal{C}})_{\alpha_{1}}(v_{1})=\cdots=(\mathrm{d}\ttt_{\mathcal{C}})_{\alpha_{n}}(v_{n})\in T_{\gamma}\mathrm{Kos}_{\mathcal{C}}}\}.
		\end{align*}
		
		Write $\mathcal{A}_g:G\times_P\p\longrightarrow G\times_P\p$ for the action of $g\in G$ on $G\times_P\p$. Consider the subspaces
		$$V_j\coloneqq\{[((\mathrm{d}\mathcal{A}_{g_j})_{[h_j:x_j]}(v)-(X_{\mathrm{Ad}_{g_j}(\xi)})_{[g_jh_j:x_j]},v)]:\xi\in\g,\text{ }[v]\in T_{\gamma}\mathrm{Kos}_{\mathcal{C}}\}\s T_{\alpha_j}(T^*G)_{\mathcal{C}}^{0,1}\text{ for }j\in\{1,\ldots,n\},$$
		$$V_k\coloneqq\{[(v,(X_{\xi})_{[g_k^{-1}h_k:x_k]}+(\mathrm{d}\mathcal{A}_{g_k^{-1}})_{[h_k:x_k]}(v))]:\xi\in\g,\text{ }[v]\in T_{\gamma}\mathrm{Kos}_{\mathcal{C}}\}\s T_{\alpha_k}(T^*G)_{\mathcal{C}}^{1,0}\text{ for }k\in\{n+1,\ldots,m+n\},$$ and $$V\coloneqq V_1\oplus\cdots\oplus V_{m+n}\s T_{\alpha}(((T^*G)_{\mathcal{C}}^{0,1})^n\times((T^*G)_{\mathcal{C}}^{1,0})^m).$$ Our objective is to prove that $T_{\alpha}(\nu^{m,n})^{-1}(\mathrm{Kos}_{\mathcal{C}}^{m,n})\oplus V$ and $T_{(\alpha,\alpha)}\Delta(((T^*G)_{\mathcal{C}}^{0,1})^n\times((T^*G)_{\mathcal{C}}^{1,0})^m)$ are transverse in $T_{(\alpha,\alpha)}((((T^*G)_{\mathcal{C}}^{0,1})^n\times((T^*G)_{\mathcal{C}}^{1,0})^m)\times(((T^*G)_{\mathcal{C}}^{0,1})^n\times((T^*G)_{\mathcal{C}}^{1,0})^m))$.
		
		Define a linear map $\phi:\g^{m+n}\oplus T_{\gamma}\mathrm{Kos}_{\mathcal{C}}\longrightarrow V$ as follows: the projection of $\phi(\xi_1,\ldots,\xi_{m+n},[v])\in V$ onto $V_j$ is $$[((\mathrm{d}\mathcal{A}_{g_j})_{[h_i:x_i]}(v)-(X_{\mathrm{Ad}_{g_j}(\xi_j)})_{[g_jh_j:x_j]},v)]$$ for $j\in\{1,\ldots,n\}$, and the projection to $V_k$ is $$[(v,(X_{\xi})_{[g_k^{-1}h_k:x_k]}+(\mathrm{d}\mathcal{A}_{g_k^{-1}})_{[h_k:x_k]}(v))]$$ for $k\in\{n+1,\ldots,m+n\}$. A straightforward exercise reveals that \begin{equation}\label{Equation: Intersection}(T_{\alpha}(\nu^{m,n})^{-1}(\mathrm{Kos}_{\mathcal{C}}^{m,n})\oplus V)\cap(T_{(\alpha,\alpha)}\Delta(((T^*G)_{\mathcal{C}}^{0,1})^n\times((T^*G)_{\mathcal{C}}^{1,0})^m))=\Delta(\mathrm{im}\hspace{2pt}\phi).\end{equation}
		
		We claim that \begin{equation}\label{Equation: Kernel}\mathrm{ker}\hspace{2pt}\phi=\g_{\gamma}\oplus\cdots\oplus\g_{\gamma}\oplus\mathrm{Ad}_{g_{n+1}^{-1}}(\g_{\gamma})\oplus\cdots\oplus\mathrm{Ad}_{g_{m+n}^{-1}}(\g_{\gamma})\oplus\{0\},\end{equation} where $\g_{\gamma}\s\g$ is the Lie algebra of $G_{\gamma}\s G$. To this end, suppose that $(\xi_1,\ldots,\xi_{m+n},[v])\in\g^{m+n}\oplus T_{\gamma}\mathrm{Kos}_{\mathcal{C}}$. This vector belongs to $\mathrm{ker}\hspace{2pt}\phi$ if and only if $((\mathrm{d}\mathcal{A}_{g_j})_{[h_j:x_k]}(v)-(X_{\mathrm{Ad}_{g_j}(\xi_j)})_{[g_jh_j:x_j]},v)$ is tangent to the diagonal $L(P)$-orbit of $([g_jh_j:x_j],[h_j:x_j])\in (G\times_{U(P)}\p)^2$ for all $j\in\{1,\ldots,n\}$, and $(v,(X_{\xi})_{[g_k^{-1}h_k:x_k]}+(\mathrm{d}\mathcal{A}_{g_k^{-1}})_{[h_k:x_k]}(v))$ is tangent to the diagonal $L(P)$-orbit of $([h_k:x_k],[g_k^{-1}h_k:x_k])\in (G\times_{U(P)}\p)^2$ for all $k\in\{n+1,\ldots,m+n\}$. One necessary condition for this to occur is that $[v]=0$. Observe that $(\mathrm{d}\mathcal{A}_{g_j})_{[h_j:x_j]}(v)$ and $(\mathrm{d}\mathcal{A}_{g_k^{-1}})_{[h_k:x_k]}(v)$ are then necessarily tangent to the $L(P)$-orbits of $[g_jh_j:x_j]\in G\times_{U(P)}\mathfrak{p}$ and $[g_k^{-1}h_k:x_k]\in G\times_{U(P)}\mathfrak{p}$, respectively, for all $j\in\{1,\ldots,n\}$ and $k\in\{n+1,\ldots,m+n\}$. We conclude that $(\xi_1,\ldots,\xi_{m+n},[v])\in\mathrm{ker}\hspace{2pt}\phi$ if and only if $[v]=0$, $(X_{\mathrm{Ad}_{g_j}(\xi_j)})_{[g_jh_j:x_j]}$ is tangent to the $L(P)$-orbit of $[g_jh_j:x_j]\in G\times_{U(P)}\p$ for all $j\in\{1,\ldots,n\}$, and $(X_{\xi_k})_{[g_k^{-1}h_k:x_k]}$ is tangent to the $L(P)$-orbit of $[g_k^{-1}h_k:x_k]\in G\times_{U(P)}\p$ for all $k\in\{n+1,\ldots,m+n\}$. This holds if and only if $[v]=0$, $(X_{\mathrm{Ad}_{g_j}(\xi_j)})_{[g_jh_j:x_j]}$ is zero in tangent  space of $[g_jh_j:x_j]=g_j\cdot\gamma\in G\times_P\p$ for all $j\in\{1,\ldots,n\}$, and $(X_{\xi_k})_{[g_k^{-1}h_k:x_k]}$ is zero in the tangent space of $[g_k^{-1}h_k:x_k]=g_k^{-1}\cdot\gamma\in G\times_P\p$ for all $k\in\{n+1,\ldots,m+n\}$. We conclude that $\mathrm{ker}\hspace{2pt}\phi$ is as claimed.
		
		Now observe that the $G$-orbit of $\gamma$ in $G\times_P\p$ has dimension at least that of the $G$-orbit of $\mu_{\mathfrak{p}}(\gamma)$ in $\g$. Since $\mu_p(\gamma)\in\mathrm{Kos}$, the former orbit has dimension at least $\dim\g-\ell$. An application of Proposition \ref{Proposition: Nice} then reveals that the $G$-orbit of $\gamma$ has dimension exactly $\dim\g-\ell$. The $\g$-centralizer $\g_{\gamma}$ must therefore have dimension equal to $\ell$. By \eqref{Equation: Kernel}, $\dim\mathrm{ker}\hspace{2pt}\phi=(m+n)\ell$. This combines with \eqref{Equation: Intersection} to imply that the intersection of $T_{\alpha}(\nu^{m,n})^{-1}(\mathrm{Kos}_{\mathcal{C}}^{m,n})\oplus V$ and $T_{(\alpha,\alpha)}\Delta(((T^*G)_{\mathcal{C}}^{0,1})^n\times((T^*G)_{\mathcal{C}}^{1,0})^m)$ has dimension $$(m+n)\dim\g+\ell-(m+n)\ell=(m+n)(\dim\g-\ell)+\ell.$$ The dimension of the sum of $T_{\alpha}(\nu^{m,n})^{-1}(\mathrm{Kos}_{\mathcal{C}}^{m,n})\oplus V$ and $T_{(\alpha,\alpha)}\Delta(((T^*G)_{\mathcal{C}}^{0,1})^n\times((T^*G)_{\mathcal{C}}^{1,0})^m)$ is then 
		\begin{align*}
			& \dim(\nu^{m,n})^{-1}(\mathrm{Kos}_{\mathcal{C}}^{m,n})+\dim V+n\dim(T^*G)_{\mathcal{C}}^{0,1}+m\dim(T^*G)_{\mathcal{C}}^{1,0}-(m+n)(\dim\g-\ell)-\ell \\ & = (m+n)\dim(T^*G)_{\mathcal{C}}-((m+n)\dim\g-\ell)+\dim V+n\dim(T^*G)_{\mathcal{C}}^{0,1}+m\dim(T^*G)_{\mathcal{C}}^{1,0}-(m+n)(\dim\g-\ell)-\ell\\
			& = n(\dim(T^*G)_{\mathcal{C}}^{0,1}+\dim\g-\ell)+m(\dim(T^*G)_{\mathcal{C}}^{1,0}+\dim\g-\ell)-((m+n)\dim\g-\ell)+\dim V \\ & \hspace{9pt} +n\dim(T^*G)_{\mathcal{C}}^{0,1}+m\dim(T^*G)_{\mathcal{C}}^{1,0}-(m+n)(\dim\g-\ell)-\ell\\
			& = 2n\dim(T^*G)_{\mathcal{C}}^{0,1}+2m\dim(T^*G)_{\mathcal{C}}^{1,0}-(m+n)\dim\g+\dim V 
		\end{align*}
		We conclude that $T_{\alpha}(\nu^{m,n})^{-1}(\mathrm{Kos}_{\mathcal{C}}^{m,n})\oplus V$ and $T_{(\alpha,\alpha)}\Delta(((T^*G)_{\mathcal{C}}^{0,1})^n\times((T^*G)_{\mathcal{C}}^{1,0})^m)$ are transverse in $$T_{(\alpha,\alpha)}((((T^*G)_{\mathcal{C}}^{0,1})^n\times((T^*G)_{\mathcal{C}}^{1,0})^m)\times(((T^*G)_{\mathcal{C}}^{0,1})^n\times((T^*G)_{\mathcal{C}}^{1,0})^m))$$ if and only if $\dim V=(m+n)\dim\g$. It therefore suffices to prove that $\dim V_i=\dim\g$ for all $i\in\{1,\ldots,m+n\}$.
		
		Given $j\in\{1,\ldots,n\}$ and $k\in\{n+1,\ldots,m+n\}$, consider the surjective linear maps
		$$\g\oplus T_{\gamma}\mathrm{Kos}_{\mathcal{C}}\longrightarrow V_j,\quad(\xi,v)\mapsto[((\mathrm{d}\mathcal{A}_g)_{[h_j:x_j]}(v)-(X_{\mathrm{Ad}_{g_j}(\xi)})_{[g_jh_j:x_j]},v)]$$ and $$\g\oplus T_{\gamma}\mathrm{Kos}_{\mathcal{C}}\longrightarrow V_k,\quad(\xi,v)\mapsto[(v,(X_{\xi})_{[g_k^{-1}h_k:x_k]}+(\mathrm{d}\mathcal{A}_{g_k^{-1}})_{[h_k:x_k]}(v))].$$ Their respective kernels are $\g_{\gamma}\oplus\{0\}\s \g\oplus T_{\gamma}\mathrm{Kos}_{\mathcal{C}}$ and $\mathrm{Ad}_{g_k^{-1}}(\g_{\gamma})\oplus\{0\}\s \g\oplus T_{\gamma}\mathrm{Kos}_{\mathcal{C}}$. We also recall that $\dim\g_{\gamma}=\ell$. These last two sentences imply that $$\dim V_j=\dim V_k=\dim\g+\dim\mathrm{Kos}_{\mathcal{C}}-\ell=\dim\g.$$ Our proof of (i) is therefore complete.
		
		We now verify (ii). Let $\pi$ denote the projection of $$(T^*G)_{\mathcal{C}}^{m,n}\times((T^*G)_{\mathcal{C}}^{0,1})^n\times((T^*G)_{\mathcal{C}}^{1,0})^m\times((T^*G)_{\mathcal{C}}^{0,1})^n\times((T^*G)_{\mathcal{C}}^{1,0})^m\times((T^*G)^{0,1})^n\times((T^*G)^{1,0})^m$$ to the first and last two factors $$(T^*G)_{\mathcal{C}}^{m,n}\times((T^*G)^{0,1})^n\times((T^*G)^{1,0})^m.$$ Our first step is to prove that $\pi$ restricts to a homeomorphism from $X$ to its image $Y\coloneqq\pi(X)$. To this end, suppose that $$([\alpha],\alpha,\alpha,\beta),([\alpha'],\alpha',\alpha',\beta')\in X$$ satisfy $[\alpha]=[\alpha']$ and $\beta=\beta'$. We may find $(g_i,[h_i:x_i]),(g_i',[h_i':x_i'])\in G\times\mathrm{Kos}_{\mathcal{C}}$ for $i\in\{1,\ldots,m+n\}$ such that $$(\alpha_j,\beta_j)=([[g_jh_j:x_j]:[h_j:x_j]],(g_j,\mathrm{Ad}_{h_j}(x_j)))\text{ for }j\in\{1,\ldots,n\},$$
		$$(\alpha_j',\beta_j')=([[g_j'h_j':x_j']:[h_j':x_j']],(g_j',\mathrm{Ad}_{h_j'}(x_j')))\text{ for }j\in\{1,\ldots,n\},$$
		$$(\alpha_k,\beta_k)=([[h_k:x_k]:[g_k^{-1}h_k:x_k]],(g_k,\mathrm{Ad}_{g_k^{-1}h_k}(x_k)))\text{ for }k\in\{n+1,\ldots,m+n\},$$ 
		$$(\alpha_k',\beta_k')=([[h_k':x_k']:[(g_k')^{-1}h_k':x_k']],(g_k',\mathrm{Ad}_{(g_k')^{-1}h_k'}(x_k')))\text{ for }k\in\{n+1,\ldots,m+n\},$$ $$[h_1:x_1]=\cdots=[h_{m+n}:x_{m+n}]\in\mathrm{Kos}_{\mathcal{C}},$$ and $$[h_1':x_1']=\cdots=[h_{m+n}':x_{m+n}']\in\mathrm{Kos}_{\mathcal{C}}.$$ The condition $\beta=\beta'$ then implies that $g_i=g_i'$ for all $i\in\{1,\ldots,m+n\}$. On the other hand, the condition $[\alpha]=[\alpha']$ means that $\alpha$ and $\alpha'$ belong to the same orbit of the group scheme $\mathcal{H}_{\mathrm{Kos}_{\mathcal{C}}^{m,n}}\longrightarrow\mathrm{Kos}_{\mathcal{C}}^{m,n}$ on $(\nu^{m,n})^{-1}(\mathrm{Kos}_{\mathcal{C}}^{m,n})\longrightarrow\mathrm{Kos}_{\mathcal{C}}^{m,n}$. It follows that $\nu^{m,n}(\alpha)=\nu^{m,n}(\alpha')$, so that $[h_i:x_i]=[h_i':x_i']\in\mathrm{Kos}_{\mathcal{C}}$ for all $i\in\{1,\ldots,m+n\}$. These last three sentences imply that $\alpha_i=\alpha_i'$ for all $i\in\{1,\ldots,n\}$, i.e. $\alpha=\alpha'$. We conclude that $\pi$ restricts to a bijection $X\longrightarrow Y$. Since this bijection is evidently continuous, it is a homeomorphism if and only if its inverse is continuous.
		
		Consider the descent of $\nu^{m,n}$ to a continuous map $$(T^*G)_{\mathcal{C}}^{m,n}=(\nu^{m,n})^{-1}(\mathrm{Kos}_{\mathcal{C}}^{m,n})/\mathcal{H}_{\mathrm{Kos}_{\mathcal{C}}^{m,n}}\longrightarrow\mathrm{Kos}_{\mathcal{C}}^{m,n}/\mathcal{H}_{\mathrm{Kos}_{\mathcal{C}}^{m,n}}=\mathrm{Kos}_{\mathcal{C}}^{m,n}.$$ Write $\theta:(T^*G)_{\mathcal{C}}^{m,n}\longrightarrow\mathrm{Kos}_{\mathcal{C}}$ for its composition with $$\mathrm{Kos}_{\mathcal{C}}^{m,n}\longrightarrow\mathrm{Kos}_{\mathcal{C}},\quad (x_1,\ldots,x_n,y,\ldots,y,z_{n+1},\ldots,z_{m+n})\mapsto y.$$ Let us also consider the projection $$\sigma:T^*G=G\times\g\longrightarrow G.$$ These preliminaries allow us to define the map $\tau$ from $$(T^*G)_{\mathcal{C}}^{m,n}\times(((T^*G)^{0,1})^n\times((T^*G)^{1,0})^m)$$
		to $$(T^*G)_{\mathcal{C}}^{m,n}\times((T^*G)_{\mathcal{C}}^{0,1})^n\times((T^*G)_{\mathcal{C}}^{1,0})^m\times((T^*G)_{\mathcal{C}}^{0,1})^n\times((T^*G)_{\mathcal{C}}^{1,0})^m\times((T^*G)^{0,1})^n\times((T^*G)^{1,0})^m$$ by $\tau(\gamma,\beta)=(\gamma,\alpha,\alpha,\beta)$, where $\alpha=(\alpha_1,\ldots\alpha_{m+n})\in((T^*G)_{\mathcal{C}}^{0,1})^n\times((T^*G)_{\mathcal{C}}^{1,0})^m$ is given by 
		$$\alpha_j=[\sigma(\beta_j)\cdot\theta(\gamma):\theta(\gamma)]\text{ for }j\in\{1,\ldots,n\}$$
		and
		$$\alpha_k=[\theta(\gamma):\sigma(\beta_k)^{-1}\cdot\theta(\gamma)]\text{ for }k\in\{n+1,\ldots,m+n\}.$$ This map is evidently continuous. On the other hand, the discussion in the previous paragraph implies that $\tau$ restricts to an inverse of the map $X\longrightarrow Y$.
		
		Part (i) and \cite[Lemma 2.0.5]{WehrheimWoodward} imply that $\pi$ restricts to an immersion of $X$ into $(T^*G)_{\mathcal{C}}^{m,n}\times((T^*G)^{0,1})^n\times((T^*G)^{1,0})^m$. Recalling that $X\longrightarrow\pi(X)=Y$ is a homeomorphism, we see that $\pi\big\vert_X$ is an embedding of manifolds. It therefore remains only to prove that this embedding is a locally closed embedding of algebraic varieties. A first step is to observe that $Y$ is a constructible subset of $(T^*G)_{\mathcal{C}}^{m,n}\times((T^*G)^{0,1})^n\times((T^*G)^{1,0})^m$. Its closure $\overline{Y}$ in the Zariski topology must therefore coincide with its closure in the Euclidean topology \cite[Section XII, Proposition 2.2]{SGA1}. On the other hand, $Y$ being an embedded submanifold of $(T^*G)_{\mathcal{C}}^{m,n}\times((T^*G)^{0,1})^n\times((T^*G)^{1,0})^m$ forces $Y$ to be a Euclidean-open subset of $\overline{Y}$. The constructibility of $Y$ then implies that $Y$ must be Zariski-open in $\overline{Y}$ \cite[Section XII, Corollary 2.3]{SGA1}. In other words, $Y$ is locally closed in the Zariski topology of $(T^*G)_{\mathcal{C}}^{m,n}\times((T^*G)^{0,1})^n\times((T^*G)^{1,0})^m$. It is thereby a subvariety of $(T^*G)_{\mathcal{C}}^{m,n}\times((T^*G)^{0,1})^n\times((T^*G)^{1,0})^m$. Since $Y$ is a submanifold, it is smooth as a subvariety of $(T^*G)_{\mathcal{C}}^{m,n}\times((T^*G)^{0,1})^n\times((T^*G)^{1,0})^m$. The map $X\longrightarrow\pi(X)=Y$ is then a bijective morphism of smooth varieties, i.e. an isomorphism. This completes the proof of (ii).
	\end{proof}
	
	The following is an immediate consequence.
	
	\begin{corollary}
		If $m,n\in\mathbb{Z}_{\geq 0}$ and $(m,n)\neq (0,0)$, then the Lagrangian relations $$\Lambda_{\mathcal{C}}^{m,n}:(T^*G)_{\mathcal{C}}^{m,n}\Longrightarrow ((T^*G)_{\mathcal{C}}^{0,1})^n\times((T^*G)_{\mathcal{C}}^{1,0})^m$$ and $$\Gamma_{\mathcal{C}}^{m,n}:((T^*G)_{\mathcal{C}}^{0,1})^n\times((T^*G)_{\mathcal{C}}^{1,0})^m\Longrightarrow ((T^*G)^{0,1})^n\times((T^*G)^{1,0})^m$$ are strongly composable. The composition $\Gamma_{\mathcal{C}}^{m,n}\circ\Lambda_{\mathcal{C}}^{m,n}$ is a smooth Lagrangian subvariety of $(T^*G)_{\mathcal{C}}^{m,n}\times\overline{((T^*G)^{0,1})^n\times((T^*G)^{1,0})^m}$.
	\end{corollary}
	
	Let $\sss:T^*G\longrightarrow\g$ and $\ttt:T^*G\longrightarrow\g$ denote the source and target of $T^*G$, respectively. Set
	$$\nu^{m,n}:(\underbrace{\sss,\ldots,\sss}_{m+n\text{ times}},\underbrace{\ttt,\ldots,\ttt}_{m+n\text{ times}}):(T^*G)^{m+n}\longrightarrow\g^{m+n}\times\overline{\g}^{m+n}$$ and
	$$\Lambda^{m,n}\coloneqq\{((\alpha_1,\ldots,\alpha_{m+n}),[(\alpha_1,\ldots,\alpha_{m+n})]):(\alpha_1,\ldots,\alpha_{m+n})\in(\nu^{m,n})^{-1}(\mathrm{Kos}^{m,n})\}.$$
	In analogy with the discussion of $\Lambda_{\mathcal{C}}^{m,n}$, $\Lambda^{m,n}$ is a Lagrangian subvariety of  $(((T^*G)^{0,1})^n\times((T^*G)^{1,0})^m)\times\overline{(T^*G)^{m,n}}$. We may consider the set-theoretic relation $$\Lambda^{m,n}\circ\Gamma_{\mathcal{C}}^{m,n}\s((T^*G)_{\mathcal{C}}^{0,1})^n\times(((T^*G)_{\mathcal{C}}^{1,0})^m)\times (T^*G)^{m,n}.$$ By arguments analogous to those in the proof of Theorem \ref{Theorem: Technical}, $\Lambda^{m,n}$ and $\Gamma_{\mathcal{C}}^{m,n}$ are composable as Lagrangian relations. We are unable to prove that $\Lambda^{m,n}$ and $\Gamma_{\mathcal{C}}^{m,n}$ are strongly composable. One may nevertheless form the Lagrangian relation 
	$$\Lambda^{m,n}\circ\Gamma_{\mathcal{C}}^{m,n}\circ\Lambda_{\mathcal{C}}^{m,n}:(T^*G)_{\mathcal{C}}^{m,n}\Longrightarrow (T^*G)^{m,n}.$$ This completes the proof of Main Theorem \ref{Theorem: Main Theorem 3}.
	
	\section{Some Lie-theoretic considerations in quasi-Poisson geometry}\label{Section: Some Lie-theoretic considerations in quasi-Poisson geometry}
	This section is a multiplicative counterpart to Section \ref{Section: Some Lie-theoretic considerations in Poisson geometry}. In Subsection \ref{Subsection: Regular elements in reductive groups}, we recall aspects of regular elements in reductive groups. A useful description of regular elements in Levi subgroups is given in Subsection \ref{Subsection: Levi subgroups}. Subsection \ref{Subsection: The double} then reviews the quasi-Hamiltonian geometry of the double $\mathrm{D}(G)$. The quasi-Hamiltonian structure on $G\times_{U(P)}P$ is reviewed in Subsection \ref{Subsection: Quasi-Hamiltonian structure}, where $P\s G$ is a parabolic subgroup.
	
	\subsection{Regular elements in reductive groups}\label{Subsection: Regular elements in reductive groups}
	Let $H$ be a connected reductive affine algebraic group. Write $H_h\s H$ for the centralizer of $h\in H$ under the conjugation action of $H$ on itself. One has the \textit{regular locus}
	$$H_{\text{reg}}\coloneqq\{h\in H:\dim H_h=\mathrm{rank}\hspace{2pt} H\}.$$ If we equip the Lie algebra of $H$ with an $H$-invariant, non-degenerate, symmetric bilinear form, then $H$ inherits the canonical Cartan--Dirac structure. On the other hand, the regular locus of a Dirac manifold is the union of its top-dimensional pre-symplectic leaves. It turns out that $H_{\text{reg}}$ is the regular locus of the Cartan--Dirac structure on $H$.
	
	\subsection{Regular elements in Levi subgroups}\label{Subsection: Levi subgroups} 
	Let $P\s G$ be a parabolic subgroup. Given any $p\in P$, let $[p]\in L(P)$ denote the image of $p$ under the quotient $P\longrightarrow L(P)$. Write $P_{[p]}\s P$ for the $P$-centralizer of $[p]\in L(P)$. We have the following counterpart to Lemma \ref{Lemma: Single orbit}.
	
	\begin{lemma}\label{Lemma: Single orbit group case}
		Suppose that $p\in P$. 
		\begin{itemize}
			\item[\textup{(i)}] If $pU(P)\cap G_{\emph{reg}}$ is non-empty, then it is an orbit of $P_{[p]}$ in $G$.
			\item[\textup{(ii)}] One has $pU(P)\cap G_{\emph{reg}}\neq\emptyset$ if and only if $[p]\in L(P)_{\emph{reg}}$.
			\item[\textup{(iii)}] If $P$ is a Borel subgroup, then $pU(P)\cap G_{\emph{reg}}$ is non-empty.
		\end{itemize}
	\end{lemma}
	
	\begin{proof}
		The proofs of (i) and the forward implication in (ii) are completely analogous to their counterparts in Lemma \ref{Lemma: Single orbit}. At the same time, Part (iii) is an immediate consequence of (ii). It therefore remains only to establish the backward implication in (ii). To this end, choose a Levi factor $L\s G$ of $P$. Our task is to prove that $pU(P)\cap G_{\text{reg}}\neq\emptyset$ for all $p\in L_{\text{reg}}$.
		
		Suppose that $p\in L_{\text{reg}}$. Choose a maximal torus $T\s G$ and Borel subgroup $B\s G$ for which $L$ and $P$ are standard, i.e. associated to a subset of the simple roots. Let $p=p_sp_u$ be the Jordan decomposition of $p$ into a semisimple element $p_s\in L$ and unipotent element $p_u\in L$. As $U(P)$ is invariant under conjugation by elements of $L$, we may assume that $p_s\in T$. Let us also note that $p_u$ is a unipotent element of $L_{p_s}$. We may therefore assume that $p_u\in U(B)\cap L_{p_s}$.
		
		It is straightforward to verify that $L_{p_s}$ is a Levi subgroup of $G_{p_s}$. By \cite[Theorem 1.3(a)]{lus-spa:79}, there exists $h\in U(P)$ satisfying $p_uh\in (G_{p_s})_{\text{reg}}$. Note that $p_uh$ is a unipotent element of $G$ that commutes with $p_s$. We conclude that $G_{ph}=(G_{p_s})_{p_uh}$ is $\ell$-dimensional, i.e. $ph\in G_{\text{reg}}$. This completes the proof.
	\end{proof}
	
	\begin{lemma}\label{Lemma: In Borel group case}
		If $p\in P\cap G_{\emph{reg}}$, then $G_p\s P$. 
	\end{lemma}
	
	\begin{proof}
		Consider the adjoint group $G_{\mathrm{ad}}\coloneqq G/Z(G)$ and its conjugation action on $G$. An application of \cite[Lemma 3.1]{balibanu-steinberg} reveals that $(G_{\mathrm{ad}})_{p}$ is connected. On the other hand, arguments analogous to those in the proof of Lemma \ref{Lemma: In Borel} imply that $\mathrm{Lie}(G_p)\s\mathfrak{p}$. It follows that $(G_{\mathrm{ad}})_p\s P_{\mathrm{ad}}\coloneqq P/Z(G)$. By taking preimages under the quotient map $G\longrightarrow G_{\mathrm{ad}}$, we deduce that $G_p\s P$. 
	\end{proof}
	
	\subsection{The double $\mathrm{D}(G)$}\label{Subsection: The double}
	Adopt the notation $\mathrm{D}(G)\coloneqq G\times G$. Note that $G\times G$ acts on $\mathrm{D}(G)$ by
	$$(h_1,h_2)\cdot (g_1,g_2)\coloneqq (h_1g_1h_2^{-1},h_2g_2h_2^{-1}),\quad (h_1,h_2)\in G\times G,\text{ }(g_1,g_2)\in\mathrm{D}(G).$$ At the same time, let $\theta^L,\theta^R\in\Omega^1(G,\g)$ denote the left and right-invariant Maurer Cartan forms on $G$, respectively. Let us also write $\pi_1,\pi_2:\mathrm{D}(G)\longrightarrow G$ for the projections to the first and second factors, respectively, and $\langle\cdot,\cdot\rangle:\g\otimes\g\longrightarrow\g$ for the Killing form. Consider the $2$-form $\omega\in\Omega^2(G)$ defined by 
	$$\omega_{\mathrm{D}(G)}\coloneqq\frac{1}{2}\langle\mathrm{Ad}_{\pi_2}(\pi_1^*\theta^L),\pi_1^*\theta^L\rangle+\frac{1}{2}\langle\pi_1^*\theta^L,\pi_2^*\theta^L+\pi_2^*\theta^R\rangle,$$ as well as the algebraic map
	$$\mu_{\mathrm{D}(G)}=(\mu_1,\mu_2):\mathrm{D}(G)\longrightarrow G\times G,\quad (g_1,g_2)\mapsto (g_1g_2g_1^{-1},g_2^{-1}).$$ With respect to the above-mentioned $G\times G$-action, $(\mathrm{D}(G),\omega_{\mathrm{D}(G)},\mu_{\mathrm{D}(G)})$ is a quasi-Hamiltonian $G\times G$-space \cite{ale-mal-mei:jdg}. 
	
	\subsection{A quasi-Hamiltonian structure on $G\times_{U(P)}P$}\label{Subsection: Quasi-Hamiltonian structure}
	Let $P\s G$ be a parabolic subgroup. Note that $G\times L(P)$ acts on $G\times_{U(P)}P$ by
	\begin{equation}\label{Equation: Product action}(g,[p])\cdot[h:k]\coloneqq [ghp^{-1}:php^{-1}],\quad (g,[p])\in G\times L(P),\text{ }[h:k]\in G\times_{U(P)}P.\end{equation}
	We also have the map
	$$\mu_P:G\times_{U(P)}P\longrightarrow G\times L(P),\quad [g:p]\mapsto (gpg^{-1},[p^{-1}]).$$
	Write $j:G\times P\longrightarrow\mathrm{D}(G)$ and $\pi:G\times P\longrightarrow G\times_{U(P)}P$ for the inclusion and quotient maps, respectively.
	
	\begin{proposition}
		The following statements are true:
		\begin{itemize}
			\item[\textup{(i)}] $\pi^*\omega_P=j^*\omega_{\mathrm{D(G)}}$ for a unique algebraic $2$-form $\omega_P$ on $G\times_{U(P)}P$;
			\item[\textup{(ii)}] the $G\times L(P)$-action \eqref{Equation: Product action}, $2$-form $\omega_P$, and map $\mu_P$ render $G\times_{U(P)}P$ a quasi-Hamiltonian $G\times L(P)$-variety.
		\end{itemize}
	\end{proposition}
	
	\begin{proof}
		These results follow immediately from Theorem 2.21(4), Lemma 4.3, and Example 4.5 of \cite{balibanu-mayrand}.
	\end{proof}
	
	The reader should note that the quasi-Hamiltonian structure of $G\times_{U(P)}P$ is studied extensively in \cite{boa:11}.
	
	\section{TQFTs from multiplicative Grothendieck--Springer resolutions}\label{Section: Multiplicative TQFTs}
	
	This section is a multiplicative counterpart to Sections \ref{Section: Partial Grothendieck--Springer resolutions} and \ref{Section: TQFTs from}. We begin by introducing the multiplicative partial Grothendieck--Springer resolutions $\nu_{\mathcal{C}}:G_{\mathcal{C}}\longrightarrow G$ in Subsection \ref{Subsection: The multiplicative}. In Subsection \ref{Subsection: Quasi}, we introduce and examine a quasi-symplectic groupoid $\mathrm{D}(G)_{\mathcal{C}}\tto G_{\mathcal{C}}$. Steinberg slices are recalled in Subsection \ref{Subsection: Steinberg}. These slices are used to construct admissible global slices to $(\mathrm{D}(G)_{\mathcal{C}})_{\text{reg}}\tto(G_{\mathcal{C}})_{\text{reg}}$ in Subsection \ref{Subsection: Quasi slices}. In Subsection \ref{Subsection: Multiplicative alternations}, we describe the TQFTs that result. 
	
	\subsection{The multiplicative partial Grothendieck--Springer resolution $\nu_{\mathcal{C}}:G_{\mathcal{C}}\longrightarrow G$}\label{Subsection: The multiplicative} 
	Let $\mathcal{C}$ be a conjugacy class of parabolic subgroups of $G$. Consider subvariety
	$$G_{\mathcal{C}}\coloneqq\{(P,g)\in\mathcal{C}\times G:g\in P\}$$ of $\mathcal{C}\times G$. Note that $G$ acts on $G_{\mathcal{C}}$ by
	\begin{equation}\label{Equation: Action quasi} g\cdot(P,h)\coloneqq (gPg^{-1},ghg^{-1}),\quad g\in G,\text{ }(P,h)\in G_{\mathcal{C}}.\end{equation} The morphism $$\nu_{\mathcal{C}}:G_{\mathcal{C}}\longrightarrow G,\quad (P,g)\mapsto g$$ is equivariant with respect to the conjugation of $G$ on itself. 
	
	\begin{definition}
		The morphism $\nu_{\mathcal{C}}:G_{\mathcal{C}}\longrightarrow G$ is called the \textit{multiplicative partial Grothendieck--Springer resolution} determined by $\mathcal{C}$.
	\end{definition}
	
	The reader should note that \cite{boa:thesis} outlines a role for partial Grothendieck--Springer resolutions in quasi-Poisson geometry. For $P\in\mathcal{C}$, the map $$\gamma_{P}:G\times_P P\longrightarrow G_{\mathcal{C}}\quad [g:p]\mapsto (gPg^{-1},gpg^{-1})$$ is a $G$-equivariant variety isomorphism.
	
	\subsection{The quasi-symplectic groupoid $\mathrm{D}(G)_{\mathcal{C}}$}\label{Subsection: Quasi}
	Consider a conjugacy class $\mathcal{C}$ of parabolic subgroups of $G$, as well as an element $P\in\mathcal{C}$. As discussed in Subsection \ref{Subsection: Quasi-Hamiltonian structure}, $G\times_{U(P)}P$ is a quasi-Hamiltonian $G$-variety. Proposition \ref{Proposition: Integrating quasi-Poisson manifolds} then yields a quasi-symplectic groupoid $$\mathrm{D}(G)_{P}\coloneqq\mathrm{Pair}(G\times_{U(P)}P)\sll{e}L(P)\tto G\times_P P.$$
	Note that the source (resp. target) morphism is projection from the first (resp. second) factor of $G\times_{U(P)}P$. Write $\sss_{P}:\mathrm{D}(G)_{P}\longrightarrow G_{\mathcal{C}}$ (resp. $\ttt_{P}:\mathrm{D}(G)_{P}\longrightarrow G_{\mathcal{C}}$) for the result of composing the source (resp. target) of $\mathrm{D}(G)_{P}\tto G\times_PP$ with the isomorphism $\gamma_{P}:G\times_P P\longrightarrow G_{\mathcal{C}}$. It follows that $\mathrm{D}(G)_{P}$ is a quasi-symplectic groupoid over $G_{\mathcal{C}}$ with source $\sss_{P}$ and target $\ttt_{P}$; groupoid multiplication and inversion are induced from those of the pair groupoid of $G\times_{U(P)} P$, and composing $\gamma_{P}^{-1}:G_{\mathcal{C}}\longrightarrow G\times_P P$ with the unit bisection of $\mathrm{D}(G)_{P}\tto G\times_{P}P$ gives the unit bisection of $\mathrm{D}(G)_{P}\tto G_{\mathcal{C}}$. Straightforward computations also reveal that
	$$\sss_{P}([[g_1:p_1]:[g_2:p_2]])=(g_1Pg_1^{-1},g_1p_1g_1^{-1})\quad\text{and}\quad\ttt_{P}([[g_1:p_1]:[g_2:p_2]])=(g_2Pg_2^{-1},g_2p_2g_2^{-1})$$ for all $[[g_1:p_1]:[g_2:p_2]]\in\mathrm{D}(G)_{P}$.
	
	\begin{proposition}\label{Proposition: Third canonical}
		If $P_1,P_2\in\mathcal{C}$, then there is a canonical $G$-equivariant isomorphism
		\begin{equation}\label{Equation: Diag2}\begin{tikzcd}[column sep=30pt, row sep=40pt]
				\mathrm{D}(G)_{P_1} \ar[rr, "\Psi_{(P_1,P_2)}"] \ar[dr, shift left = 0.8ex, "\ttt_{P_1}"]\ar[dr, shift right = 0.8ex, swap, "\sss_{P_1}"] & & \mathrm{D}(G)_{P_2} \ar[dl, shift left = 0.8ex, "\ttt_{P_2}"]\ar[dl, shift right = 0.8ex, swap, "\sss_{P_2}"] \\
				& G_{\mathcal{C}} &
		\end{tikzcd}\end{equation}
		of quasi-symplectic groupoids.
	\end{proposition}
	
	\begin{proof}
		The proof is analogous to that of Proposition \ref{Proposition: Second canonical}.
	\end{proof}
	
	We now consider the set $$\mathrm{D}(G)_{\mathcal{C}}\coloneqq\bigg(\bigsqcup_{P\in\mathcal{C}}\mathrm{D}(G)_{P}\bigg)\Bigm/\sim,$$ where $\sim$ is the equivalence relation defined by $$(\alpha_1\in\mathrm{D}(G)_{P_1})\sim(\alpha_2\in\mathrm{D}(G)_{P_2})\Longleftrightarrow\alpha_2=\Psi_{(P_1,P_2)}(\alpha_1).$$
	
	\begin{corollary}\label{Corollary: Canonical2}
		The set $\mathrm{D}(G)_{\mathcal{C}}$ has a unique $G$-equivariant quasi-symplectic groupoid structure $\mathrm{D}(G)_{\mathcal{C}}\xbigtoto[\ttt_{\mathcal{C}}]{\sss_{\mathcal{C}}}G_{\mathcal{C}}$ such that $$\begin{tikzcd}[column sep=30pt, row sep=40pt]
			\mathrm{D}(G)_{P} \ar[rr] \ar[dr, shift left = 0.8ex, "\ttt_{P}"]\ar[dr, shift right = 0.8ex, swap, "\sss_{P}"] & & \mathrm{D}(G)_{\mathcal{C}} \ar[dl, shift left = 0.8ex, "\ttt_{\mathcal{C}}"]\ar[dl, shift right = 0.8ex, swap, "\sss_{\mathcal{C}}"] \\
			& G_{\mathcal{C}} &
		\end{tikzcd}$$
		is a $G$-equivariant quasi-symplectic groupoid isomorphism for all $P\in\mathcal{C}$.
	\end{corollary}
	
	\begin{proof}
		This follows immediately from Proposition \ref{Proposition: Third canonical} and the definition of $\mathrm{D}(G)_{\mathcal{C}}$.
	\end{proof}
	
	Given $(P,p)\in G_{\mathcal{C}}$, the isomorphism $\mathrm{D}(G)_{\mathcal{C}}\cong\mathrm{D}(G)_{P}$ restricts to an isomorphism
	\begin{align}(\mathrm{D}(G)_{\mathcal{C}})_{(P,p)} & \cong (\mathrm{D}(G)_{P})_{[e:p]} \\ & = \left\{[[g:y]:[h:z]]\in \left(G\times_{U(P)}P)\times_{L(P)}(G\times_{U(P)}P)\right)/L(P):[g:y]=[e:x]=[h:z]\text{ in }G\times_P P\right\}\notag\end{align} between the isotropy groups of $(P,p)$ and $[e:p]\in G\times_P P$.  We use this isomorphism to freely identify the two isotropy groups in our next proposition.
	
	\begin{proposition}\label{Proposition: Isotropy group 2}
		If $(P,p)\in G_{\mathcal{C}}$, then 
		$$L(P)_{[p]}\longrightarrow (\mathrm{D}(G)_{\mathcal{C}})_{(P,p)}=\left(\left(G\times_{U(P)} P)\times_{L(P)}(G\times_{U(P)} P)\right)/L(P)\right)_{[e:p]},\quad [q]\mapsto [[q:q^{-1}pq]:[e:p]]$$
		is a well-defined isomorphism of algebraic groups.
	\end{proposition}
	
	\begin{proof}
		The proof is entirely analogous to that of Proposition \ref{Proposition: Isotropy group}.
	\end{proof}
	
	\subsection{Steinberg slices}\label{Subsection: Steinberg}
	Fix a maximal torus $T\s G$ and Borel subgroup $B\s G$ satisfying $T\s B$. Choose an enumeration $\alpha_1,\ldots,\alpha_{\ell}:T\longrightarrow\mathbb{C}^{\times}$ of the resulting simple roots. Given $i\in\{1,\ldots,\ell\}$, let $s_i\in W\coloneqq N_G(T)/T$ denote the simple reflection corresponding to $\alpha_i$. Choose a lift $\widetilde{s_i}\in N_G(T)$ of $s_i$ for each $i\in\{1,\ldots,\ell\}$. Consider the subset
	$$\mathrm{Ste}\coloneqq\prod_{i=1}^{\ell}\big(\exp({\g_{\alpha_i}})\widetilde{s_i}\big)\s G,$$ where $\exp:\g\longrightarrow G$ is the exponential map and $\g_{\alpha_i}\s\g$ denotes the root space associated to $\alpha_i$. One calls $\mathrm{Ste}$ a \textit{Steinberg slice}, largely to recognize the following results of Steinberg \cite{Steinberg}.
	
	\begin{theorem}
		The Steinberg slice $\mathrm{Ste}$ enjoys the following properties.
		\begin{itemize}
			\item[\textup{(i)}] One has $\mathrm{Ste}\s G_{\emph{reg}}$.
			\item[\textup{(ii)}] If $G$ is simply-connected, then $\mathrm{Ste}$ intersects each conjugacy class in $G_{\emph{reg}}$ transversely and in a singleton.
			\item[\textup{(iii)}] If $G$ is simply-connected, then the categorical quotient $G\longrightarrow G\sll{} G\coloneqq\mathrm{Spec}(\mathbb{C}[G]^G)$ restricts to an isomorphism $\mathrm{Ste}\overset{\cong}\longrightarrow G\sll{} G$. 
		\end{itemize}
	\end{theorem}
	
	\subsection{Global slices to $\mathrm{D}(G)_{\mathcal{C}}$}\label{Subsection: Quasi slices}
	Let $X$ be a quasi-Poisson Hamiltonian $G$-variety. We define the \textit{regular locus} $X_{\text{reg}}$ to be the union of the top-dimensional quasi-Hamiltonian leaves of $X$.
	
	We now return to the matter at hand. Let $\mathcal{C}$ be a conjugacy class of parabolic subgroups of $G$. In light of \cite[Corollary 2.5]{balibanu-steinberg}, $$\mathrm{Ste}_{\mathcal{C}}\coloneqq\nu_{\mathcal{C}}^{-1}(\mathrm{Ste})$$ is a smooth subvariety of $G_{\mathcal{C}}$. Write $(\mathrm{D}(G)_{\mathcal{C}})_{\text{reg}}\tto(G_{\mathcal{C}})_{\text{reg}}$ for the pullback of $\mathrm{D}(G)_{\mathcal{C}}\tto G_{\mathcal{C}}$ to $(G_{\mathcal{C}})_{\text{reg}}\s G_{\mathcal{C}}$. We relate the preceding discussion to the notion of an \textit{admissible global slice}, a notion defined in Subsection \ref{Subsection: TQFTs}.  
	
	\begin{theorem}\label{Theorem: Quasi leaf intersection}
		If $G$ is simply-connected, then $\mathrm{Ste}_{\mathcal{C}}$ is an admissible global slice to $(\mathrm{D}(G)_{\mathcal{C}})_{\emph{reg}}\tto(G_{\mathcal{C}})_{\emph{reg}}$. 
	\end{theorem}
	
	\begin{proof}
		We first show that $\mathrm{Ste}_{\mathcal{C}}$ is a global slice to $(\mathrm{D}(G)_{\mathcal{C}})_{\text{reg}}\tto(G_{\mathcal{C}})_{\text{reg}}$. To this end, choose $P\in\mathcal{C}$. Let $(\mathrm{D}(G)_P)_{\text{reg}}\tto (G\times_P P)_{\text{reg}}$ denote the pullback of $\mathrm{D}(G)_P\tto G\times_P P$ to $(G\times_P P)_{\text{reg}}\s G\times_P P$. By Proposition \ref{Proposition: Third canonical}, it suffices to show that $\nu_P^{-1}(\mathrm{Ste})$ is a global slice to $(\mathrm{D}(G)_P)_{\text{reg}}\tto (G\times_P P)_{\text{reg}}$. A first step is to notice that the orbits of $(\mathrm{D}(G)_P)_{\text{reg}}\tto (G\times_P P)_{\text{reg}}$ are the top-dimensional orbits of $\mathrm{D}(G)_P\tto G\times_P P$. A straightforward application of Remark \ref{Remark: Leaves} reveals that the orbits of $\mathrm{D}(G)_P\tto G\times_P P$ are the subvarieties $G\times_{P_{[p]}}(pU(P))$, where $[p]\in L(P)$. Lemma \ref{Lemma: Single orbit group case}(ii) now tells us that the top-dimensional orbits are the subvarieties $G\times_{P_{[p]}}(pU(P))$ for which $pU(P)\cap G_{\text{reg}}\neq 0$.  It therefore suffices to prove the following: $\nu_P^{-1}(\mathrm{Ste})$ has a non-empty intersection with $G\times_{P_{[p]}}(pU(P))$ if and only if $pU(P)\cap G_{\text{reg}}\neq\emptyset$, in which case the intersection is a singleton. The ``if and only if" assertion follows immediately from the fact that $\mathrm{Ste}$ is a fundamental domain for the conjugation action of $G$ on $G_{\text{reg}}$.
		
		To finish proving that $\mathrm{Ste}_{\mathcal{C}}$ is a global slice, suppose that $[p]\in L(P)$ satisfies $pU(P)\cap G_{\text{reg}}\neq\emptyset$. We must prove that any two elements of $(G\times_{P_{[p]}}(pU(P)))\cap\nu_P^{-1}(\mathrm{Ste})$ coincide. In light of Lemmas \ref{Lemma: Single orbit group case}(i) and \ref{Lemma: In Borel group case}, the argument is analogous to one in the proof of Theorem \ref{Theorem: Leaf intersection}. 
		
		It remains only to establish that the global slice $\mathrm{Ste}_{\mathcal{C}}$ is admissible. Given $(P,p)\in\mathrm{Ste}_{\mathcal{C}}$, the previous two paragraphs and Lemma \ref{Lemma: Single orbit group case}(ii) imply that $[p]\in L(P)_{\text{reg}}$. It follows that $L(P)_{[p]}$ is abelian \cite{springerta}. By Proposition \ref{Proposition: Isotropy group 2}, the isotropy group $((T^*G)_{\mathcal{C}})_{(\p,x)}$ is also abelian. We conclude that $\mathrm{Ste}_{\mathcal{C}}$ is admissible. 
	\end{proof}
	
	Consider the flag variety $\mathcal{B}$ of all Borel subgroups of $G$. The following is presumably known by experts in the field.
	
	\begin{corollary}\label{Corollary: Borel 2}
		If $G$ is simply-connected, then $\mathrm{Ste}_{\mathcal{B}}$ is an admissible global slice to $\mathrm{D}(G)_{\mathcal{B}}\tto G_{\mathcal{B}}$.
	\end{corollary}
	
	\begin{proof}
		In light of Theorem \ref{Theorem: Quasi leaf intersection}, the proof is analogous to that of Corollary \ref{Corollary: Borel}.
	\end{proof}
	
	\subsection{Grothendieck--Springer alterations of the multiplicative Moore--Tachikawa TQFT}\label{Subsection: Multiplicative alternations}
	Take $G$ to be simply-connected, and let $\mathcal{C}$ be a conjugacy class of parabolic subgroups of $G$. Theorem \ref{Theorem: Quasi leaf intersection} and \cite[Corollary 2.5]{balibanu-steinberg} imply that $\G=(\mathrm{D}(G)_{\mathcal{C}})_{\text{reg}}$ and $S=\mathrm{Ste}_{\mathcal{C}}$ satisfy the hypotheses of Theorem \ref{Theorem: TQFTs from quasi-symplectic groupoids}. Write $\eta_{\mathcal{C}}:\Cob_2\longrightarrow\WS$ for the resulting TQFT, and set
	$\mathrm{D}(G)_{\mathcal{C}}^{m,n}\coloneqq \eta_{\mathcal{C}}(C_{m,n})$ for $(m,n)\neq (0,0)$.
	The following two special cases warrant further discussion. 
	
	\begin{itemize}
		\item[\textup{(i)}] If $\mathcal{C}=\{G\}$, then $\eta_{\mathcal{C}}$ is a version of the multiplicative \textit{open Moore--Tachikawa TQFT}; it is similar to the one first constructed in \cite{balibanu-mayrand}, except that here we do not need to quotient the first factor of $D(G)$ by the center. The varieties $\mathrm{D}(G)^{m,n}\coloneqq \mathrm{D}(G)_{\mathcal{C}}^{m,n}$ also feature in \cite{cro-may2:24}.
		\item[\textup{(ii)}] Suppose that $\mathcal{C}=\mathcal{B}$ is the conjugacy class of Borel subgroups. By Corollary \ref{Corollary: Borel 2}, $\mathrm{Ste}_{\mathcal{C}}=\mathrm{Ste}_{\mathcal{B}}$ is an admissible global slice to the entire symplectic groupoid $\mathrm{D}(G)_{\mathcal{C}}=\mathrm{D}(G)_{\mathcal{B}}$.
	\end{itemize}
	
	\bibliographystyle{acm} 
	\bibliography{grothendieck-springer}

\begin{thebibliography}{10}

\bibitem{ale-kos-mei:02}
{\sc Alekseev, A., Kosmann-Schwarzbach, Y., and Meinrenken, E.}
\newblock Quasi-{P}oisson manifolds.
\newblock {\em Canad. J. Math. 54}, 1 (2002), 3--29.

\bibitem{ale-mal-mei:jdg}
{\sc Alekseev, A., Malkin, A., and Meinrenken, E.}
\newblock Lie group valued moment maps.
\newblock {\em J. Differential Geom. 48}, 3 (1998), 445--495.

\bibitem{ale-mei:24}
{\sc Alekseev, A., and Meinrenken, E.}
\newblock On the coadjoint {V}irasoro action.
\newblock {\em J. Geom. Phys. 195\/} (2024), Paper No. 105029, 32.

\bibitem{ara:18}
{\sc Arakawa, T.}
\newblock Chiral algebras of class $\mathcal{S}$ and {M}oore--{T}achikawa
  symplectic varieties.
\newblock arXiv:1811.01577.

\bibitem{balibanu-steinberg}
{\sc B{\u a}libanu, A.}
\newblock Steinberg slices and group-valued moment maps.
\newblock {\em Adv. Math. 402\/} (2022), Paper No. 108344, 46.

\bibitem{balibanu-mayrand}
{\sc B{\u a}libanu, A., and Mayrand, M.}
\newblock Reduction along strong {D}irac maps.
\newblock {\em Proc. Lond. Math. Soc. (3) 132}, 2 (2026), Paper No. e70123, 57.

\bibitem{Bates}
{\sc Bates, S., and Weinstein, A.}
\newblock {\em Lectures on the geometry of quantization}, vol.~8 of {\em
  Berkeley Mathematics Lecture Notes}.
\newblock American Mathematical Society, Providence, RI; Berkeley Center for
  Pure and Applied Mathematics, Berkeley, CA, 1997.

\bibitem{bielawski97}
{\sc Bielawski, R.}
\newblock Hyperk\"{a}hler structures and group actions.
\newblock {\em J. London Math. Soc. (2) 55}, 2 (1997), 400--414.

\bibitem{bie:23}
{\sc Bielawski, R.}
\newblock On the {M}oore-{T}achikawa varieties.
\newblock {\em J. Geom. Phys. 183\/} (2023), Paper No. 104685, 8.

\bibitem{boa:07}
{\sc Boalch, P.}
\newblock Quasi-{H}amiltonian geometry of meromorphic connections.
\newblock {\em Duke Math. J. 139}, 2 (2007), 369--405.

\bibitem{boa:thesis}
{\sc Boalch, P.}
\newblock Geometry of moduli spaces of meromorphic connections on curves,
  {S}tokes data, wild nonabelian {H}odge theory, hyperkähler manifolds,
  isomonodromic deformations, {P}ainlevé equations, and relations to {L}ie
  theory.
\newblock arXiv:1305.6593.

\bibitem{boa-yam:15}
{\sc Boalch, P., and Yamakawa, D.}
\newblock Twisted wild character varieties.
\newblock arXiv:1512.08091.

\bibitem{boa:11}
{\sc Boalch, P.~P.}
\newblock Riemann-{H}ilbert for tame complex parahoric connections.
\newblock {\em Transform. Groups 16}, 1 (2011), 27--50.

\bibitem{bra-fer:14}
{\sc Brahic, O., and Fernandes, R.~L.}
\newblock Integrability and reduction of {H}amiltonian actions on {D}irac
  manifolds.
\newblock {\em Indag. Math. (N.S.) 25}, 5 (2014), 901--925.

\bibitem{bra-fin-nak:19}
{\sc Braverman, A., Finkelberg, M., and Nakajima, H.}
\newblock Ring objects in the equivariant derived {S}atake category arising
  from {C}oulomb branches.
\newblock {\em Adv. Theor. Math. Phys. 23}, 2 (2019), 253--344.
\newblock Appendix by Gus Lonergan.

\bibitem{cal:14}
{\sc Calaque, D.}
\newblock Three lectures on derived symplectic geometry and topological field
  theories.
\newblock {\em Indag. Math. (N.S.) 25}, 5 (2014), 926--947.

\bibitem{cal:15}
{\sc Calaque, D.}
\newblock Lagrangian structures on mapping stacks and semi-classical {TFT}s.
\newblock In {\em Stacks and categories in geometry, topology, and algebra},
  vol.~643 of {\em Contemp. Math.} Amer. Math. Soc., Providence, RI, 2015,
  pp.~1--23.

\bibitem{cattaneo-zambon-2007}
{\sc Cattaneo, A.~S., and Zambon, M.}
\newblock Pre-poisson submanifolds.
\newblock In {\em Travaux math\'{e}matiques. {V}ol. {XVII}}, vol.~17 of {\em
  Trav. Math.} Fac. Sci. Technol. Commun. Univ. Luxemb., Luxembourg, 2007,
  pp.~61--74.

\bibitem{cattaneo-zambon-2009}
{\sc Cattaneo, A.~S., and Zambon, M.}
\newblock Coisotropic embeddings in {P}oisson manifolds.
\newblock {\em Trans. Amer. Math. Soc. 361}, 7 (2009), 3721--3746.

\bibitem{cro+:25}
{\sc Crooks, P., Gao, X., Pound, M., and Thompson, C.}
\newblock Some incarnations of {H}amiltonian reduction in symplectic geometry
  and geometric representation theory. {T}o appear in the {P}roceedings of
  {P}oisson 2024. ar{X}iv:2503.23636 (2025).

\bibitem{crooks-mayrand}
{\sc Crooks, P., and Mayrand, M.}
\newblock Symplectic reduction along a submanifold.
\newblock {\em Compos. Math. 158}, 9 (2022), 1878--1934.

\bibitem{cro-may2:24}
{\sc Crooks, P., and Mayrand, M.}
\newblock The {M}oore--{T}achikawa conjecture via shifted symplectic geometry.
  ar{X}iv:2409.03532 (2024).

\bibitem{cro-may:24}
{\sc Crooks, P., and Mayrand, M.}
\newblock Scheme-theoretic coisotropic reduction.
\newblock {\em Q. J. Math. 77}, 1 (2026), 221--247.

\bibitem{crooks-roeser}
{\sc Crooks, P., and R\"oser, M.}
\newblock The log symplectic geometry of {P}oisson slices.
\newblock {\em J. Symplectic Geom. 20}, 1 (2022), 135--190.

\bibitem{dan-kir-mar:24}
{\sc Dancer, A., Kirwan, F., and Martens, J.}
\newblock Implosion, contraction and {M}oore--{T}achikawa.
\newblock {\em Internat. J. Math. 35}, 9 (2024), Paper No. 2441004.

\bibitem{eva-lu:07}
{\sc Evens, S., and Lu, J.-H.}
\newblock Poisson geometry of the {G}rothendieck resolution of a complex
  semisimple group.
\newblock {\em Mosc. Math. J. 7}, 4 (2007), 613--642, 766.

\bibitem{fre-mar:17}
{\sc Frejlich, P., and M\u{a}rcu\c{t}, I.}
\newblock The normal form theorem around {P}oisson transversals.
\newblock {\em Pacific J. Math. 287}, 2 (2017), 371--391.

\bibitem{gan-ginzburg}
{\sc Gan, W.~L., and Ginzburg, V.}
\newblock Quantization of {S}lodowy slices.
\newblock {\em Int. Math. Res. Not.}, 5 (2002), 243--255.

\bibitem{gan-web}
{\sc Gannon, T., and Webster, B.}
\newblock Functoriality of {C}oulomb branches.
\newblock arXiv:2501.09962.

\bibitem{gin-kaz:23}
{\sc Ginzburg, V., and Kazhdan, D.}
\newblock Algebraic symplectic manifolds arising in {S}icilian theories.
\newblock {\em Private manuscript\/}.

\bibitem{SGA1}
{\sc Grothendieck, A.}
\newblock {\em Rev\^{e}tements \'{e}tales et groupe fondamental. {F}asc. {I}:
  {E}xpos\'{e}s 1 \`a 5}.
\newblock Institut des Hautes \'{E}tudes Scientifiques, Paris, 1963.
\newblock Troisi\`eme \'{e}dition, corrig\'{e}e, S\'{e}minaire de
  G\'{e}om\'{e}trie Alg\'{e}brique, 1960/61.

\bibitem{kostant-american}
{\sc Kostant, B.}
\newblock The principal three-dimensional subgroup and the {B}etti numbers of a
  complex simple {L}ie group.
\newblock {\em Amer. J. Math. 81\/} (1959), 973--1032.

\bibitem{kostant}
{\sc Kostant, B.}
\newblock Lie group representations on polynomial rings.
\newblock {\em Amer. J. Math. 85\/} (1963), 327--404.

\bibitem{kostant-inventiones}
{\sc Kostant, B.}
\newblock On {W}hittaker vectors and representation theory.
\newblock {\em Invent. Math. 48}, 2 (1978), 101--184.

\bibitem{lus-spa:79}
{\sc Lusztig, G., and Spaltenstein, N.}
\newblock Induced unipotent classes.
\newblock {\em J. London Math. Soc. (2) 19}, 1 (1979), 41--52.

\bibitem{mai-may:25}
{\sc Maiza, M.~M., and Mayrand, M.}
\newblock Lax--{K}irchhoff moduli spaces and {H}amiltonian 2{D} {T}{Q}{F}{T}.
  {T}o appear in {G}eometriae {D}edicata. ar{X}iv:2510.23567 (2025).

\bibitem{marsden-weinstein}
{\sc Marsden, J.~E., and Weinstein, A.}
\newblock Reduction of symplectic manifolds with symmetry.
\newblock {\em Rep. Mathematical Phys. 5}, 1 (1974), 121--130.

\bibitem{may:23}
{\sc Mayrand, M.}
\newblock Shifted coisotropic structures for differentiable stacks.
\newblock {\em Adv. Math. 475\/} (2025), Paper No. 110345, 67.

\bibitem{mikami-weinstein}
{\sc Mikami, K., and Weinstein, A.}
\newblock Moments and reduction for symplectic groupoids.
\newblock {\em Publ. Res. Inst. Math. Sci. 24}, 1 (1988), 121--140.

\bibitem{moo-tac:11}
{\sc Moore, G.~W., and Tachikawa, Y.}
\newblock On 2d {TQFT}s whose values are holomorphic symplectic varieties.
\newblock In {\em String-{M}ath 2011}, vol.~85 of {\em Proc. Sympos. Pure
  Math.} Amer. Math. Soc., Providence, RI, 2012, pp.~191--207.

\bibitem{ptvv}
{\sc Pantev, T., To\"{e}n, B., Vaqui\'{e}, M., and Vezzosi, G.}
\newblock Shifted symplectic structures.
\newblock {\em Publ. Math. Inst. Hautes \'{E}tudes Sci. 117\/} (2013),
  271--328.

\bibitem{springerta}
{\sc Springer, T.~A.}
\newblock A note on centralizers in semi-simple groups.
\newblock {\em Indag. Math. 28\/} (1966), 75--77.
\newblock Nederl. Akad. Wetensch. Proc. Ser. A {{\bf{6}}9}.

\bibitem{Steinberg}
{\sc Steinberg, R.}
\newblock Regular elements of semisimple algebraic groups.
\newblock {\em Inst. Hautes \'{E}tudes Sci. Publ. Math.}, 25 (1965), 49--80.

\bibitem{tac:18}
{\sc Tachikawa, Y.}
\newblock On `categories' of quantum field theories.
\newblock In {\em Proceedings of the {I}nternational {C}ongress of
  {M}athematicians---{R}io de {J}aneiro 2018. {V}ol. {III}. {I}nvited
  lectures\/} (2018), World Sci. Publ., Hackensack, NJ, pp.~2709--2731.

\bibitem{weh-woo:10}
{\sc Wehrheim, K., and Woodward, C.~T.}
\newblock Functoriality for {L}agrangian correspondences in {F}loer theory.
\newblock {\em Quantum Topol. 1}, 2 (2010), 129--170.

\bibitem{WehrheimWoodward}
{\sc Wehrheim, K., and Woodward, C.~T.}
\newblock Quilted {F}loer cohomology.
\newblock {\em Geom. Topol. 14}, 2 (2010), 833--902.

\bibitem{weinstein-symplectic-category}
{\sc Weinstein, A.}
\newblock The symplectic ``category''.
\newblock In {\em Differential geometric methods in mathematical physics
  ({C}lausthal, 1980)}, vol.~905 of {\em Lecture Notes in Math.} Springer,
  Berlin-New York, 1982, pp.~45--51.

\bibitem{WeinsteinCategories}
{\sc Weinstein, A.}
\newblock Symplectic categories.
\newblock {\em Port. Math. 67}, 2 (2010), 261--278.

\bibitem{xu}
{\sc Xu, P.}
\newblock Momentum maps and {M}orita equivalence.
\newblock {\em J. Differential Geom. 67}, 2 (2004), 289--333.

\end{thebibliography}
	
\end{document}